\documentclass{article}

\usepackage[cm]{fullpage}
\usepackage{amsmath, amssymb, amsthm, enumerate, tikz, booktabs, caption,
pgffor, mathrsfs, color, comment, xcolor, algorithmicx,  mathtools, enumitem,
float, cite}

\pgfdeclarelayer{background layer}
\pgfsetlayers{background layer,main}

\usepackage[colorlinks, linktoc=all]{hyperref}
\hypersetup{linkcolor=blue, urlcolor=blue, citecolor=red}

\let\oldmarginpar\marginpar%
\renewcommand\marginpar[1]{\-\oldmarginpar[\raggedleft\footnotesize #1]%
{\raggedright\footnotesize #1}}

\newcommand{\N}{\mathbb{N}}
\newcommand{\Z}{\mathbb{Z}}

\newcommand{\stab}{\operatorname{Stab}}

\newcommand{\id}{\operatorname{id}}
\newcommand{\im}{\operatorname{im}}
\newcommand{\rank}{\operatorname{rank}}
\renewcommand{\ker}{\operatorname{ker}}
\newcommand{\coker}{\operatorname{coker}}
\newcommand{\dom}{\operatorname{dom}}

\newcommand{\codom}{\operatorname{codom}}

\renewcommand{\L}{\mathscr{L}}
\renewcommand{\H}{\mathscr{H}}
\newcommand{\D}{\mathscr{D}}
\newcommand{\K}{\mathscr{K}}
\newcommand{\J}{\mathscr{J}}

\newcommand{\R}{\mathscr{R}}

\newcommand{\set}[2]{\{#1:#2\}}
\newcommand{\n}{\{1, \ldots, n\}}
\newcommand{\np}{\{1^{\prime}, \ldots, n^{\prime}\}}
\newcommand{\npp}{\{1^{\prime\prime}, \ldots, n^{\prime\prime}\}}
\newcommand{\genset}[1]{\langle#1\rangle}
\newcommand{\floor}[1]{\lfloor#1\rfloor}
\newcommand{\ceiling}[1]{\lceil#1\rceil}

\newcommand{\bigset}[2]{\big\{ {#1}:{#2} \big\}}

\newcommand{\proofrefs}[2]{\noindent \textit{Proof of Theorems\
}\ref{#1}\textit{\ and\ }\ref{#2}.}

\newcommand{\edge}[2]{\draw (#1) -- (#2);}

\newtheorem{thm}{Theorem}[section]
\newtheorem{lem}[thm]{Lemma}
\newtheorem{cor}[thm]{Corollary}
\newtheorem{prop}[thm]{Proposition}

\theoremstyle{definition}

\usepackage{tocloft}

\title{Maximal subsemigroups of finite transformation and diagram monoids}
\author{James East, Jitender Kumar, James D. Mitchell, Wilf A. Wilson}

\begin{document}

\maketitle

\begin{abstract}
  We describe and count the maximal subsemigroups of many well-known monoids of
  transformations and monoids of partitions. More precisely, we find the
  maximal subsemigroups of the full spectrum of monoids of order- or
  orientation-preserving transformations and partial permutations considered by
  V. H. Fernandes and co-authors (12 monoids in total); the partition, Brauer,
  Jones, and Motzkin monoids; and certain further monoids.

  Although descriptions of the maximal subsemigroups of some of the
  aforementioned classes of monoids appear in the literature, we present a
  unified framework for determining these maximal subsemigroups. This approach
  is based on a specialised version of an algorithm for determining the maximal
  subsemigroups of any finite semigroup, developed by the third and fourth
  authors. This allows us to concisely present the descriptions
  of the maximal subsemigroups, and to more clearly see their common features. 
\end{abstract}

{\hypersetup{hidelinks}\tableofcontents}

%%%%%%%%%%%%%%%%%%%%%%%%%%%%%%%%%%%%%%%%%%%%%%%%%%%%%%%%%%%%%%%%%%%%%%%%%%%%%%%%
%%%%%%%%%%%%%%%%%%%%%%%%%%%%%%%%%%%%%%%%%%%%%%%%%%%%%%%%%%%%%%%%%%%%%%%%%%%%%%%%

\section{Introduction, definitions, and summary of results}

A proper subsemigroup of a semigroup $S$
is \textit{maximal} if it is not contained in any other proper subsemigroup of
$S$.  Similarly, a proper subgroup of a group $G$ is \textit{maximal} if it is
not contained in any other proper subgroup of $G$.  If $G$ is a finite group,
then every non-empty subsemigroup of $G$ is a subgroup, and so these notions
are not really distinct in this case. The same is not true if $G$ is an
infinite group. For instance, the natural numbers form a subsemigroup, but not
a subgroup, of the integers under addition.  

Maximal subgroups of finite groups have been extensively studied, in part
because of their relationship to primitive permutation representations, and,
for example, the Frattini subgroup. The maximal subgroups of the symmetric
group are described, in some sense, by the O'Nan-Scott
Theorem~\cite{Scott1980aa} and the Classification of Finite Simple
Groups.  Maximal subgroups of infinite
groups have also been extensively investigated; see~\cite{Ball1966aa,
Ball1968aa, Baumgartner1993aa, Biryukov2000aa, Brazil1994aa, 
Covington1996aa,Macpherson1993aa, Macpherson1990aa,
Mishkin1995aa,Richman1967aa} and the references therein.  

There are also many papers in the literature relating to maximal subsemigroups
of semigroups that are not groups. We describe the finite case in more detail
below; for the infinite case see~\cite{East2015ac} and the references therein.
Maximal subgroups of infinite groups, and maximal subsemigroups of infinite
semigroups, are very different from their finite counterparts. For
example, there exist infinite groups with no maximal subgroups at all, infinite
groups with as many maximal subgroups as subsets, and subgroups that are not
contained in any maximal subgroup.  Analogous statements hold for semigroups
also. 

In~\cite{Graham1968aa}, Graham, Graham, and Rhodes showed that every maximal
subsemigroup of a finite semigroup has certain features, and that every maximal
subsemigroup must be one of a small number of types. As is often the case for
semigroups, this classification depends on the description of maximal subgroups
of certain finite groups.  In~\cite{Donoven2016aa}, Donoven, Mitchell, and
Wilson describe an algorithm for calculating the maximal subsemigroups of an
arbitrary finite semigroup, starting from the results in~\cite{Graham1968aa}.
In the current paper, we use the framework provided by this algorithm to
describe and count the maximal subsemigroups of several families of finite
monoids of partial transformations and monoids of partitions. The maximal
subsemigroups of several of these transformation monoids were described or
counted in~\cite{dimitrova2008maximal, dimitrova2012maximal,
dimitrova2009maximal, dimitrova2012classification, Ganyushkin2003,
gyudzhenov2006maximal}.  However, these results have been somewhat disparate,
their proofs rather ad hoc, and the methods devised in each instance did not
lead to a general theory.  To our knowledge, except for those of the dual
symmetric inverse monoid~\cite{maltcev2007}, the maximal subsemigroups of the
monoids of partitions considered here have not been previously determined.  We
also calculate the maximal subsemigroups of several transformation monoids
related to those in the literature, which were not previously known. We
approach this problem in a concise and consistent way, which could be applied
to many further semigroups.

This paper is structured as follows. In Section~\ref{sec-definitions}, we
describe the notation and definitions relating to semigroups in general that
are used in the paper. In Sections~\ref{sec-trans-definitions}
and~\ref{sec-diagram-definitions}, we define the monoids of transformations
and partitions whose maximal subsemigroups we classify.  These are monoids of
order and orientation preserving and reversing partial transformations; the
partition, Brauer, Jones, and Motzkin monoids; and some related monoids.  In
Section~\ref{sec-summary}, we summarise the results in this paper.  In
Section~\ref{sec-general-results}, we present several results about the maximal
subsemigroups of an arbitrary finite monoid. Many of the results in
Section~\ref{sec-general-results} follow from~\cite{Donoven2016aa},  and
provide a foundation that is adapted to the specific monoids under consideration
in the later sections.  In Sections~\ref{sec-transformation}
and~\ref{sec-diagram}, we classify the maximal subsemigroups of the monoids
defined in Sections~\ref{sec-trans-definitions}
and~\ref{sec-diagram-definitions}, respectively.

%%%%%%%%%%%%%%%%%%%%%%%%%%%%%%%%%%%%%%%%%%%%%%%%%%%%%%%%%%%%%%%%%%%%%%%%%%%%%%%%

\subsection{Background and preliminaries for arbitrary semigroups}
\label{sec-definitions}

A \emph{semigroup} is a set with an associative binary operation. 
A \emph{subsemigroup} of a semigroup is a subset that is 
also a semigroup under the same operation. A subsemigroup of $S$ is
\emph{proper} if it does not equal $S$ and it is \emph{maximal}
if it is a proper subsemigroup of $S$ that is not contained in any other proper
subsemigroup of $S$.
 A \emph{monoid} is
a semigroup $S$ with an \emph{identity} element $1$, which has the property
that $1s = s1 = s$ for all $s \in S$, and a \emph{submonoid} of a monoid
$S$ is a subsemigroup that contains $1$. For a subset $X$ of a semigroup $S$,
the \emph{subsemigroup of $S$ generated by $X$}, denoted by $\genset{X}$, is the
least subsemigroup of $S$, with respect to containment, containing $X$.  More
generally, for a collection of subsets $X_{1}, \ldots, X_{m}$ of $S$ and a
collection of elements $x_{1}, \ldots, x_{n}$ in $S$, we use the notation
$\genset{X_{1}, \ldots, X_{m}, x_{1}, \ldots, x_{n}}$, or some reordering of
this, to denote the subsemigroup of $S$ generated by $X_{1} \cup \cdots \cup
X_{m} \cup \{x_{1}, \ldots, x_{n}\}$. A \emph{generating set} for
$S$ is a subset $X$ of $S$ such that $S = \genset{X}$.

Let $S$ be a semigroup. A \emph{left ideal} of $S$ is a subset $I$ of $S$ such
that $SI = \set{sx}{s\in S,\ x \in I} \subseteq I$. A \emph{right ideal} is
defined analogously, and an \emph{ideal} of $S$ is a subset of $S$ that is both
a left ideal and a right ideal.  Let $x,y\in S$ be arbitrary.  The
\emph{principal left ideal generated by $x$} is the set $Sx \cup \{x\}$, which
is a left ideal of $S$, whereas the \emph{principal ideal generated by $x$} is
the set $SxS \cup Sx \cup xS \cup \{x\}$, and is an ideal. We say that $x$ and
$y$ are $\L$-related if the principal left ideals generated by $x$ and $y$ in
$S$ are equal.  Clearly $\L$ defines an equivalence relation on $S$ --- called
\textit{Green's $\L$-relation} on $S$.  We write $x\L y$ to denote that $(x,y)$
belongs to $\L$.  Green's $\R$-relation is defined dually to Green's
$\L$-relation; Green's $\H$-relation is the meet, in the lattice of equivalence
relations on $S$, of $\L$ and $\R$.  Green's $\D$-relation is the composition
$\L \circ \R = \R \circ \L$, and if $x, y\in S$, then $x\J y$ whenever the
(two-sided) principal ideals generated by $x$ and $y$ are equal.  In a finite
semigroup  $\D = \J$.  We will refer to the equivalence classes of Green's
$\K$-relation, where $\K\in \{\H, \L, \R, \D, \J\}$, as $\mathscr{K}$-classes
where $\mathscr{K}$ is any of $\R$, $\L$, $\H$, or $\J$, and the
$\mathscr{K}$-class of $x\in S$ will be denoted by $K_x$. We write $K_{x}^{S}$
if it is necessary to explicitly refer to the semigroup $S$ on which the
relation is defined.  For a $\J$-class $J$ of $S$ and a Green's relation $\K \in
\{ \H, \L, \R \}$, we denote by $J / \K$ the set of $\K$-classes of $S$
contained in $J$. A partial order on the $\J$-classes of $S$ is induced by
containment of the corresponding principal ideals; more precisely, for arbitrary
elements $x, y \in S$, $J_{x} \leq J_{y}$ if and only if the principal ideal
generated by $y$ contains the principal ideal generated by $x$.  A semigroup $S$
is \emph{$\H$-trivial} if Green's $\H$-relation is the equality relation on $S$.

An \emph{idempotent} is a semigroup element $x$ such that $x^{2} = x$, and the
collection of all idempotents in a semigroup $S$ is denoted by $E(S)$.  An
$\H$-class of $S$ that contains an idempotent is a subgroup of
$S$~\cite[Corollary~2.2.6]{Howie1995aa}.  An element $x \in S$ is
\emph{regular} if there exists $y \in S$ such that $xyx = x$, and a semigroup
is called \emph{regular} if each of its elements is regular.  A $\D$-class is
\emph{regular} if it contains a regular element; in this case, each of
its elements is regular, and each of its $\L$-classes and $\R$-classes contains
an idempotent~\cite[Propositions~2.3.1 and~2.3.2]{Howie1995aa}.

A semigroup $S$ is a \emph{regular $\ast$-semigroup}~\cite{Nordahl1978} if it
possesses a unary operation $^{\ast}$ that satisfies ${(x^{*})}^{*} = x$,
${(xy)}^{*} = y^{*}x^{*}$, and $x = xx^{*}x$ for all $x, y \in S$. Clearly a
regular $\ast$-semigroup is regular. Throughout this paper, for a subset $X$ of
a regular $\ast$-semigroup, we use the notation $X^{*}$ to denote $\set{x^{*}}{x
\in X}$. For a regular $\ast$-semigroup $S$ and elements $x, y \in S$, $x \R y$
if and only if $x^{*} \L y^{*}$. An idempotent $x$ of a regular $\ast$-semigroup
is called a \emph{projection} if $x^{*} = x$.  The idempotent $xx^{*}$ is the
unique projection in the $\R$-class of $x$, and the idempotent $x^{*}x$ is the
unique projection in the $\L$-class of $x$.

An \emph{inverse} semigroup is a semigroup $S$ in which for each element $x \in
S$, there is a unique element $x^{-1} \in S$, called the \emph{inverse} of $x$
in $S$, that satisfies $x=xx^{-1}x$ and $x^{-1} = x^{-1}xx^{-1}$. With the
operation $^{\ast}$ on $S$ defined by $x^{*} = x^{-1}$, an inverse semigroup is
a regular $\ast$-semigroup in which every idempotent is a projection.
A \emph{semilattice} is a commutative semigroup in which every element is an
idempotent; any semilattice is an inverse semigroup.

An element $x$ in a monoid $S$ with identity $1$ is a \emph{unit} if there
exists $x' \in S$ such that $xx' = x'x = 1$. The collection of units in a monoid
is the $\H$-class of the identity, and is called the \emph{group of units} of
the monoid. In a finite monoid, the $\H$-class of the identity is also a
$\J$-class; in this case, it is the unique maximal $\J$-class in the partial
order of $\J$-classes of $S$.

We also require the following graph theoretic notions. A \emph{graph} $\Gamma =
(V, E)$ is a pair of sets $V$ and $E$, called the \emph{vertices} and the
\emph{edges} of $\Gamma$, respectively.  An \emph{edge} $e \in E$ is a pair
$\{u, v\}$ of distinct vertices $u, v \in V$. A vertex $u$ is \emph{adjacent}
to a vertex $v$ in $\Gamma$ if $\{u, v\}$ is an edge of $\Gamma$.  The
\emph{degree} of a vertex $v$ in $\Gamma$ is the number of edges in $\Gamma$
that contain $v$.  An \emph{independent subset} of $\Gamma$ is a subset $K$ of
$V$ such that there are no edges in $E$ of the form $\{k, l\}$, where $k, l \in
K$.  A \emph{maximal independent subset} of $\Gamma$ is an independent subset
that is contained in no other independent subset of $\Gamma$.  A
\emph{bipartite graph} is a graph whose vertices can be partitioned into two
maximal independent subsets.
If $\Gamma = (V, E)$ is a graph, then the
\emph{induced subgraph} of $\Gamma$ on a subset $U \subseteq V$ is the graph
$\left(U,\ \bigset{\{u, v\} \in E}{u, v \in U}\right)$.

In this paper, we define $\N = \{1, 2, 3, \ldots\}$.

%%%%%%%%%%%%%%%%%%%%%%%%%%%%%%%%%%%%%%%%%%%%%%%%%%%%%%%%%%%%%%%%%%%%%%%%%%%%%%%%

\subsection{Partial transformation monoids --- definitions}
\label{sec-trans-definitions}

In this section, we introduce the cast of partial transformation monoids whose
maximal subsemigroups we determine. 

Let $n \in \N$. A \emph{partial transformation of degree $n$} is a partial map
from $\n$ to itself. We define $\mathcal{PT}_{n}$, the \emph{partial
transformation monoid of degree $n$}, to be the monoid consisting of all partial
transformations of degree $n$, under composition as binary relations. The
identity element of this monoid is $\id_{n}$, the \emph{identity transformation}
of degree $n$.

Let $\alpha \in \mathcal{PT}_{n}$. We define
$$\dom(\alpha) = \bigset{i \in \n}{i \alpha\ \text{is defined}},\quad
\im(\alpha) = \bigset{i \alpha}{i \in \dom(\alpha)},\quad
\text{and}\
\rank(\alpha) = |\im(\alpha)|,$$
which are called the \emph{domain}, \emph{image}, and \emph{rank} of $\alpha$,
respectively.  We also define the \textit{kernel} of $\alpha\in \mathcal{PT}_n$
to be the equivalence $$\ker(\alpha) = \bigset{(i, j) \in \dom(\alpha) \times
\dom(\alpha)}{i\alpha = j\alpha}$$
on $\dom(\alpha)$.
If $\dom(\alpha) = \n$, then $\alpha$ is called a \emph{transformation}.
If $\ker(\alpha)$ is the equality relation on $\dom(\alpha)$, i.e.\ if
$\alpha$ is injective, then $\alpha$ is called a \emph{partial permutation}.
A \emph{permutation} is a partial permutation $\alpha\in \mathcal{PT}_n$ such
that $\dom(\alpha)= \{1, \ldots, n\}$.  
We define the following:
\begin{itemize}
  \item
    $\mathcal{T}_{n} = \set{\alpha \in \mathcal{PT}_{n}}{\alpha\ \text{is
    a transformation}}$, the \emph{full transformation monoid of degree $n$};
  \item
    $\mathcal{I}_{n} = \set{\alpha \in \mathcal{PT}_{n}}{\alpha\ \text{is a
    partial permutation}}$, the \emph{symmetric inverse monoid of degree $n$};
    and
  \item
    $\mathcal{S}_{n} = \set{\alpha \in \mathcal{PT}_{n}}{\alpha\ \text{is a
    permutation}}$, the \emph{symmetric group of degree $n$}.
\end{itemize}

The full transformation monoids and the symmetric inverse monoids play a role
analogous to that of the symmetric group, in that every semigroup is isomorphic
to a subsemigroup of some full transformation
monoid~\cite[Theorem~1.1.2]{Howie1995aa}, and every inverse semigroup is
isomorphic to an inverse subsemigroup of some symmetric inverse
monoid~\cite[Theorem~5.1.7]{Howie1995aa}. 

Let $\alpha$ be a partial transformation of degree $n$.
Then $\dom(\alpha) = \{i_{1}, \ldots, i_{k}\} \subseteq \n$, for some $i_{1} <
\cdots < i_{k}$. Throughout this paper, where we refer to an ordering of
natural numbers, we mean the usual ordering $1 < 2 < 3 < \ldots$. We say that
$\alpha$ is \emph{order-preserving} if $i_{1}\alpha \leq \cdots \leq
i_{k}\alpha$, and \emph{order-reversing} if $i_{1}\alpha \geq \cdots \geq
i_{k}\alpha$.  We say that $\alpha$ is \emph{orientation-preserving} if there
exists at most one value $l$, where $1 \leq l \leq k - 1$, such that
$i_{l}\alpha > i_{l + 1}\alpha$, and similarly, we say that $\alpha$ is
\emph{orientation-reversing} if there exists at most one value $1 \leq l \leq k
- 1$ such that $i_{l}\alpha < i_{l + 1}\alpha$. Note that an order-preserving
partial transformation is orientation-preserving, and an order-reversing partial
transformation is orientation-reversing. Having defined these notions, we can
introduce the twelve monoids of partial transformations that we consider here.
These monoids have been extensively studied,
see~\cite{dimitrova2012maximal,dimitrova2012classification} and the references
therein, where the notation used in this paper originates.

We define the following submonoids of $\mathcal{PT}_{n}$:
\begin{itemize}
  \item
    $\mathcal{PO}_{n} = \set{\alpha \in \mathcal{PT}_{n}}{\text{$\alpha$ is
    order-preserving}}$,
  \item
    $\mathcal{POD}_{n} = \set{\alpha \in \mathcal{PT}_{n}}{\text{$\alpha$ is
    order-preserving or order-reversing}}$,
  \item
    $\mathcal{POP}_{n} = \set{\alpha \in \mathcal{PT}_{n}}{\text{$\alpha$ is
    orientation-preserving}}$, and
  \item
    $\mathcal{POR}_{n} = \set{\alpha \in \mathcal{PT}_{n}}{\text{$\alpha$ is
    orientation-preserving or orientation-reversing}}$.
\end{itemize}
We also define the following submonoids of $\mathcal{T}_{n}$ as the
intersections:
\begin{itemize}
  \item
    $\mathcal{O}_{n} = \mathcal{PO}_{n} \cap \mathcal{T}_{n}$,\quad
    $\mathcal{OD}_{n} = \mathcal{POD}_{n} \cap \mathcal{T}_{n}$,\quad
    $\mathcal{OP}_{n} = \mathcal{POP}_{n} \cap \mathcal{T}_{n}$,\quad and\quad
    $\mathcal{OR}_{n} = \mathcal{POR}_{n} \cap \mathcal{T}_{n}$;
\end{itemize}
and we define the following inverse submonoids of $\mathcal{I}_{n}$ as the
intersections:
\begin{itemize}
  \item
    $\mathcal{POI}_{n} = \mathcal{PO}_{n} \cap \mathcal{I}_{n}$,\quad
    $\mathcal{PODI}_{n} = \mathcal{POD}_{n} \cap \mathcal{I}_{n}$,\quad
    $\mathcal{POPI}_{n} = \mathcal{POP}_{n} \cap \mathcal{I}_{n}$,\quad and\quad
    $\mathcal{PORI}_{n} = \mathcal{POR}_{n} \cap \mathcal{I}_{n}$.
\end{itemize}

We require the groups of units of these monoids. We define $\gamma_{n}$ to be
the permutation of degree $n$ that reverses the usual order of $\n$, i.e.
$i\gamma_{n} = n - i + 1$ for all $i \in \n$. Thus when $n \geq 2$, the group
$\genset{\gamma_{n}}$ has order $2$.  We define $\mathcal{C}_{n}$ to be the
cyclic group generated by the permutation $(1\ 2\ \ldots\ n)$, and
$\mathcal{D}_{n} = \genset{(1\ 2\ \ldots\ n),\ \gamma_{n}}$.  When $n \geq 3$,
$\mathcal{D}_{n}$ is a dihedral group of order $2n$. Note that $\mathcal{C}_{2}
= \mathcal{D}_{2} = \genset{\gamma_{2}}$.

The groups of units of $\mathcal{PO}_{n}$, $\mathcal{O}_{n}$, and
$\mathcal{POI}_{n}$ are trivial; the groups of units of $\mathcal{POD}_{n}$,
$\mathcal{OD}_{n}$, and $\mathcal{PODI}_{n}$ are $\genset{\gamma_{n}}$; the
groups of units of $\mathcal{POP}_{n}$, $\mathcal{OP}_{n}$, and
$\mathcal{POPI}_{n}$ are $\mathcal{C}_{n}$; and the groups of units of
$\mathcal{POR}_{n}$, $\mathcal{OR}_{n}$, and $\mathcal{PORI}_{n}$ are
$\mathcal{D}_{n}$. Finally, the symmetric group $\mathcal{S}_{n}$ is the group
of units of $\mathcal{PT}_{n}$, $\mathcal{T}_{n}$, and $\mathcal{I}_{n}$.

See Figure~\ref{fig-PTn-lattice} for an illustration of how these submonoids of
$\mathcal{PT}_{n}$ are interrelated by containment.

%%%%%%%%%%%%%%%%%%%%%%%%%%%%%%%%%%%%%%%%%%%%%%%%%%%%%%%%%%%%%%%%%%%%%%%%%%%%%%%%
% Lattice
%%%%%%%%%%%%%%%%%%%%%%%%%%%%%%%%%%%%%%%%%%%%%%%%%%%%%%%%%%%%%%%%%%%%%%%%%%%%%%%%

\newcommand{\ww}{4}
\newcommand{\xx}{3}
\newcommand{\yy}{4}
\newcommand{\zz}{1}
\newcommand{\nodes}{
\draw(0,3*\zz) node [plain] (Sn) {$\mathcal{S}_{n}$};
\draw(0,2*\zz) node [plain] (Dn) {$\mathcal{D}_{n}$};
\draw(-\xx,1*\zz) node [plain] (Cn) {$\mathcal{C}_{n}$};
\draw(\xx,1*\zz) node [plain] (C2) {$\genset{\gamma_{n}}$};
\draw(0,0*\zz) node [plain] (1n) {$\{\id_{n}\}$};
\draw(0-\ww,3*\zz+\yy) node [plain] (In) {$\mathcal{I}_{n}$};
\draw(0-\ww,2*\zz+\yy) node [plain] (PORIn) {{\small $\mathcal{PORI}_{n}$}};
\draw(-\xx-\ww,1*\zz+\yy) node [plain] (POPIn) {{\small $\mathcal{POPI}_{n}$}};
\draw(\xx-\ww,1*\zz+\yy) node [plain] (PODIn) {{\small $\mathcal{PODI}_{n}$}};
\draw(0-\ww,0*\zz+\yy) node [plain] (POIn) {{\small $\mathcal{POI}_{n}$}};
\draw(0+\ww,3*\zz+\yy) node [plain] (Tn) {$\mathcal{T}_{n}$};
\draw(0+\ww,2*\zz+\yy) node [plain] (ORn) {$\mathcal{OR}_{n}$};
\draw(-\xx+\ww,1*\zz+\yy) node [plain] (OPn) {$\mathcal{OP}_{n}$};
\draw(\xx+\ww,1*\zz+\yy) node [plain] (ODn) {$\mathcal{OD}_{n}$};
\draw(0+\ww,0*\zz+\yy) node [plain] (On) {$\mathcal{O}_{n}$};
\draw(0,3*\zz+\yy+\yy) node [plain] (PTn) {$\mathcal{PT}_{n}$};
\draw(0,2*\zz+\yy+\yy) node [plain] (PORn) {$\mathcal{POR}_{n}$};
\draw(-\xx,1*\zz+\yy+\yy) node [plain] (POPn) {$\mathcal{POP}_{n}$};
\draw(\xx,1*\zz+\yy+\yy) node [plain] (PODn) {$\mathcal{POD}_{n}$};
\draw(0,0*\zz+\yy+\yy) node [plain] (POn) {$\mathcal{PO}_{n}$};
}
\newcommand{\shading}[2]{\filldraw [gray!30] plot [smooth cycle] coordinates
{(0-.5+#1,3*\zz+.5+#2) (0-1+#1,2*\zz+.5+#2) (-\xx-.5+#1,1*\zz+.5+#2)
(-\xx-.5+#1,1*\zz-.5+#2)  (0+#1,0*\zz-.5+#2) (\xx+.5+#1,1*\zz-.5+#2)
(\xx+.5+#1,1*\zz+.5+#2) (0+1+#1,2*\zz+.5+#2) (0+.5+#1,3*\zz+.5+#2) }; }

\begin{figure}[ht]
  \begin{center}
    \begin{tikzpicture}
      [align=center,node distance=2.2cm]
      \tikzstyle{plain}=[fill=white,rounded corners=3pt, draw]

      \nodes

      \draw (PTn)--(PORn)--(POPn)--(POn)--(PODn)--(PORn);
      \draw (Tn)--(ORn)--(OPn)--(On)--(ODn)--(ORn);
      \draw (In)--(PORIn)--(POPIn)--(POIn)--(PODIn)--(PORIn);
      \draw (Sn)--(Dn)--(Cn)--(1n)--(C2)--(Dn);

      \draw (PTn)--(In)--(Sn)--(Tn)--(PTn);
      \draw (PORn)--(PORIn)--(Dn)--(ORn)--(PORn);
      \draw (POPn)--(POPIn)--(Cn)--(OPn)--(POPn);
      \draw (PODn)--(PODIn)--(C2)--(ODn)--(PODn);
      \draw (POn)--(POIn)--(1n)--(On)--(POn);

      \nodes

      \begin{pgfonlayer}{background layer}
        \shading00 \shading{\ww}{\yy} \shading{-\ww}{\yy} \shading0{2*\yy}
      \end{pgfonlayer}

    \end{tikzpicture}
    \caption{Part of the subsemigroup lattice of $\mathcal{PT}_{n}$. This
    diagram shows the monoids defined in Section~\ref{sec-trans-definitions},
    and their groups of units. For a monoid in the top/middle/left/right/bottom
    position of a shaded region, its group of units is shown in the
    corresponding place in the bottom shaded region.}
    \label{fig-PTn-lattice}
  \end{center}
\end{figure}
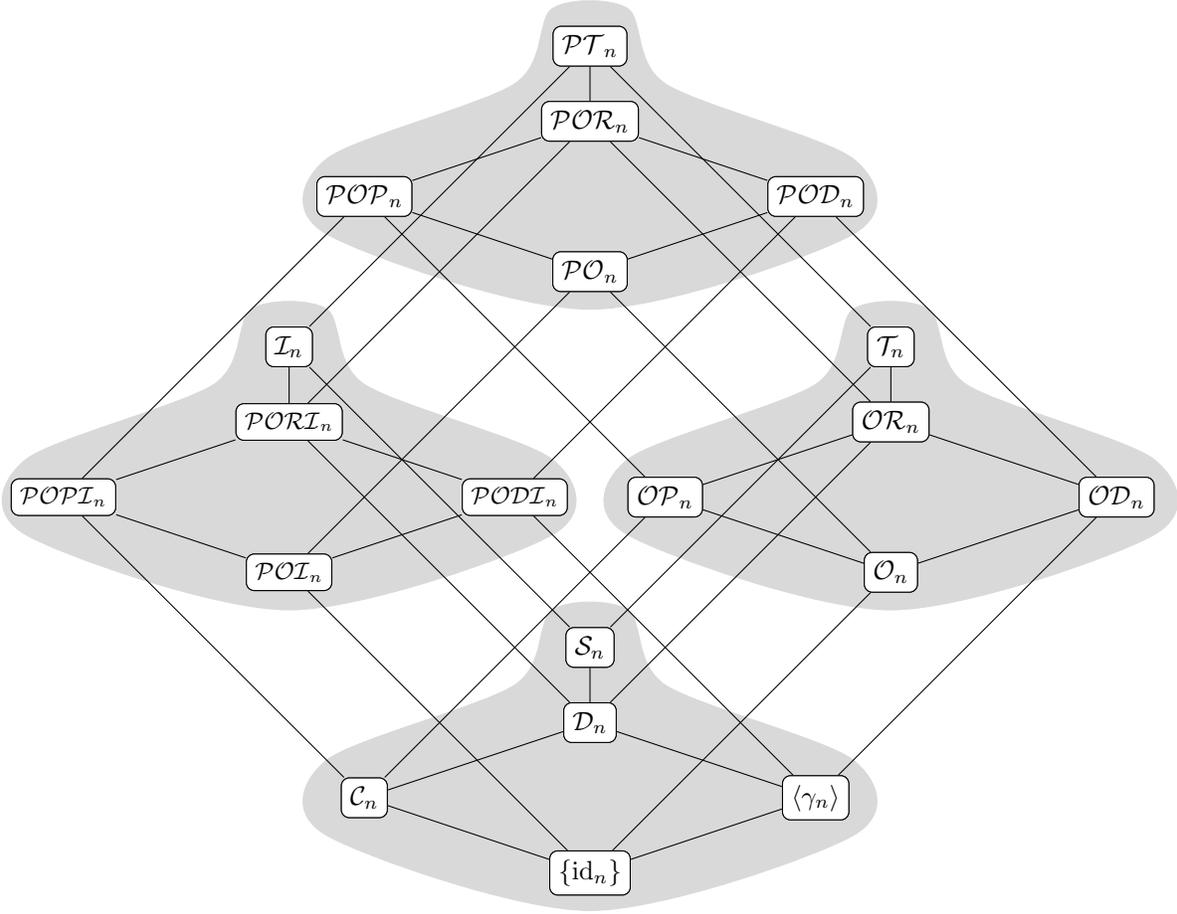

%%%%%%%%%%%%%%%%%%%%%%%%%%%%%%%%%%%%%%%%%%%%%%%%%%%%%%%%%%%%%%%%%%%%%%%%%%%%%%%%

\subsection{Diagram monoids --- definitions}
\label{sec-diagram-definitions}

In this section, we define those monoids of partitions whose maximal
subsemigroups we determine. 

Let $n \in \N$ be arbitrary.  A \emph{partition of degree $n$} is an equivalence
relation on the set $\n \cup \np$.  An equivalence class of a partition is
called a \emph{block}, and a block is \emph{transverse} if it contains points
from both $\n$ and $\np$. A \emph{block bijection} is a partition all of whose
blocks are transverse,  and a block bijection is \emph{uniform} if each of its
blocks contains an equal number of points of $\n$ and $\np$. The \emph{rank} of
a partition is the number of transverse blocks that it contains.  A partition
may be represented visually; see~\cite{Halverson2005869}, for instance.

Let $\alpha$ and $\beta$ be partitions of degree $n$. To calculate the
product $\alpha\beta$, we require three auxiliary partitions, each being a
partition of a different set.  From $\alpha$ we create $\alpha'$ by replacing
every occurrence of each $i'$ by $i''$ in $\alpha$, so that $\alpha'$ is a
partition of $\n \cup \npp$.  Similarly, replacing $i$ by $i''$, we obtain a
partition $\beta'$ of $\npp \cup \np$ from $\beta$.  We define $(\alpha\beta)'$
to be the smallest equivalence on $\n \cup \np \cup \npp$ that contains the
relation $\alpha' \cup \beta'$, which is the transitive closure of $\alpha'
\cup \beta'$.  The product $\alpha\beta$ is the intersection of $(\alpha\beta)'$
and $(\n \cup \np)\times(\n \cup \np)$.
This operation is associative, and so the collection $\mathcal{P}_n$ of all
partitions of degree $n$ forms a semigroup under this operation. The partition
$\id_{n}$, whose blocks are $\{i, i'\}$ for all $i \in \n$, is the identity
element of this semigroup and so $\mathcal{P}_{n}$ is a monoid --- called the
\emph{partition monoid of degree $n$}.  A \emph{diagram monoid} is simply a
submonoid of $\mathcal{P}_{n}$ for some $n \in \N$.

Let $\alpha$ be a partition of degree $n$.  We define $\alpha^{\ast}$ to be the
partition of $\n \cup \np$ created from $\alpha$ by replacing the point $i$ by
$i'$ in the block in which it appears, and by replacing the point $i'$ by
$i$, for all $i \in \n$. For arbitrary partitions $\alpha, \beta \in
\mathcal{P}_{n}$,
$(\alpha^{*})^{*} = \alpha$, $\alpha\alpha^{*}\alpha = \alpha$,
and
$(\alpha\beta)^{*} = \beta^{*}\alpha^{*}$.
In particular, $\mathcal{P}_{n}$ is a regular $\ast$-monoid, as
defined in Section~\ref{sec-definitions}.

There is a canonical embedding of the symmetric group of degree $n$
in $\mathcal{P}_{n}$, where a permutation $\alpha$ is
mapped to the partition with blocks $\{i, (i\alpha)'\}$ for all $i \in \n$. 
Since an element of $\mathcal{P}_{n}$ is a unit if and only if each of its
blocks has the form $\{i, j'\}$ for some $i, j \in \n$, it follows that the
image of this embedding is the group of units of $\mathcal{P}_{n}$.
We reuse the notation $\mathcal{S}_{n}$ to refer to this group.
We define $\rho_{n}$ to be the
partition whose blocks are $\{n, 1'\}$ and $\{i, (i + 1)'\}$ for $i \in
\{1,\ldots,n - 1\}$; in the context of this embedding, $\rho_{n}$ can be thought
of as the $n$-cycle $(1\ 2\ \ldots\ n)$.  Thus the subsemigroup generated by
$\rho_{n}$ is a cyclic group of order $n$.
We define a canonical ordering
$$n' < (n - 1)' < \cdots < 1' < 1 < 2 < \cdots < n$$
on $\n \cup \np$. We say that $\alpha \in \mathcal{P}_{n}$ is \emph{planar} if
there do not exist distinct blocks $A$ and $X$ of $\alpha$, and points $a, b \in
A$ and $x, y \in X$, such that $a < x < b < y$.  More generally, $\alpha$ is
\emph{annular} if $\alpha = \rho_{n}^{k} \beta \rho_{n}^{l}$ for some planar
partition $\beta \in \mathcal{P}_{n}$ and for some $k, l \in \mathbb{Z}$
(note that $\rho_{n}^{n} = \id_{n}$).  For a graphical description of these
properties, see~\cite{auinger2012krohn,Halverson2005869}.

In Section~\ref{sec-diagram}, we determine the maximal subsemigroups of
$\mathcal{P}_n$ and the following submonoids:
\begin{itemize}
  \item
    $\mathcal{PB}_{n} = \set{\alpha \in \mathcal{P}_{n}}{\text{each block of}\
    \alpha\ \text{has size at most 2}}$, the \emph{partial Brauer monoid of
    degree $n$}, introduced in~\cite{Mazorchuk1998aa};
  \item
    $\mathcal{B}_{n} = \set{\alpha \in \mathcal{P}_{n}}{\text{each block of}\
    \alpha\ \text{has size 2}}$, the \emph{Brauer monoid of degree
    $n$}, introduced in~\cite{Mazorchuk1998aa};
  \item
    $\mathcal{I}_{n}^{*} = \set{\alpha \in \mathcal{P}_{n}}{\alpha\ \text{is a
    block bijection}}$, the \emph{dual symmetric inverse monoid of degree $n$},
    introduced in~\cite{fitzgerald1998dual};
  \item
    $\mathfrak{F}_{n} = \set{\alpha \in \mathcal{P}_{n}}{\alpha\ \text{is a
    uniform block bijection}}$, the \emph{uniform block bijection monoid of
    degree $n$}, or the \emph{factorisable dual symmetric inverse
    monoid of degree $n$}, see~\cite{fitzgerald2003presentation} for more
    details;
  \item
    $\mathcal{PP}_{n} = \set{\alpha \in \mathcal{P}_{n}}{\alpha\ \text{is
    planar}}$, the \emph{planar partition monoid of degree $n$}, introduced
    in~\cite{Halverson2005869};
  \item
    $\mathcal{M}_{n} = \set{\alpha \in \mathcal{PB}_{n}}{\alpha\ \text{is
    planar}}$, the \emph{Motzkin monoid of degree $n$},
    see~\cite{Benkart2014aa} for more details;
  \item
    $\mathcal{J}_{n} = \set{\alpha \in \mathcal{B}_{n}}{\alpha\ \text{is
    planar}}$, the \emph{Jones monoid of degree $n$}, also known as the
    \emph{Temperley-Lieb monoid}, introduced in~\cite{Lau2006aa}; and
  \item
    $\mathcal{AJ}_{n} = \set{\alpha \in \mathcal{B}_{n}}{\alpha\ \text{is
    annular}}$, the \emph{annular Jones monoid of degree $n$}, introduced
    in~\cite{auinger2012krohn}.
\end{itemize}

By~\cite{Halverson2005869}, the planar partition monoid of degree $n$ is
isomorphic to the Jones monoid of degree $2n$. Therefore, we will not determine
the maximal subsemigroups of $\mathcal{PP}_{n}$ directly, since their
description can be obtained from the results in Section~\ref{sec-jones}.

The group of units of $\mathcal{PB}_{n}$,
$\mathcal{B}_{n}$, $\mathcal{I}_{n}^{*}$, and $\mathfrak{F}_{n}$ is 
$\mathcal{S}_{n}$, the group of units of $\mathcal{M}_{n}$ and
$\mathcal{J}_{n}$ is the trivial group $\{\id_{n}\}$, and the group of units of
$\mathcal{AJ}_{n}$ is the cyclic group $\mathcal{C}_{n} =
\genset{\rho_{n}}$.

%%%%%%%%%%%%%%%%%%%%%%%%%%%%%%%%%%%%%%%%%%%%%%%%%%%%%%%%%%%%%%%%%%%%%%%%%%%%%%%%

\subsection{Summary of results}\label{sec-summary}

A summary of the results of this paper is shown in Table~\ref{tab-summary}; a
description of the maximal subsemigroups that we count in
Table~\ref{tab-summary} can be found in the referenced results. 

\begin{table}[ht]
  \begin{center}
    \begin{tabular}{llllll}
      \toprule
      Monoid   & Group of units & Number of & OEIS~\cite{OEIS} & Result & \\
       && maximal subsemigroups \\
      \bottomrule
      \toprule

      $\mathcal{POI}_{n}$ &
      Trivial &
      $2^{n} - 1$ 
      & \href{http://oeis.org/A000225}{A000225}
      & Theorem~\ref{thm-POIn}
      & cf.~\cite[Theorem 2]{Ganyushkin2003}
      \\

      $\mathcal{PO}_{n}$ &
      &
      $2^{n} + 2n - 2$ 
      & \href{http://oeis.org/A131520}{A131520}
      & Theorem~\ref{thm-POn} 
      & cf.~\cite[Theorem~1]{dimitrova2012classification}
      \\

      $\mathcal{M}_{n}$ &
      &
      $2^{n} + 2n - 3$ 
      & \href{http://oeis.org/A131898}{A131898}
      & Theorem~\ref{thm-motzkin} 
      & --- \\

      $\mathcal{O}_{n}$ &
      &
      $A_{2n - 1} + 2n - 4$ 
      & \href{http://oeis.org/A000931}{A000931}
      & Theorem~\ref{thm-On} 
      & cf.~\cite[Theorem~2]{dimitrova2008maximal}
      \\

      $\mathcal{J}_{n}$ &
      &
      $2 F_{n - 1} + 2n - 3$ 
      & \href{http://oeis.org/A290140}{A290140}
      & Theorem~\ref{thm-jones} 
      & --- \\

      $\mathcal{PP}_{n}$ &
      &
      $2 F_{2n - 1} + 4n - 3$
      & \href{http://oeis.org/A290140}{A290140}
      & Theorem~\ref{thm-jones} 
      & --- \\

      \midrule

      $\mathcal{PODI}_{2n}$ &
      Order $2$ &
      $3 \cdot 2 ^{n - 1} - 1$ 
      & \href{http://oeis.org/A052955}{A052955}
      & Theorem~\ref{thm-PODIn} 
      & cf.~\cite[Theorem 4]{dimitrova2009maximal}\\

      $\mathcal{PODI}_{2n - 1}$ &
      & 
      $2 ^{n} - 1$ 
      & \href{http://oeis.org/A052955}{A052955}
      & Theorem~\ref{thm-PODIn} 
      & cf.~\cite[Theorem 4]{dimitrova2009maximal}\\
    
      $\mathcal{POD}_{n}$ &
      &
      $2 ^ {\ceiling{n / 2}} + n - 1$ 
      & \href{http://oeis.org/A016116}{A016116}
      & Theorem~\ref{thm-PODn} 
      & --- \\

      $\mathcal{OD}_{n}$ &
      &
      $A_{n} + n - 3$ 
      & \href{http://oeis.org/A000931}{A000931}
      & Theorem~\ref{thm-ODn}
      & cf.~\cite[Theorem~2]{gyudzhenov2006maximal}\\

      \midrule

      $\mathcal{POPI}_{n}$ &
      $\mathcal{C}_n$ (cyclic)&
      $|\mathbb{P}_{n}| + |\mathbb{P}_{n - 1}|$ 
      & \href{http://oeis.org/A059957}{A059957}
      & Theorem~\ref{thm-POPIn} 
      & --- \\

      $\mathcal{POP}_{n}$ &
      &
      $|\mathbb{P}_{n}| + 2$ 
      & \href{http://oeis.org/A083399}{A083399}
      & Theorem~\ref{thm-POPn}
      & --- \\

      $\mathcal{AJ}_{n}$ &
      &
      $|\mathbb{P}_{n}| + 1$ 
      & \href{http://oeis.org/A083399}{A083399}
      & Theorem~\ref{thm-annular-jones} 
      & --- \\

      $\mathcal{OP}_{n}$ &
      &
      $|\mathbb{P}_{n}| + 1$
      & \href{http://oeis.org/A083399}{A083399}
      & Theorem~\ref{thm-OPn}
      & cf.~\cite[Theorem~1.6]{dimitrova2012maximal}\\

      \midrule

      $\mathcal{PORI}_{n}$ &
      $\mathcal{D}_n$ (dihedral)&
      $1 + |\mathbb{P}_{n - 1}| + \sum_{p \in \mathbb{P}_{n}} p$
      & \href{http://oeis.org/A290289}{A290289}
      & Theorem~\ref{thm-PORIn} 
      & --- \\

      $\mathcal{POR}_{n}$ & & 
      $3 + \sum_{p \in \mathbb{P}_{n}} p$
      & \href{http://oeis.org/A008472}{A008472}
      & Theorem~\ref{thm-PORn} 
      & --- \\

      $\mathcal{OR}_{n}$ &
      &
      $2 + \sum_{p \in \mathbb{P}_{n}} p$
      & \href{http://oeis.org/A008472}{A008472}
      & Theorem~\ref{thm-ORn} 
      & cf.~\cite[Theorem~2.6]{dimitrova2012maximal}\\

      \midrule

      $\mathcal{T}_{n}$ &
      $\mathcal{S}_n$ (symmetric) &
      $s_{n} + 1$
      & \href{http://oeis.org/A290138}{A290138}
      & Theorem~\ref{thm-maximals-PTn-Tn-In} 
      & --- \\

      $\mathcal{I}_{n}$ & & 
      $s_{n} + 1$ 
      & \href{http://oeis.org/A290138}{A290138}
      & Theorem~\ref{thm-maximals-PTn-Tn-In} 
      & --- \\

      $\mathcal{I}_{n}^{*}$ &
      &
      $s_{n} + 1$ 
      & \href{http://oeis.org/A290138}{A290138}
      & Theorem~\ref{thm-dual-symmetric} 
      & cf.~\cite[Theorem~19]{maltcev2007}\\

      $\mathfrak{F}_{n}$ &
      &
      $s_{n} + 1$ 
      & \href{http://oeis.org/A290138}{A290138}
      & Theorem~\ref{thm-factorisable} 
      & --- \\

      $\mathcal{B}_{n}$ &
      &
      $s_{n} + 1$ 
      & \href{http://oeis.org/A290138}{A290138}
      & Theorem~\ref{thm-Brauer} 
      & --- \\

      $\mathcal{PT}_{n}$ &
      &
      $s_{n} + 2$ 
      & \href{http://oeis.org/A290138}{A290138}
      & Theorem~\ref{thm-maximals-PTn-Tn-In} 
      & --- \\

      $\mathcal{PB}_{n}$ &
      &
      $s_{n} + 3$ 
      & \href{http://oeis.org/A290138}{A290138}
      & Theorem~\ref{thm-partial-brauer} 
      & --- \\

      $\mathcal{P}_{n}$ &
      &
      $s_{n} + 4$ 
      & \href{http://oeis.org/A290138}{A290138}
      & Theorem~\ref{thm-partition}
      & --- \\

      \bottomrule
    \end{tabular}
    \caption{The maximal subsemigroups of the monoids from this paper, where $n$
    is sufficiently large (usually $n \geq 2$ or $n \geq 3$). For $k \in \N$,
    $s_{k}$ is the number of maximal subgroups of the symmetric group of degree
    $k$~\cite[\href{http://oeis.org/A290138}{A290138}]{OEIS}; $\mathbb{P}_{k}$
    is
    the set of primes that divide $k$; $A_{k}$ is the $k^{\text{th}}$ term of
    the sequence defined by $A_{1} = 1$, $A_{2} = A_{3} = 2$, and $A_{k} = A_{k
    - 2} + A_{k - 3}$ for $k \geq 4$, i.e.\ the $(k + 6)^{\text{th}}$ term
    of the Padovan
    sequence~\cite[\href{http://oeis.org/A000931}{A000931}]{OEIS}; and $F_{k}$
    is the $k^{\text{th}}$ term of the Fibonacci
    sequence~\cite[\href{http://oeis.org/A000045}{A000045}]{OEIS}, defined by
    $F_{1} = F_{2} = 1$, and $F_{k} = F_{k - 1} + F_{k - 2}$ for $k \geq
    3$.}\label{tab-summary}
  \end{center}
\end{table}

%%%%%%%%%%%%%%%%%%%%%%%%%%%%%%%%%%%%%%%%%%%%%%%%%%%%%%%%%%%%%%%%%%%%%%%%%%%%%%%%
%%%%%%%%%%%%%%%%%%%%%%%%%%%%%%%%%%%%%%%%%%%%%%%%%%%%%%%%%%%%%%%%%%%%%%%%%%%%%%%%
\section{The maximal subsemigroups of an arbitrary finite
monoid}\label{sec-general-results}

In this section, we present some results about the maximal subsemigroups of an
arbitrary finite monoid, which are related to those given
in~\cite{Donoven2016aa, Graham1968aa} for an arbitrary finite semigroup.
Since each of the semigroups to which we apply these results is a monoid, we
state the following results in the context of finite monoids. While some
of the results given in this section hold for an arbitrary finite semigroup,
many of them they do not.  

Let $S$ be a finite monoid. By~\cite[Proposition 1]{Graham1968aa}, for
each maximal subsemigroup $M$ of $S$ there exists a single $\J$-class $J$ that
contains $S \setminus M$, or equivalently $S \setminus J \subseteq M$.
Throughout this paper, we call a maximal subsemigroup whose complement is
contained in a $\J$-class $J$ a \emph{maximal subsemigroup arising from $J$}.

We consider the question of which $\J$-classes of $S$ give rise to maximal
subsemigroups.  Let $J$ be a $\J$-class of $S$. There exist maximal
subsemigroups arising from $J$ if and only if every generating set for
$S$ intersects $J$ non-trivially. To see the truth of this statement, let $M$ be
a maximal subsemigroup of $S$ arising from $J$, so that  $S \setminus J
\subseteq M$. For any subset $A$ of $S$ that is disjoint from $J$, it follows
that $$\genset{A} \leq \genset{S \setminus J} \leq M \neq S,$$ and $A$ does not
generate $S$.  Conversely, if $J$ intersects every generating set for $S$
non-trivially, then certainly $S \setminus J$ does not generate $S$. Thus the
subsemigroup $\genset{S \setminus J}$ of $S$ is proper, and is contained in a
maximal subsemigroup.

In order to calculate the maximal subsemigroups of $S$, we identify those
$\J$-classes of $S$ that intersect every generating set of $S$ non-trivially.
Then, for each of these $\J$-classes $J$, we find the maximal subsemigroups of
$S$ arising from $J$.

Let $S$ be a finite monoid, let $J$ be a regular $\J$-class of $S$, and
let $M$ be a maximal subsemigroup of $S$ arising from $J$.  By~\cite[Section
3]{Donoven2016aa}, the intersection $M \cap J$ has one of the following
mutually-exclusive forms:
\begin{enumerate}[label=(M\arabic*), ref=(M\arabic*)]
  \item\label{item-remove-j}
    $M \cap J = \varnothing$.

  \item\label{item-rectangle}
    $M \cap J$ is a non-empty union of both $\L$- and $\R$-classes of $J$;

  \item\label{item-remove-l}
    $M \cap J$ is a non-empty union of $\L$-classes of $J$;

  \item\label{item-remove-r}
    $M \cap J$ is a non-empty union of $\R$-classes of $J$;

  \item\label{item-intersect}
    $M \cap J$ has non-empty intersection with every $\H$-class of $J$;
\end{enumerate}
In general, the collection of maximal subsemigroups arising from a particular
regular $\J$-class $J$ can have any combination of
types~\ref{item-rectangle},~\ref{item-remove-l},~\ref{item-remove-r},
and~\ref{item-intersect}.  However, if $S \setminus J$ is a maximal subsemigroup
of $S$, then it is the only maximal subsemigroup to arise from $J$.  In
other words, there is at most one maximal subsemigroup of
type~\ref{item-remove-j} arising from $J$, and its existence precludes the
occurrence of maximal subsemigroups of
types~\ref{item-rectangle}--\ref{item-intersect}.

It can be most difficult to calculate the maximal subsemigroups of $S$ that
arise from $J$ and have type~\ref{item-intersect} --- we consider a special case
in Section~\ref{sec-intersect-inverse}, which covers the majority of instances
in this paper.  However, in many cases it can be easily shown that no maximal
subsemigroups of type~\ref{item-intersect} exist, such as when $S$ is
$\H$-trivial, or when $S$ is idempotent generated. More generally, since a
maximal subsemigroup of $S$ of type~\ref{item-intersect} contains $E(S)$, the
following lemma holds.

\begin{lem}\label{lem-no-type-intersect}
  Let $S$ be a finite monoid with group of units $G$, and let $J$ be a
  $\J$-class of $S$ that is not equal to $G$.  If
  \begin{enumerate}[label=\emph{(\alph*)}]
    \item
      each $\H$-class of $J$ is trivial, or
    \item
      $J \subseteq \genset{G,\ E(S)}$,
  \end{enumerate}
  then there are no maximal subsemigroups of type~\emph{\ref{item-intersect}}
  arising from $J$.
\end{lem}

When $S$ is a finite regular $\ast$-semigroup and $J$ is a $\J$-class of $S$, it
is routine to verify that the $^{\ast}$ operation permutes the maximal
subsemigroups of $S$ that arise from $J$. In other words, for a subset $M$ of
$S$, $M$ is a maximal subsemigroup of $S$ arising from $J$ if and only if
$M^{*}$ is a maximal subsemigroup of $S$ arising from $J$.  If $M$ is a maximal
subsemigroup of type~\ref{item-remove-j},~\ref{item-rectangle},
or~\ref{item-intersect}, then $M^{*}$ is a maximal subsemigroup of the same
type. However, since the $^{\ast}$ operation of a regular $\ast$-semigroup
transposes $\L$-classes and $\R$-classes, $M$ is a maximal subsemigroup of
type~\ref{item-remove-l} if and only if $M^{*}$ is a maximal subsemigroup of
type~\ref{item-remove-r}.  In particular, the maximal subsemigroups of
type~\ref{item-remove-l} can be deduced from those of type~\ref{item-remove-r},
and vice versa.
These ideas are summarised in the following lemma.

\begin{lem}\label{lem-l-or-r}
  Let $S$ be a finite regular $\ast$-semigroup, let $M$ be a subset of $S$, and
  let $J$ be a $\J$-class of $S$.  Then $M$ is a maximal subsemigroup of $S$
  arising from $J$ if and only if $M^{*}$ is a maximal subsemigroup of $S$
  arising from $J$.  In particular, $M$ is a maximal subsemigroup of $S$ arising
  from $J$ of type~\emph{\ref{item-remove-l}} if and only $M^{*}$ is a maximal
  subsemigroup of $S$ arising from $J$ of type~\emph{\ref{item-remove-r}}.
\end{lem}

%%%%%%%%%%%%%%%%%%%%%%%%%%%%%%%%%%%%%%%%%%%%%%%%%%%%%%%%%%%%%%%%%%%%%%%%%%%%%%%%
\subsection{Maximal subsemigroups arising from the group of
units}\label{sec-group-of-units}

Let $S$ be a finite monoid with group of units $G$.  Since the subset of
non-invertible elements of a finite monoid is an ideal, the $\J$-class $G$
intersects every generating set for $S$ non-trivially, and so there exist
maximal subsemigroups of $S$ arising from $G$.  Note that $G$ is the unique such
$\J$-class only when $S$ is a group.  Another consequence of the fact that the
non-invertible elements of $S$ form an ideal is that the maximal subsemigroups
arising from the group of units can be calculated from $G$ in isolation. As
shown in the following lemma, they are straightforward to describe in terms of
the maximal subsemigroups of $G$.

\begin{lem}\label{lem-group-of-units}
  Let $S$ be a finite monoid with group of units $G$.  Then the maximal
  subsemigroups of $S$ arising from $G$ are the sets $(S \setminus G) \cup U$,
  for each maximal subsemigroup $U$ of $G$.
\end{lem}

A \emph{subsemigroup} of a finite group is a \emph{subgroup}, unless it is
empty; the only group to possess the empty semigroup as a maximal subsemigroup
is the trivial group.  Thus, when the group of units of a particular monoid is
known to be non-trivial, we may use the term ``maximal subgroup'' in place of
``maximal subsemigroup'', as appropriate. This observation permits the following
corollary.

\begin{cor}\label{cor-group-of-units}
  Let $S$ be a finite monoid with group of units $G$.  If $G$ is trivial, then
  the unique maximal subsemigroup of $S$ arising from $G$ is $S \setminus G$,
  which has type~\emph{\ref{item-remove-j}}.  If $G$ is non-trivial, then the
  maximal subsemigroups of $S$ arising from $G$ are the sets $(S \setminus G)
  \cup U$, for each maximal subgroup $U$ of $G$, which have
  type~\emph{\ref{item-intersect}}.
\end{cor}

Only three families of non-trivial groups appear as the group of units of a
monoid in this paper: the cyclic groups, the dihedral groups, and the symmetric
groups.  The conjugacy classes of maximal subgroups of the finite symmetric
groups are described in~\cite{LIEBECK1987365} and counted
in~\cite{LIEBECK1996341};
see~\cite[\href{http://oeis.org/A066115}{A066115}]{OEIS}. However, no simple
formula is known for the total number of maximal subgroups. Thus we use the
notation $s_{k}$ to denote the number of maximal subsemigroups of the symmetric
group of degree $k$~\cite[\href{http://oeis.org/A290138}{A290138}]{OEIS}.  For
the maximal subgroups of the cyclic and dihedral groups, we present the
following well-known results.

\begin{lem}\label{lem-maximals-cyclic}
  Let $n \in \N$, $n \geq 2$, and let $G = \genset{\alpha\ |\ \alpha^{n}}$ be a
  cyclic group of order $n$. The maximal subgroups of $G$ are the subgroups
  $\genset{\alpha^{p}}$, for each prime divisor $p$ of $n$.  In particular, the
  total number of maximal subgroups is the number of prime divisors of $n$.
\end{lem}

\begin{lem}\label{lem-maximals-dihedral}
  Let $n \in \N$, $n \geq 3$, and let $G = \genset{\sigma,\ \rho\ |\
  \sigma^{2},\ \rho^{n},\ {(\sigma\rho)}^{2}}$ be a dihedral group of order
  $2n$.  The maximal subgroups of $G$ are $\genset{\rho}$ and the subgroups
  $\genset{\rho^{p},\ \sigma \rho^{i}}$, for each prime divisor $p$ of
  $n$ and for each integer $i$ with $0 \leq i \leq p - 1$.  In particular, the
  total number of maximal subgroups is one more than the sum of the prime
  divisors of $n$.
\end{lem}

%%%%%%%%%%%%%%%%%%%%%%%%%%%%%%%%%%%%%%%%%%%%%%%%%%%%%%%%%%%%%%%%%%%%%%%%%%%%%%%%
\subsection{Maximal subsemigroups arising from a regular $\J$-class covered
by the group of units}\label{sec-covered}

Let $S$ be a finite monoid with group of units $G$, and let $J$ be a
$\J$-class of $S$ that is not equal to $G$. The maximal
subsemigroups that arise from $J$ are, in general, more complicated to describe
than those maximal subsemigroups that arise from $G$. This is because the
elements of $S$ contained in $\J$-classes that are above $J$ (in
the $\J$-class partial order) may act on, or generate, elements within $J$.
Therefore it is not possible to calculate the maximal subsemigroups that arise
from $J$ without considering these other $\J$-classes. 

Certainly $G$ is a $\J$-class of $S$ that is strictly above $J$, since it
is the unique maximal $\J$-class of $S$. When the group of units is the
\emph{only} $\J$-class strictly above $J$, the problem of finding the maximal
subsemigroups that arise from $J$ is simpler than the general case. We say that
such a $\J$-class is \emph{covered} by the group of units.

Suppose that $J$ is a $\J$-class of $S$ that is covered by $G$. Since the
elements contained in $\J$-classes above $J$ are units, their action on $J$ is
easier than understand than the action of arbitrary semigroup elements.
Moreover, since the only $\J$-class above $J$ is closed under
multiplication, it follows that $S \setminus J$ is a subsemigroup of $S$.  In
particular, any generating set for $S$ intersects $J$ non-trivially, and there
exist maximal subsemigroups that arise from $J$.

We summarise this discussion in the following proposition.

\begin{prop}[Maximal subsemigroups of
  type~\ref{item-remove-j}]\label{prop-remove-j}
  Let $S$ be a finite monoid and let $J$ be a $\J$-class of $S$ that is covered
  by the group of units of $S$. Then $S \setminus J$ is a maximal subsemigroup
  of $S$ if and only if no maximal subsemigroups of
  types~\emph{\ref{item-rectangle}--\ref{item-intersect}} arise from $J$.
\end{prop}

For the majority of the monoids considered in this paper, the \emph{only}
$\J$-classes that give rise to maximal subsemigroups are the group of units, and
regular $\J$-classes that are covered by the group of units.

By~\cite[Proposition 2]{Graham1968aa}, a maximal subsemigroup either:
\begin{itemize}
  \item
    is a union of $\H$-classes of the semigroup, or
  \item
    intersects every $\H$-class of the semigroup non-trivially.
\end{itemize}
A maximal subsemigroup of
type~\ref{item-remove-j},~\ref{item-rectangle},~\ref{item-remove-l},
or~\ref{item-remove-r} is a union of $\H$-classes, whereas maximal subsemigroups
of the second kind are those of type~\ref{item-intersect}.  We consider these
cases separately.  We present versions of some results from~\cite{Donoven2016aa}
that are simplified to suit the current context, and that are used for
calculating maximal subsemigroups of
types~\ref{item-rectangle},~\ref{item-remove-l}, and~\ref{item-remove-r}, in
Section~\ref{sec-type-234}.  Only a few of the monoids considered in this paper
exhibit maximal subsemigroups of type~\ref{item-intersect}, and we present
results tailored to some of these monoids in
Section~\ref{sec-intersect-inverse}.

%%%%%%%%%%%%%%%%%%%%%%%%%%%%%%%%%%%%%%%%%%%%%%%%%%%%%%%%%%%%%%%%%%%%%%%%%%%%%%%%
\subsubsection{Maximal subsemigroups that are unions of $\H$-classes:
types~\ref{item-rectangle},~\ref{item-remove-l},
and~\ref{item-remove-r}}\label{sec-type-234}

Let $S$ be a finite monoid and let $J$ be a regular $\J$-class of $S$ that is
covered by the group of units $G$ of $S$.  To find the maximal subsemigroups of
$S$ arising from $J$ that have types~\ref{item-rectangle}--\ref{item-remove-r},
we construct from $J$ a bipartite graph $\Delta(S, J)$, and analyse its
properties according to the forthcoming results. This bipartite graph was
introduced by Donoven, Mitchell, and Wilson in~\cite[Section 3]{Donoven2016aa}.
When the context unambiguously identifies the $\J$-class that is under
consideration, i.e.\ when a monoid possesses only one $\J$-class covered by the
group of units, we will often use the shorter notation $\Delta(S)$ in place of
$\Delta(S, J)$.

By Green's Lemma~\cite[Lemmas~2.2.1 and~2.2.2]{Howie1995aa}, the group of units
$G$ of $S$ acts on the $\L$-classes of $J$ by right multiplication, and on the
$\R$-classes of $J$ by left multiplication.  The vertices of $\Delta(S, J)$ are
the orbits of $\L$-classes of $J$ and the orbits of $\R$-classes of $J$, under
these actions. In the special case that $J$ consists of a single $\H$-class, we
differentiate between the orbit of $\L$-classes $\{J\}$ and the orbit of
$\R$-classes $\{J\}$, so that $\Delta(S, J)$ contains two vertices. There is an
edge in $\Delta(S, J)$ between an orbit of $\L$-classes $A$ and an orbit of
$\R$-classes $B$ if and only if there exists an $\L$-class $L \in A$ and an
$\R$-class $R \in B$ such that the $\H$-class $L \cap R$ is a group.
We define the two \emph{bicomponents} of $\Delta(S, J)$ as follows: one
bicomponent is the collection of all orbits of $\L$-classes of $J$, the
other bicomponent is the collection of all orbits of $\R$-classes of $J$; the
bicomponents of $\Delta(S, J)$ partition its vertices into two maximal
independent subsets.  Note that $\Delta(S, J)$ is isomorphic to a quotient of
the Graham-Houghton graph of the principal factor of $J$, as defined
in~\cite{East201763, Graham1968ab, Houghton1977aa} --- in the case that the
orbits of $\L$- and $\R$-classes are trivial, these graphs are isomorphic.

The following results characterize the maximal subsemigroups of $S$ of
types~\ref{item-rectangle}--\ref{item-remove-r} that arise from $J$ in terms of
the graph $\Delta(S, J)$. These propositions follow from the results
of~\cite[Section 3]{Donoven2016aa}, having been simplified according to the
assumption that $J$
is covered by the group of units of $S$. More specifically, the results
of~\cite[Section 3]{Donoven2016aa} are formulated in terms of two graphs
$\Delta$ and $\Theta$, and two coloured digraphs $\Gamma_{\L}$ and
$\Gamma_{\R}$, that are constructed from the relevant $\J$-class. When the
semigroup in question is a monoid and the $\J$-class is covered by the group of
units, the graph $\Theta$ and the digraphs $\Gamma_{\L}$ and $\Gamma_{\R}$ have
no edges, and each vertex of $\Gamma_{\L}$ and $\Gamma_{\R}$ has colour $0$.
Thus the conditions on $\Theta$, $\Gamma_{\L}$, and $\Gamma_{\R}$ are
immediately satisfied. The graph $\Delta$ in~\cite[Section 3]{Donoven2016aa} is
equivalent to $\Delta(S, J)$.

\begin{prop}[Maximal subsemigroups of type~\ref{item-rectangle};
  \mbox{\text{cf.}~\cite[Corollary~3.13]{Donoven2016aa}}]\label{prop-rectangle}
  Let $T$ be a subset of $S$ such
  that $S \setminus T \subseteq J$.  Then $T$ is a maximal subsemigroup of $S$
  of type~\emph{\ref{item-rectangle}} if and only if there exist
  proper non-empty subsets $A \subsetneq J / \L$ and $B \subsetneq J / \R$ such
  that $T \cap J$ is the union of the $\L$-classes in $A$ and the $\R$-classes
  in $B$, and $A$ and $B$ are unions of vertices that together form a maximal
  independent subset of $\Delta(S, J)$.
\end{prop}

By Proposition~\ref{prop-rectangle}, the maximal subsemigroups of $S$ of
type~\ref{item-rectangle} arising from $J$ are in bijective correspondence
with the maximal independent subsets of $\Delta(S, J)$ --- excluding the
bicomponents of $\Delta(S, J)$.  Thus we deduce the following corollary.

\begin{cor}\label{cor-rectangle}
  The number of maximal subsemigroups of $S$ of type~\emph{\ref{item-rectangle}}
  arising from $J$ is two less than the number of maximal independent subsets of
  $\Delta(S, J)$.
\end{cor}

The connection between the graph $\Delta(S, J)$ and the maximal
subsemigroups of $S$ of types~\ref{item-remove-l} and~\ref{item-remove-r}
that arise from $J$ is given in the following propositions.

\begin{prop}[Maximal subsemigroups of type~\ref{item-remove-l};
  \mbox{\text{cf.}~\cite[Corollary~3.15]{Donoven2016aa}}]\label{prop-remove-l}
  Let $T$ be a subset of $S$ such
  that $S \setminus T \subseteq J$. Then $T$ is a maximal subsemigroup of $S$ of
  type~\emph{\ref{item-remove-l}} if and only if there exists a proper non-empty
  subset $A \subsetneq J / \L$ such that $T \cap J$ is the union of the
  $\L$-classes in $A$, and $(J / \L) \setminus A$ is a vertex in $\Delta(S, J)$
  that is not adjacent to a vertex of degree $1$.
\end{prop}

\begin{prop}[Maximal subsemigroups of
  type~\ref{item-remove-r}]\label{prop-remove-r}
  Let $T$ be a subset of $S$ such
  that $S \setminus T \subseteq J$. Then $T$ is a maximal subsemigroup of $S$ of
  type~\emph{\ref{item-remove-r}} if and only if there exists a proper non-empty
  subset $B \subsetneq J / \R$ such that $T \cap J$ is the union of the
  $\R$-classes in $B$, and $(J / \R) \setminus B$ is a vertex in $\Delta(S, J)$
  that is not adjacent to a vertex of degree $1$.
\end{prop}

By Proposition~\ref{prop-remove-l}, the number of maximal subsemigroups of $S$
of type~\ref{item-remove-l} is the number of orbits of $\L$-classes that are
adjacent in $\Delta(S, J)$ only to orbits of $\R$-classes with degree at least
$2$.  In the case that every orbit of $\R$-classes has degree $2$ or more in
$\Delta(S, J)$, then the number of maximal subsemigroups of
type~\ref{item-remove-l} is simply the number of orbits of $\L$-classes.  By
Proposition~\ref{prop-remove-r}, the analogous statements hold for maximal
subsemigroups of type~\ref{item-remove-r}.

On the other hand, the number of maximal subsemigroups is restricted when the
group of units acts transitively.

\begin{lem}\label{lem-transitive}
  If $G$ acts transitively on the $\L$-classes of $J$, then no maximal
  subsemigroups of types~\emph{\ref{item-rectangle}}
  or~\emph{\ref{item-remove-l}} arise from
  $J$. Similarly, if $G$ acts transitively on the $\R$-classes of $J$, then no
  maximal subsemigroups of types~\emph{\ref{item-rectangle}}
  or~\emph{\ref{item-remove-r}} arise from $J$.
\end{lem}

\begin{proof}
  Suppose that the group of units acts transitively on the $\L$-classes of $J$,
  so that the graph $\Delta(S, J)$ has a single vertex of $\L$-classes. Since
  $J$ is regular, there are no isolated vertices in $\Delta(S, J)$.  Therefore
  each vertex of $\R$-classes has degree $1$, and is adjacent to the unique
  vertex of $\L$-classes. Thus $\Delta(S, J)$ has just two maximal independent
  subsets --- its bicomponents.  By Corollary~\ref{cor-rectangle}, it follows
  that there are no maximal subsemigroups of type~\ref{item-rectangle} arising
  from $J$, and by Proposition~\ref{prop-remove-l}, no maximal subsemigroups of
  type~\ref{item-remove-l} arise from $J$ either.  The proof of the second
  statement is dual.
\end{proof}

When $S$ is a regular $\ast$-monoid, the graph $\Delta(S, J)$ is particularly
easy to describe.  Suppose there exist $\L$-classes $L_{x}$ and $L_{y}$ in $J$,
and a unit $g \in G$, such that $L_{x}g = L_{y}$. Then $$g^{*}R_{x^{*}} =
g^{*}L_{x}^{*} = {(L_{x}g)}^{*} = L_{y}^{*} = R_{y^{*}}.$$ In this way, it is
easy to see that the orbits of $\L$-classes of $J$ are in bijective
correspondence with the orbits of $\R$-classes of $J$. Specifically, the set
$\{L_{x_{1}}, \ldots,L_{x_{n}}\}$ is an orbit of $\L$-classes of $J$ if and only
if $\{L_{x_{1}}^{*}, \ldots,L_{x_{n}}^{*}\} = \{R_{x_{1}^{*}},
\ldots,R_{x_{n}^{*}}\}$ is an orbit of $\R$-classes of $J$.  These two orbits
are adjacent vertices in $\Delta(S, J)$, since the $\H$-class $L_{x} \cap
R_{x^{*}}$ contains the projection $x^{*}x$ for each $x$, and is therefore a
group.  Furthermore, since an element $e \in S$ is an idempotent if and only if
$e^{*}$ is an idempotent, it follows that an $\H$-class $H_{e}$ is a group if
and only if $$H_{e}^{*} = L_{e}^{*} \cap R_{e}^{*} = R_{e^{*}} \cap L_{e^{*}} =
H_{e^{*}}$$ is a group. Thus the function that maps an orbit of $\L$-classes
$\{L_{x_{1}}, \ldots,L_{x_{n}}\}$ to the orbit of $\R$-classes $\{R_{x_{1}^{*}},
\ldots, R_{x_{n}^{*}}\}$, and vice versa, is an automorphism of $\Delta(S, J)$
of order $2$.

The situation is further simplified when every idempotent of $J$ is a projection
(such as when $S$ is inverse). In this case, since the only group $\H$-class of
$L_{x}$ is $L_{x} \cap R_{x^{*}}$, it follows that an orbit of $\L$-classes
$\{L_{x_{1}}, \ldots, L_{x_{n}}\}$ is adjacent in $\Delta(S, J)$ only to the
corresponding orbit of $\R$-classes $\{R_{x_{1}^{*}}, \ldots, R_{x_{n}^{*}}\}$.
Thus every vertex has degree one, and a maximal independent subset of $\Delta(S,
J)$ is formed by choosing one vertex from each edge. Due to this observation,
and using Propositions~\ref{prop-remove-l} and~\ref{prop-remove-r}, we
obtain the following corollary.

\begin{cor}\label{cor-delta-projections}
  Let $S$ be a finite regular $\ast$-monoid with group of units $G$, and let $J$
  be a $\J$-class of $S$ that is covered by $G$ and whose only idempotents are
  projections. Suppose that $\{O_{1}, \ldots, O_{n}\}$ are the orbits of the
  right action of $G$ on the $\L$-classes of $J$.  Then the maximal
  subsemigroups of $S$ arising from $J$ are of
  types~\emph{\ref{item-remove-j}},~\emph{\ref{item-rectangle}},
  or~\emph{\ref{item-intersect}}.
  A maximal subsemigroup of type~\emph{\ref{item-rectangle}} is the union of $S
  \setminus
  J$ and the union of the Green's classes $$\bigset{L}{L \in O_{i},\ i \in A}
  \cup \bigset{L^{*}}{L \in O_{i},\ i \not\in A},$$ where $A$ is any proper
  non-empty subset of $\{1, \ldots, n\}$.  In particular, there are $2^{n} - 2$
  maximal subsemigroups of type~\emph{\ref{item-rectangle}}, and no maximal
  subsemigroups of types~\emph{\ref{item-remove-l}}
  or~\emph{\ref{item-remove-r}}.
\end{cor}

%%%%%%%%%%%%%%%%%%%%%%%%%%%%%%%%%%%%%%%%%%%%%%%%%%%%%%%%%%%%%%%%%%%%%%%%%%%%%%%%
\subsubsection{Maximal subsemigroups that intersect every $\H$-class:
type~\ref{item-intersect}}\label{sec-intersect-inverse}

To describe maximal subsemigroups of type~\ref{item-intersect} --- i.e.\ those
that intersect each $\H$-class of $S$ non-trivially --- we must use a different
approach from that in Section~\ref{sec-type-234}.  Few of the monoids in
this paper exhibit maximal subsemigroups of type~\ref{item-intersect} that
arise from a $\J$-class covered by the group of units. However, such maximal
subsemigroups do occur in some instances, and in
Proposition~\ref{prop-regularstar-intersect}, we present a result that will
be useful for these cases.

Let $S$ be a finite regular $\ast$-monoid with group of units $G$.  To prove
Proposition~\ref{prop-regularstar-intersect}, we require the following
definition: for a subset $A \subseteq S$, define the \emph{setwise stabilizer
of $A$ in $G$}, $\stab_{G}(A)$, to be the subgroup $\set{g \in G}{Ag = A}$ of
$G$.  Note that $\stab_{G}(A)$ is defined to be the set of elements of $G$ that
stabilize $A$ on the \emph{right}.  However, if we define $A^{*} = \set{a^{*}}{a
\in A}$, then the set of elements of $G$ that stabilize $A$ on the \emph{left}
is equal to $\stab_{G}(A^{*})$, since $$\bigset{g \in G}{gA = A} = \bigset{g \in
G}{A^{*}g^{*} = A^{*}} = {\stab_{G}(A^{*})}^{*} = \stab_{G}(A^{*})^{- 1} =
\stab_{G}(A^{*}).$$ Thus, for a subset $H$ of $S$ that satisfies $H^{*} = H$,
such as for the $\H$-class of a projection, $$\stab_{G}(H) = \set{g \in G}{Hg =
H = gH}.$$ This observation is required in the proof of
Proposition~\ref{prop-regularstar-intersect}.

In Proposition~\ref{prop-regularstar-intersect},  we require the set $e
\stab_{G}(H_{e}) = \bigset{ es }{ s \in \stab_{G}(H_{e}) }$, where $e$ is a
projection of the regular $\ast$-monoid $S$, and the $\J$-class $J_{e}$ is
covered by $G$.  Any submonoid of $S$ that contains both $e$ and $G$ also
contains $e \stab_{G}(H_{e})$.  In particular, every maximal subsemigroup of
type~\ref{item-intersect} arising from $J_{e}$ contains $G$ and all idempotents
in $J_{e}$, and hence contains $e \stab_{G}(H_{e})$.  A stronger result,
necessary for the proof of Proposition~\ref{prop-regularstar-intersect}, is
given by the following lemma.

\begin{lem}\label{lem-stab-subgroup}
  Let $S$ be a finite monoid with group of units $G$, let $e \in E(S)$, and let
  $T$ be a submonoid of $S$ that contains both $e$ and $G$.  Then the set
  $e\stab_{G}(H_{e}^{S})$ is a subgroup of $H_{e}^{T}$.
\end{lem}

\begin{proof}
  Since $e$ is an idempotent, $H_{e}^{T} = T \cap H_{e}^{S}$. Clearly
  $e\stab_{G}(H_{e}^{S}) \subseteq eG \subseteq T$.  Let $g \in
  \stab_{G}(H_{e}^{S})$. Then $eg \in H_{e}^{S}$ by definition, and
  so $e\stab_{G}(H_{e}^{S}) \subseteq H_{e}^{S}$. Thus $e\stab_{G}(H_{e}^{S})
  \subseteq T \cap H_{e}^{S} = H_{e}^{T}$, and the subset is non-empty since $e
  = e1 \in e\stab_{G}(H_{e}^{S})$, where $1$ is the identity of $S$. Since $S$
  is finite, it remains to show that $e\stab_{G}(H_{e}^{S})$ is closed under
  multiplication. Let $g, g' \in \stab_{G}(H_{e}^{S})$. Since $eg \in
  H_{e}^{S}$ and $e$ is the identity of $H_{e}^{S}$, it follows that $(eg)e =
  eg$. Thus
  \begin{equation*}
    (eg)(eg') = (ege)g' = (eg)g' = e(gg') \in e \stab_{G}(H_{e}^{S}).\qedhere
  \end{equation*}
\end{proof}

The following two technical lemmas are also required for the proof of
Proposition~\ref{prop-regularstar-intersect}.

\begin{lem}[\!\!\mbox{\cite[Theorem~A.2.4]{Rhodes2009aa}}]
  \label{lem-stability}
  Let $S$ be a finite semigroup and let $x, y \in S$. Then $x \J xy$ if and only
  if $x \R xy$, and $x \J yx$ if and only if $x \L yx$.
\end{lem}

\begin{lem}[\mbox{\text{follows from}~\cite[Proposition 2.3.7]{Howie1995aa}}]
  \label{lem-R-group}
  Let $R$ be an $\R$-class of an arbitrary finite semigroup, and let $x, y \in
  R$. Then $xy \in R$ if and only if $H_{x}$ is a group.
\end{lem}

\begin{prop}\label{prop-regularstar-intersect}
  Let $S$ be a finite regular $\ast$-monoid with group of units $G$, let $J$ be
  a $\J$-class of $S$ that is covered by $G$, and let $H_{e}^{S}$ be the
  $\H$-class of a projection $e \in J$.  Suppose that $G$ acts transitively on
  the $\R$-classes or the $\L$-classes of $J$, and that $J$ contains one
  idempotent per $\L$-class and one idempotent per $\R$-class (i.e.\ every
  idempotent of $J$ is a projection). Then the maximal subsemigroups of $S$
  arising from $J$ are either:
  \begin{enumerate}[label=\emph{(\alph*)}]
    \item
      $(S \setminus J) \cup GUG = \genset{S \setminus J,\ U}$, for each
      maximal subgroup $U$ of $H_{e}^{S}$ that contains $e
      \stab_{G}(H_{e}^{S})$ \emph{(type~\ref{item-intersect})}, or
    \item
      $S \setminus J$, if no maximal subsemigroups of
      type~\emph{\ref{item-intersect}} exist \emph{(type~\ref{item-remove-j})}.
  \end{enumerate}
\end{prop}

\begin{proof}
  Since $S$ is a regular $\ast$-monoid, $G$ acts transitively on the
  $\L$-classes of $J$ if and only if $G$ acts transitively on the $\R$-classes
  of $J$.  Hence there are no maximal subsemigroups of
  types~\ref{item-rectangle},~\ref{item-remove-l}, or~\ref{item-remove-r}
  arising from $J$, by Lemma~\ref{lem-transitive}. By
  Proposition~\ref{prop-remove-j}, it remains to describe the maximal
  subsemigroups of type~\ref{item-intersect}.

  Let $U$ be a maximal subgroup of $H_{e}^{S}$ that contains
  $e\stab_{G}(H_{e}^{S})$, and define $M_{U} = (S \setminus J) \cup GUG$. To
  prove that $M_{U}$ is a maximal subsemigroup of $S$, we first show that
  $M_{U}$ is a proper subset of $S$, then that it is a subsemigroup, and finally
  that it is maximal in $S$. Since $G$ acts transitively on the $\L$- and
  $\R$-classes of $J$ and $M_{U}$ contains $S \setminus J$, it follows that the
  set $M_{U}$ intersects every $\H$-class of $S$ non-trivially.
  Given that $M_{U}$ is a subsemigroup, and since $G \subseteq S \setminus J$,
  it is obvious that it is generated by $(S \setminus J) \cup U$.

  To prove that $M_{U}$ is a proper subset of $S$, it suffices to show that $GUG
  \cap H_{e}^{S} \subseteq U$. Let $x \in GUG \cap H_{e}^{S}$. Since $x \in
  GUG$, we may write $x = \alpha u \beta$ for some $\alpha, \beta \in G $ and $u
  \in U$.  Since $u, \alpha u \beta \in H_{e}^{S}$, it is straightforward to
  show that $\alpha u$, $u \beta \in H_{e}^{S}$. Thus $$\alpha H_{e}^{S} =
  \alpha (u H_{e}^{S}) = (\alpha u) H_{e}^{S} = H_{e}^{S}, \quad \text{\ and\ }
  \quad H_{e}^{S} \beta = (H_{e}^{S} u) \beta = H_{e}^{S} (u \beta) =
  H_{e}^{S}.$$ In other words, $\alpha$ and $\beta$ stabilize $H_{e}^{S}$ on the
  left and right, respectively.  Thus $\alpha, \beta \in \stab_{G}(H_{e}^{S})$,
  and $$x = ex = e\alpha u e \beta \in \big( e \stab_{G}(H_{e}^{S}) \big) U
  \big( e \stab_{G}(H_{e}^{S}) \big) \subseteq U^{3} = U.$$

  In order to show that $S$ is a subsemigroup, it suffices to show that $xy \in
  M_{U}$ whenever $x, y \in G \cup GUG$, because $S \setminus (G \cup J)$ is an
  ideal of $S$.  If $x \in G$ and $y \in G$, then certainly $xy \in G$.  If $x
  \in G$ and $y \in GUG$, then $xy \in G^{2}UG = GUG$ and $yx \in GUG^{2} =
  GUG$.  For the final case, assume that $x, y \in GUG$ and that $xy \in J$.  By
  definition, $x = \alpha u \beta$ and $y = \sigma v \tau$ for some $\alpha,
  \beta, \sigma, \tau \in G$ and $u, v \in U$.  It suffices to show that $\beta
  \sigma \in \stab_{G}(H_{e}^{S})$, because then
  $$xy =   \alpha u \beta \sigma v \tau
       =   \alpha (ue) \beta \sigma v \tau
       =   \alpha u (e \beta \sigma) v \tau
       \in G U \left( e \stab_{G}(H_{e}^{S}) \right) U G
       \subseteq G U^{3} G = GUG.$$
  Since $H_{e}^{S}$ is a group containing $u$ and $v$, it follows that $u^{*}u =
  vv^{*} = e$. Thus
  $$e \beta \sigma e = u^{*} u \beta \sigma v v^{*}
                     = u^{*} \alpha^{-1} (\alpha u \beta \sigma v \tau)
                       \tau^{-1} v^{*}
                     = u^{*} \alpha^{-1} (x y) \tau^{-1} v^{*}.$$
  Together with $xy = \alpha u (e \beta \sigma e) v \tau$, it
  follows that $e \beta \sigma e \in J$.  By Lemma~\ref{lem-stability}, $e
  \beta \sigma e \in R_{e}^{S}$. Since the elements $e \beta \sigma$ and $e$,
  and their product $e \beta \sigma e$, are all contained in $R_{e}^{S}$,
  Lemma~\ref{lem-R-group} implies that $H_{e \beta \sigma}^{S}$ is a group.  By
  assumption, $R_{e}^{S}$ contains only one group $\H$-class, which is
  $H_{e}^{S}$. Thus $e \beta \sigma \in H_{e}^{S}$, and so
  $H_{e}^{S} \beta \sigma = (H_{e}^{S}e) \beta \sigma = H_{e}^{S}(e \beta
  \sigma) = H_{e}^{S}$, i.e.\ $\beta \sigma \in \stab_{G}(H_{e}^{S})$, as
  required. 

  Let $M$ be a maximal subsemigroup of $S$ that contains $M_{U}$.
  By~\cite[Proposition 4]{Graham1968aa}, $M \cap H_{e}^{S}$ is a maximal
  subgroup of $H_{e}^{S}$, and the intersection of $M$ with any $\H$-class of
  $J$ contains exactly $|M \cap H_{e}^{S}|$ elements. Since $M \cap H_{e}^{S}$
  contains $U$, the maximality of $U$ in $H_{e}^{S}$ implies that $U = M \cap
  H_{e}^{S}$.  Since the group $G$ acts transitively on the $\L$- and
  $\R$-classes of $J$, the intersection of $GUG$ with any $\H$-class of $J$
  contains at least $|U|$ elements.  Thus $|M| \leq |M_{U}|$, and so $M =
  M_{U}$.

  Conversely, suppose that $M$ is a maximal subsemigroup of $S$ of
  type~\ref{item-intersect} arising from $J$.
  By~\cite[Proposition~4]{Graham1968aa}, the intersection $U = M \cap H_{e}^{S}
  = H_{e}^{M}$ is a maximal subgroup of $H_{e}^{S}$, and it contains
  $e\stab_{G}(H_{e}^{S})$ by Lemma~\ref{lem-stab-subgroup}.  Since $M$ contains
  $G$, $U$, and $S \setminus J$, it contains the maximal subsemigroup $M_{U} =
  (S \setminus J) \cup GUG$.  But $M$ is a proper subsemigroup, which implies
  that $M = M_{U}$.
\end{proof}

%%%%%%%%%%%%%%%%%%%%%%%%%%%%%%%%%%%%%%%%%%%%%%%%%%%%%%%%%%%%%%%%%%%%%%%%%%%%%%%%
\subsection{Maximal subsemigroups arising from other $\J$-classes}

The following lemma can be used to find the maximal subsemigroups that arise
from an arbitrary $\J$-class of a finite semigroup. In the later sections, for
conciseness, we will sometimes use this lemma to find the maximal subsemigroups
that arise from a $\J$-class of a monoid that is covered by the group of units.
Additionally, a small number of the diagram monoids in
Section~\ref{sec-diagram} exhibit maximal subsemigroups arising from a
$\J$-class that is neither equal to nor covered by the group of units. The
following lemma will be particularly useful when we determine the maximal
subsemigroups that arise in this case.  Although the results
of~\cite{Donoven2016aa} are, in their full generality, applicable to such cases,
the few examples in this paper do not warrant their use.

\begin{lem}\label{lem-Xi}
  Let $S$ be a finite semigroup, and let $J$ be a $\J$-class of $S$. Suppose
  that there exist
  distinct subsets $X_{1}, \ldots, X_{k} \subseteq J$ such that for
  all $A \subseteq J$, $S = \genset{S \setminus J,\ A}$ if and only if $A
  \cap X_{i} \neq \varnothing$ for all $i \in \{1, \ldots, k\}$. Then the
  maximal subsemigroups of $S$ arising from $J$ are precisely the sets $S
  \setminus X_{i}$ for each $i \in \{1, \ldots, k\}$.
\end{lem}

\begin{proof}
  Note that, by the definition of the sets $X_{i}$ and the assumption that they
  are distinct, no set $X_{i}$ is contained in a different set $X_{j}$.  Let $i
  \in \{1, \ldots, k\}$. We show that $S \setminus X_{i}$ is a subsemigroup of
  $S$; its maximality is then obvious. Let $x, y \in S \setminus X_{i}$.  Since
  $S \setminus X_{i}$ does not generate $S$, but it contains $S \setminus J$ and
  an element $x_{j} \in X_{j}$ for each $j \in \n \setminus \{i\}$, it follows
  that $xy \not\in X_{i}$.  Conversely, let $M$ be a maximal subsemigroup of $S$
  arising from $J$. If $M \cap X_{i} \neq \varnothing$ for each $i$ then, by
  assumption, $S = \genset{M} = M$, a contradiction. Thus $M \cap X_{i} =
  \varnothing$ for some $i$.  In other words, $M \subseteq S \setminus X_{i}$.
  By the maximality of $M$ in $S$, it follows that $M = S \setminus X_{i}$.
\end{proof}

Let $S$ be a monoid with group of units $G$, and suppose there exists a
non-empty subset $X \subseteq S \setminus G$ with the property that for any $A
\subseteq S$, $S = \genset{G,\ A}$ if and only if $A \cap X \neq \varnothing$.
This is equivalent to the property that for any $x \in S$, $S =
\genset{G,\ x}$ if and only if $x \in X$. At several instances in
Sections~\ref{sec-transformation} and~\ref{sec-diagram}, we wish to determine
the maximal subsemigroups of a monoid that has such a subset $X$. Let $x \in X$.
By definition, the principal ideal generated by $x$ consists of every element of
$S$ that can be written as a product involving $x$.  Since $S = \genset{G,\ x}$
and $G$ is closed under multiplication, this ideal is $S \setminus G$. Since $x$
was arbitrary, every element of $X$ generates the same principal ideal, and so
$X$ is contained in some $\J$-class $J$ of $S$. Therefore the maximal
subsemigroups of $S$ arise from its $\J$-classes $G$ and $J$; the maximal
subsemigroup that arises from $J$, namely $S \setminus X$, can be found by
applying Lemma~\ref{lem-Xi} with $k = 1$ and $X_{1} = X$.  The preceding
argument is summarised in the following corollary.

\begin{cor}\label{cor-Xi}
  Let $S$ be a finite monoid with group of units $G$, and suppose there exists a
  non-empty subset $X$ of $S \setminus G$ with the property that $S =
  \genset{G,\ x}$ if and only if $x \in X$. Then the only maximal subsemigroup
  of $S$ that does not arise from the group of units is $S \setminus X$.
\end{cor}

%%%%%%%%%%%%%%%%%%%%%%%%%%%%%%%%%%%%%%%%%%%%%%%%%%%%%%%%%%%%%%%%%%%%%%%%%%%%%%%%

\section{Partial transformation monoids}\label{sec-transformation}

In this section, we find the maximal subsemigroups of 
the families of monoids of partial transformations defined in
Section~\ref{sec-trans-definitions}.

The maximal subsemigroups of several of the monoids considered in this section
have been described in the literature.  The maximal subsemigroups of
$\mathcal{OP}_{n}$ and $\mathcal{OR}_{n}$ were described
in~\cite{dimitrova2012maximal}, those of $\mathcal{POI}_{n}$ were found
in~\cite{Ganyushkin2003}, and those of $\mathcal{PODI}_{n}$ were found
in~\cite{dimitrova2009maximal}. The maximal subsemigroups of the singular ideal
of~$\mathcal{O}_{n}$ were found in~\cite{dimitrova2008maximal}, and those of
the singular ideal of $\mathcal{PO}_{n}$ in~\cite{dimitrova2012classification}.
Additionally, the maximal subsemigroups of the singular ideal of
$\mathcal{OD}_{n}$ were described in~\cite{dimitrova2008maximal}
and~\cite{gyudzhenov2006maximal}. However, since the group of units of
$\mathcal{OD}_{n}$ is non-trivial, this is a fundamentally different problem
than finding the maximal subsemigroups of $\mathcal{OD}_{n}$.  The maximal
subsemigroups of $\mathcal{PT}_{n}$, $\mathcal{T}_{n}$, and $\mathcal{I}_{n}$
are well-known folklore.  To our knowledge, a description of the maximal
subsemigroups of neither $\mathcal{POD}_{n}$, $\mathcal{OD}_{n}$,
$\mathcal{POP}_{n}$, $\mathcal{POPI}_{n}$, $\mathcal{POR}_{n}$, nor
$\mathcal{PORI}_{n}$ has appeared in the literature.

We reprove the known results in order to demonstrate that they may be obtained
in a largely unified manner, using the tools described in
Section~\ref{sec-general-results}. Furthermore, the descriptions of the maximal
subsemigroups of some of these monoids --- such as those of $\mathcal{PO}_{n}$
and $\mathcal{POD}_{n}$ --- are very closely linked, and so it is instructive to
present such results together, regardless of whether they were previously known.

Let $n \in \N$, $n \geq 2$.  We require some facts and notation that are common
to the submonoids of $\mathcal{PT}_{n}$ defined in
Section~\ref{sec-trans-definitions};  let $S$ be such a monoid.  Then $S$ is
regular, and any generating set for $S$ contains elements of rank $n$ and $n -
1$, but needs not contain elements of smaller rank. The Green's relations on $S$
can be characterised as follows:
\begin{itemize}
  \item
    $\alpha \L \beta$ if and only if $\im(\alpha) = \im(\beta)$,
  \item
    $\alpha \R \beta$ if and only if $\ker(\alpha) = \ker(\beta)$, and
  \item
    $\alpha \J \beta$ if and only if $\rank(\alpha) = \rank(\beta)$,
\end{itemize}
for $\alpha, \beta \in S$.  Note that $\ker(\alpha) = \ker(\beta)$ implies that
$\dom(\alpha) = \dom(\beta)$, by definition.  It is well-known that this is the
characterization of the Green's relations on $\mathcal{PT}_{n}$. Since $S$ is
regular submonoid of $\mathcal{PT}_{n}$, it follows that the characterization of
Green's $\L$- and $\R$-relations on $S$ is as
described~\cite[Proposition~2.4.2]{Howie1995aa}.  It is straightforward to see
that $\J$-equivalence in $S$ is determined by rank.

By the previous paragraph, to describe the maximal subsemigroups of $S$, we must
find those maximal subsemigroups that arise from the group of units, and those
that arise from the $\J$-class containing elements of rank $n - 1$. The results
of Section~\ref{sec-group-of-units} apply in the former case, and the results of
Section~\ref{sec-covered} apply in the latter case.

Notation for the groups of units that appear in this section was defined in
Section~\ref{sec-trans-definitions}.  In order to describe the remaining maximal
subsemigroups, we require the following notation for the Green's classes that
contain partial transformations of rank $n - 1$. Define $$J_{n - 1} =
\bigset{\alpha \in \mathcal{PT}_{n}}{\rank(\alpha) = n - 1}$$ to be the
$\J$-class of $\mathcal{PT}_{n}$ consisting of partial transformations of rank
$n - 1$. A partial transformation of rank $n - 1$ lacks exactly one element from
its image, and is either a partial permutation that lacks one element from its
domain, or is a transformation with a unique non-trivial kernel class,
which contains two points. Thus for distinct $i, j \in \n$, we define the
Green's classes
\begin{itemize}
  \item
    $L_{i} = \bigset{\alpha \in J_{n - 1}}{i \notin \im(\alpha)}$, which is an
    $\L$-class;
  \item
    $R_{i} = \bigset{\alpha \in J_{n - 1}}{i \notin \dom(\alpha)}$, which is an
    $\R$-class consisting of partial permutations; and
  \item
    $R_{\{i, j\}} = \set{\alpha \in J_{n - 1}}{(i, j) \in \ker(\alpha)}$, which
    is an $\R$-class consisting of transformations.
\end{itemize}
An $\H$-class of the form $L_{i} \cap R_{j}$ is a group if and only if $i = j$,
and an $\H$-class of the form $L_{i} \cap R_{\{j, k\}}$ is a group if and only
if $i \in \{j, k\}$.

Let $S$ be one of the submonoids of $\mathcal{PT}_{n}$ defined in
Section~\ref{sec-trans-definitions}.  It follows that the set $J_{n - 1} \cap S$
is a regular $\J$-class of $S$, that the $\L$-classes of $J_{n - 1} \cap S$ are
the sets of the form $L_{i} \cap S$, and that the $\R$-classes of $J_{n - 1}
\cap S$ are those non-empty sets of the form $R_{i} \cap S$ and $R_{\{i, j\}}
\cap S$, for distinct $i, j \in \n$.  Whenever we present a picture of
the graph $\Delta(S, J_{n - 1} \cap S)$, such as the picture of the graph
$\Delta(\mathcal{PO}_{n}, J_{n - 1} \cap
\mathcal{PO}_{n})$ given in Figure~\ref{fig-POn-delta}, we label an $\L$-class
as $L_{i}$ rather than as $L_{i} \cap S$, and so on, in order to avoid
cluttering the image. This approach also has the advantage of emphasizing the
similarities between the graphs of related monoids --- indeed, some graphs may
be obtained as induced subgraphs of others.

Note that the non-trivial kernel class of a order-preserving or -reversing
transformation of rank $n - 1$ has the form $\{i, i + 1\}$ for some $i \in \{1,
\ldots, n - 1\}$, and that the non-trivial kernel class of an 
orientation-preserving or -reversing transformation of rank $n - 1$ has the same
form, or is equal to $\{1, n\}$. Any non-empty subset of $\n$ appears as the
image of some partial transformation in each of the monoids defined in
Section~\ref{sec-trans-definitions}.

Often, the principal obstacle to describing the maximal subsemigroups of $S$ is
to determine the maximal independent subsets of $\Delta(S, J_{n - 1} \cap S)$.
Since the $\J$-class of $S$ to be considered in each case is $J_{n - 1} \cap S$,
throughout the remainder of this section, we write $\Delta(S)$ in place of
$\Delta(S, J_{n - 1} \cap S)$.  To determine the maximal independent subsets of
$\Delta(S)$, we must calculate the actions of the group of units $G$ of $S$ on
the $\R$-classes and $\L$-classes of $J_{n - 1} \cap S$. The orbits of
$\L$-classes correspond to the orbits of $G$ on $\n$, in the following way: if
$\Omega \subseteq \n$ is an orbit of $G$ on $\n$, then $\bigset{L_{i} \cap S}{i
\in \Omega}$ is an orbit of $G$ on $(J_{n - 1} \cap S) / \L$, and vice versa.
In the same way, the orbits of $\R$-classes that contain partial permutations of
rank $n - 1$ correspond to the orbits of $G$ on the indices $\bigset{i}{R_{i}
\cap S \neq \varnothing}$.  Finally, the orbits of $\R$-classes that contain
transformations of rank $n - 1$ correspond to orbits of $G$ on the sets
$\bigset{\{i, j\}}{i \neq j,\ R_{\{i, j\}} \cap S \neq \varnothing}$. The
actions of $\{\id_{n}\}$, $\genset{\gamma_{n}}$, $\mathcal{C}_{n}$,
$\mathcal{D}_{n}$, and $\mathcal{S}_{n}$ on these sets are easy to understand.

%%%%%%%%%%%%%%%%%%%%%%%%%%%%%%%%%%%%%%%%%%%%%%%%%%%%%%%%%%%%%%%%%%%%%%%%%%%%%%%%
\subsection{$\mathcal{PT}_{n}$, $\mathcal{T}_{n}$, and
$\mathcal{I}_{n}$}\label{sec-PTn-Tn-In}

First we find the maximal subsemigroups of the monoids $\mathcal{PT}_{n}$,
$\mathcal{T}_{n}$, and $\mathcal{I}_{n}$; their maximal subsemigroups are
well-known folklore, but we include the following result for completeness, and
as a gentle introduction to the application of the results of
Section~\ref{sec-general-results}.

\begin{thm}\label{thm-maximals-PTn-Tn-In}
  Let $n \in \N$, $n \geq 2$, be arbitrary.
  \begin{enumerate}[label=\emph{(\alph*)}]
    \item
      The maximal subsemigroups of $\mathcal{PT}_{n}$, the partial
      transformation monoid of degree $n$, are the sets:
      \begin{enumerate}[label=\emph{(\roman*)}]
        \item
          $(\mathcal{PT}_{n} \setminus \mathcal{S}_{n}) \cup U$,
          where $U$ is a maximal subgroup of $\mathcal{S}_{n}$
          \emph{(type~\ref{item-intersect})};
        \item
          $\mathcal{PT}_{n} \setminus \bigset{ \alpha \in \mathcal{T}_{n} }{
           \rank(\alpha) = n - 1}$
           \emph{(type~\ref{item-remove-r})}; and
        \item
          $\mathcal{PT}_{n} \setminus \bigset{ \alpha \in \mathcal{I}_{n} }{
          \rank(\alpha) = n - 1}$
          \emph{(type~\ref{item-remove-r})}.
      \end{enumerate}
    \item
      The maximal subsemigroups of $\mathcal{T}_{n}$, the full transformation
      monoid of degree $n$, are the sets:
      \begin{enumerate}[label=\emph{(\roman*)}]
        \item
          $(\mathcal{T}_{n} \setminus \mathcal{S}_{n}) \cup U$,
          where $U$ is a maximal subgroup of $\mathcal{S}_{n}$
          \emph{(type~\ref{item-intersect})}; and
        \item
          $\mathcal{T}_{n} \setminus \set{\alpha \in \mathcal{T}_{n} }{
          \rank(\alpha) = n -
          1}$
          \emph{(type~\ref{item-remove-j})}.
      \end{enumerate}
    \item
      The maximal subsemigroups of $\mathcal{I}_{n}$, the symmetric inverse
      monoid of degree $n$, are the sets:
      \begin{enumerate}[label=\emph{(\roman*)}]
        \item
          $(\mathcal{I}_{n} \setminus \mathcal{S}_{n}) \cup U$,
          where $U$ is a maximal subgroup of $\mathcal{S}_{n}$
          \emph{(type~\ref{item-intersect})}; and
        \item
          $\mathcal{I}_{n} \setminus \set{\alpha \in \mathcal{I}_{n} }{
          \rank(\alpha) = n - 1}$
          \emph{(type~\ref{item-remove-j})}.
      \end{enumerate}
  \end{enumerate}
  In particular, for $n \geq 2$, there are $s_{n} + 2$ maximal subsemigroups of
  $\mathcal{PT}_{n}$, and $s_{n} + 1$ maximal subsemigroups of both
  $\mathcal{T}_{n}$ and $\mathcal{I}_{n}$, where $s_{n}$ is the number of
  maximal subgroups of the symmetric group of degree $n$.  The monoid
  $\mathcal{PT}_{1} = \mathcal{I}_{1}$ is a semilattice of order 2: its
  maximal subsemigroups are each of its singleton subsets.
\end{thm}

\begin{proof}
  Since the group of units of $\mathcal{PT}_{n}$, $\mathcal{T}_{n}$, and
  $\mathcal{I}_{n}$ is $\mathcal{S}_{n}$, it follows by
  Corollary~\ref{cor-group-of-units} that the maximal subsemigroups arising from
  the group of units are, in each case, those described in the statement of the
  theorem.  It is well-known that $\mathcal{T}_{n}$ and $\mathcal{I}_{n}$ can
  each be generated by the symmetric group $\mathcal{S}_{n}$, along with
  any one of their elements of rank $n - 1$.  Thus, by Corollary~\ref{cor-Xi},
  for each of these two monoids, the unique maximal subsemigroup arising from
  its $\J$-class of rank $n - 1$ has type~\ref{item-remove-j}.

  It remains to consider those maximal subsemigroups that arise from the
  $\J$-class $J_{n - 1}$ of $\mathcal{PT}_{n}$.  Since $\mathcal{PT}_{n} =
  \genset{\mathcal{T}_{n},\ \mathcal{I}_{n}}$, the partial transformation
  monoid of degree $n$ can generated by $\mathcal{S}_{n}$, along with \emph{any}
  partial permutation of rank $n - 1$ and \emph{any} transformation of
  rank $n - 1$. Since $\mathcal{T}_{n}$ and $\mathcal{I}_{n}$ are both
  subsemigroups of $\mathcal{PT}_{n}$, any generating set for $\mathcal{PT}_{n}$
  contains \emph{both} a transformation and a partial permutation of rank $n -
  1$. Thus, using Lemma~\ref{lem-Xi} with $k = 2$, $X_{1} = J_{n - 1} \cap
  \mathcal{T}_{n}$, and $X_{2} = J_{n - 1} \cap \mathcal{I}_{n}$, it follows
  that the maximal subsemigroups of $\mathcal{PT}_{n}$ arising from $J_{n - 1}$
  are those given in the theorem.
  \end{proof}

The description of the maximal subsemigroups of $\mathcal{PT}_{n}$ that arise
from its $\J$-class $J_{n - 1}$ can also be obtained by using the graph
$\Delta(\mathcal{PT}_{n})$ and the results of Section~\ref{sec-type-234}.  Since
$\mathcal{PT}_{n}$ is generated by its units and its idempotents of rank $n
- 1$, Lemma~\ref{lem-no-type-intersect} implies that no maximal subsemigroups of
type~\ref{item-intersect} arise from $J_{n - 1}$.  The action of
$\mathcal{S}_{n}$ on the $\L$-classes of $J_{n - 1}$ is transitive, and so by
Lemma~\ref{lem-transitive} there are no maximal subsemigroups of
types~\ref{item-rectangle} or~\ref{item-remove-l}. However, there are two orbits
under the action of $\mathcal{S}_{n}$ on the $\R$-classes of $J_{n - 1}$: one
orbit contains the $\R$-classes of transformations, and the other contains the
$\R$-classes of partial permutations.  These two orbits of $\R$-classes are
adjacent in $\Delta(\mathcal{PT}_{n})$ to the unique orbit of $\L$-classes; a
picture of $\Delta(\mathcal{PT}_{n})$ is shown in Figure~\ref{fig-PTn-delta}.
Thus, by Proposition~\ref{prop-remove-r}, there are two maximal subsemigroups of
type~\ref{item-remove-r} arising from $J_{n - 1}$, each formed by removing the
$\R$-classes from one of these orbits.

%%%%%%%%%%%%%%%%%%%%%%%%%%%%%%%%%%%%%%%%%%%%%%%%%%%%%%%%%%%%%%%%%%%%%%%%%%%%%%%%
% Figure of Delta for partial transformation monoid
%%%%%%%%%%%%%%%%%%%%%%%%%%%%%%%%%%%%%%%%%%%%%%%%%%%%%%%%%%%%%%%%%%%%%%%%%%%%%%%%
\begin{figure}
  \begin{center}
    \begin{tikzpicture}
      % L-classes
      \node[rounded corners,rectangle,draw,fill=blue!20]
        (1)  at (2,  0) {$\bigset{L_{i}}{i \in \n}$};

      % R-classes
      \node[rounded corners,rectangle,draw,fill=blue!20]
        (11) at (0,  3) {$\bigset{R_{i}}{i \in \n}$};
      \node[rounded corners,rectangle,draw,fill=blue!20]
        (12) at (5,  3) {$\bigset{R_{\{i, j\}}}{ i, j \in \n,\ i \neq j}$};

      % edges
      \edge{1}{11};
      \edge{1}{12};
    \end{tikzpicture}
  \end{center}
  \caption{The graph $\Delta(\mathcal{PT}_{n})$.}\label{fig-PTn-delta}
\end{figure}
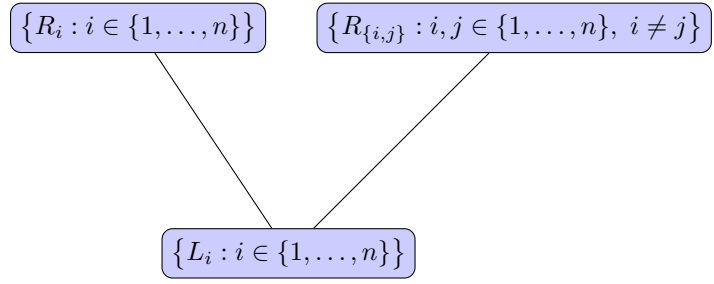

%%%%%%%%%%%%%%%%%%%%%%%%%%%%%%%%%%%%%%%%%%%%%%%%%%%%%%%%%%%%%%%%%%%%%%%%%%%%%%%%
\subsection{$\mathcal{PO}_{n}$ and $\mathcal{POD}_{n}$}\label{sec-POn-PODn}

The maximal subsemigroups of $\mathcal{PO}_{n}$ were described
in~\cite{dimitrova2012classification}.
To our knowledge, the maximal subsemigroups of
$\mathcal{POD}_{n}$ have not been described in the literature.  Using our
approach, we find that the maximal subsemigroups of $\mathcal{POD}_{n}$ are
closely linked to those of $\mathcal{PO}_{n}$.

Let $n \in \N$, $n \geq 2$.  To state the main results in this section we
require the following notation. If $S \in \{\mathcal{PO}_{n},
\mathcal{POD}_{n}\}$, then $J_{n - 1} \cap S$ is a regular $\J$-class of $S$,
the $\L$-classes of $J_{n - 1}\cap S$ are $\bigset{L_{i}\cap S}{i \in \n}$, and
the $\R$-classes are $\bigset{R_{i}\cap S}{i \in \n}$ and $\bigset{R_{\{i, i +
1\}}\cap S}{i \in \{1, \ldots, n - 1 \}}$. Note that $\mathcal{POD}_{n}$ is
generated by $\mathcal{PO}_{n}$ and the permutation $\gamma_{n}$ that reverses
the usual order on $\n$. More information about $\mathcal{PO}_{n}$ can be found
in~\cite{Gomes1992}.

The main results of this section are the following theorems.

\begin{thm}\label{thm-POn}
  Let $n \in \N$, $n \geq 2$, be arbitrary and let $\mathcal{PO}_{n}$ be
  the monoid of order-preserving partial transformations on $\n$
  with the usual order. Then the maximal subsemigroups of $\mathcal{PO}_{n}$
  are:
  \begin{enumerate}[label=\emph{(\alph*)}]
    \item
      $\mathcal{PO}_{n} \setminus \{ \id_{n} \}$
      \emph{(type~\ref{item-remove-j})};
    \item
      the union of $\mathcal{PO}_{n} \setminus J_{n - 1}$ and the union of
      $$\bigset{ L_{i} \cap \mathcal{PO}_{n}}{i \in A} \cup \bigset{R_{i} \cap
      \mathcal{PO}_{n}}{i \notin A} \cup \bigset{R_{\{ i, i + 1\}} \cap
      \mathcal{PO}_{n}}{i, i + 1 \notin A},$$ where $A$ is any non-empty proper
      subset of $\n$
      \emph{(type~\ref{item-rectangle})} and;
    \item
      $\mathcal{PO}_{n} \setminus R$, where $R$ is any $\R$-class in $J_{n - 1}
      \cap \mathcal{PO}_{n}$
      \emph{(type~\ref{item-remove-r})}.
  \end{enumerate}
  The monoid $\mathcal{PO}_{1} = \mathcal{PT}_{1}$ is a semilattice of order
  $2$: its maximal subsemigroups are each of its singleton subsets. In
  particular, for $n \in \N$, there are $2 ^ {n} + 2n - 2$ maximal subsemigroups
  of $\mathcal{PO}_{n}$. 
\end{thm}

\begin{thm}\label{thm-PODn}
  Let $n \in \N$, $n \geq 2$, be arbitrary. Let $\mathcal{POD}_{n}$ be the
  monoid of order-preserving or -reversing partial transformations on $\n$ with
  the usual order, and let $\gamma_{n}$ be the permutation of degree $n$ that
  reverses this order.  Then the maximal subsemigroups of $\mathcal{POD}_{n}$
  are:
  \begin{enumerate}[label=\emph{(\alph*)}]
    \item
      $\mathcal{POD}_{n} \setminus \{\gamma_{n}\}$
      \emph{(type~\ref{item-intersect})};
    \item
      the union of $\mathcal{POD}_{n} \setminus J_{n - 1}$ and the union of
      \begin{multline*}
        \bigset{(L_{i} \cup L_{n - i + 1}) \cap \mathcal{POD}_{n}}{i \in A} \cup
        \bigset{(R_{i} \cup R_{n - i + 1}) \cap \mathcal{POD}_{n}}{i \notin A}\\
        \cup \bigset{(R_{\{ i, i + 1 \}} \cup R_{\{n - i, n - i + 1\}}) \cap
        \mathcal{POD}_{n}}{i, i + 1 \notin A},
      \end{multline*}
      where $A$ is any non-empty proper subset of $\{1, \ldots, \ceiling{n /
      2}\}$ \emph{(type~\ref{item-rectangle})} and;
    \item
      $\mathcal{POD}_{n} \setminus (R \cup \gamma_{n} R)$,
      where $R$ is any $\R$-class in $J_{n - 1} \cap \mathcal{POD}_{n}$
      \emph{(type~\ref{item-remove-r})}; in particular:
      \begin{enumerate}[label=\emph{(\roman*)}]
        \item
          $\mathcal{POD}_{n}\setminus (R_{i}\cup R_{n - i + 1})$,
          for $i \in \{1, \ldots, \ceiling{n / 2}\}$; and
        \item
          $\mathcal{POD}_{n}\setminus (R_{\{i, i + 1\}}\cup R_{\{n - i, n - i +
          1\}})$,
          for $i \in \{1, \ldots, \floor{n / 2}\}$.
        \end{enumerate}
  \end{enumerate}
  The monoid $\mathcal{POD}_{1} = \mathcal{PT}_{1}$ is a semilattice of order
  $2$: its maximal subsemigroups are each of its singleton subsets.
  In particular, for $n \in \N$, there are $2 ^ {\ceiling{n / 2}} + n - 1$
  maximal subsemigroups of $\mathcal{POD}_{n}$. 
\end{thm}

The most substantial part of the proof of Theorems~\ref{thm-POn}
and~\ref{thm-PODn} is the description of the maximal independent subsets of
$\Delta(\mathcal{PO}_{n})$ and $\Delta(\mathcal{POD}_{n})$, respectively.  Since
$\mathcal{PO}_{n}$ has a trivial group of units, the orbits of $\L$-classes and
$\R$-classes of $J_{n - 1} \cap \mathcal{PO}_{n}$ are singletons, and
$\Delta(\mathcal{PO}_{n})$ is isomorphic to the Graham-Houghton graph of the
principal factor of $J_{n - 1} \cap \mathcal{PO}_{n}$.  A picture of
$\Delta(\mathcal{PO}_{n})$ is shown in Figure~\ref{fig-POn-delta}.

%%%%%%%%%%%%%%%%%%%%%%%%%%%%%%%%%%%%%%%%%%%%%%%%%%%%%%%%%%%%%%%%%%%%%%%%%%%%%%%%
% Figure of Delta for POn
%%%%%%%%%%%%%%%%%%%%%%%%%%%%%%%%%%%%%%%%%%%%%%%%%%%%%%%%%%%%%%%%%%%%%%%%%%%%%%%%

\begin{figure}
  \begin{center}
    \begin{tikzpicture}
      % L-classes
      \node[rounded corners,rectangle,draw,fill=blue!20]
        (1)  at (0, 0) {$\{L_{1}\}$};
      \node[rounded corners,rectangle,draw,fill=blue!20]
        (2)  at (4, 0) {$\{L_{2}\}$};
      \node[rounded corners,rectangle,draw,fill=blue!20]
        (3)  at (12, 0) {$\{L_{n-1}\}$};
      \node[rounded corners,rectangle,draw,fill=blue!20]
        (4)  at (16,0) {$\{L_{n}\}$};

      % R-classes of partial perms
      \node[rounded corners,rectangle,draw,fill=blue!20]
        (11) at (0,  3) {$\{R_{1}\}$};
      \node[rounded corners,rectangle,draw,fill=blue!20]
        (12) at (4,  3) {$\{R_{2}\}$};
      \node[rounded corners,rectangle,draw,fill=blue!20]
        (13) at (12,  3) {$\{R_{n-1}\}$};
      \node[rounded corners,rectangle,draw,fill=blue!20]
        (14) at (16, 3) {$\{R_{n}\}$};

      % R-classes of transformations
      \node[rounded corners,rectangle,draw,fill=green!20]
        (21) at (2,  3) {$\{R_{\{1,2\}}\}$};
      \node[rounded corners,rectangle,draw,fill=green!20]
        (22) at (6,  3) {$\{R_{\{2,3\}}\}$};
      \node[rounded corners,rectangle,draw,fill=green!20]
        (23) at (14, 3) {$\{R_{\{n-1,n\}}\}$};

      % Blank spots
      \node (31) at (6.833,  1.75){};
      \node (32) at (7.333,     1){};
      \node (33) at (11.166, 1.25){};
      \node (34) at (10.666,    2){};

      \node (40) at (9, 1.5) {\ldots};

      % straight edges
      \edge{1}{11};
      \edge{2}{12};
      \edge{3}{13};
      \edge{4}{14};

      % diagonal edges
      \edge{1}{21};
      \edge{2}{21};
      \edge{2}{22};
      \edge{3}{23};
      \edge{4}{23};

      % fading edges
      \edge{22}{31};
      \draw[dashed] (22) -- (32);
      \edge{3}{33};
      \draw[dashed] (3)  -- (34);
    \end{tikzpicture}
  \end{center}
  \caption{The graph $\Delta(\mathcal{PO}_{n})$.}\label{fig-POn-delta}
\end{figure}

%%%%%%%%%%%%%%%%%%%%%%%%%%%%%%%%%%%%%%%%%%%%%%%%%%%%%%%%%%%%%%%%%%%%%%%%%%%%%%%%

The group of units of $\mathcal{POD}_{n}$ is $\genset{\gamma_{n}}$, which equals
$\{ \id_{n}, \gamma_{n} \}$.  Thus
$\Delta(\mathcal{POD}_{n})$ is isomorphic to a quotient of
$\Delta(\mathcal{PO}_{n})$. Since $\genset{\gamma_{n}}$ has $\ceiling{n / 2}$
orbits on the set $\n$, there are $\ceiling{n / 2}$ corresponding orbits of
$\L$-classes, and $\ceiling{n / 2}$ orbits of $\R$-classes of partial
permutations.  Furthermore, there are $\floor{n / 2}$ orbits of
$\genset{\gamma_{n}}$ on the set $\bigset{\{i, i + 1\}}{i \in \{ 1,
\ldots, n - 1 \}}$, and these orbits correspond to $\floor{n / 2}$ orbits of
$\R$-classes of transformations.  A picture of $\Delta(\mathcal{POD}_{n})$ is
shown in Figure~\ref{fig-PODn-delta-odd} for odd $n$, and in
Figure~\ref{fig-PODn-delta-even} for even $n$; see these pictures for a
description of the edges of these graphs.  In the case that $n$ is odd, the
graph $\Delta(\mathcal{POD}_{n})$ is isomorphic to
$\Delta(\mathcal{PO}_{\ceiling{n / 2}})$.

%%%%%%%%%%%%%%%%%%%%%%%%%%%%%%%%%%%%%%%%%%%%%%%%%%%%%%%%%%%%%%%%%%%%%%%%%%%%%%%%
% Figure of Delta for PODn, n odd
%%%%%%%%%%%%%%%%%%%%%%%%%%%%%%%%%%%%%%%%%%%%%%%%%%%%%%%%%%%%%%%%%%%%%%%%%%%%%%%%

\begin{figure}
  \begin{center}
    \begin{tikzpicture}
      % L-classes
      \node[rounded corners,rectangle,draw,fill=blue!20]
        (1)  at (0,  0) {$\{L_{1}, L_{n}\}$};
      \node[rounded corners,rectangle,draw,fill=blue!20]
        (2)  at (5,  0) {$\{L_{2}, L_{n-1}\}$};
      \node[rounded corners,rectangle,draw,fill=blue!20]
        (3)  at (15.5, 0) {$\{L_{(n + 1) / 2}\}$};

      % R-classes of partial perms
      \node[rounded corners,rectangle,draw,fill=blue!20]
        (6)  at (0,  3) {$\{R_{1}, R_{n}\}$};
      \node[rounded corners,rectangle,draw,fill=blue!20]
        (8)  at (5,  3) {$\{R_{2}, R_{n-1}\}$};
      \node[rounded corners,rectangle,draw,fill=blue!20]
        (11) at (15.5, 3) {$\{R_{(n + 1) / 2}\}$};

      % R-classes of transformations
      \node[rounded corners,rectangle,draw,fill=green!20]
        (7)  at (2.5,  4) {$\{R_{\{1,2\}}, R_{\{n-1,n\}}\}$};
      \node[rounded corners,rectangle,draw,fill=green!20]
        (9)  at (7.5,  4) {$\{R_{\{2,3\}}, R_{\{n-2,n-1\}}\}$};
      \node[rounded corners,rectangle,draw,fill=green!20]
        (20) at (12.5, 4) {$\{R_{\{(n - 1) / 2,
                                   (n + 1) / 2\}},
                              R_{\{(n + 1) / 2,
                                   (n + 3) / 2\}}\}$};

      % Blank spots
      \node (15) at (8.75,   2) {};
      \node (16) at (9.375,  1) {};
      \node (18) at (11.25,  2) {};
      \node (19) at (10.625, 1) {};
      \node (21) at (10, 2) {\ldots};

      % straight edges
      \edge{1}{6};
      \edge{2}{8};
      \edge{3}{11};

      % diagonal edges
      \edge{1}{7};
      \edge{2}{7};
      \edge{2}{9};
      \edge{3}{20};

      % fading edges
      \edge{9}{15};
      \draw[dashed] (9) -- (16);
      \edge{20}{18};
      \draw[dashed] (20) -- (19);
    \end{tikzpicture}
  \end{center}
  \caption{The graph $\Delta(\mathcal{POD}_{n})$, when $n$ is
  odd.}\label{fig-PODn-delta-odd}
\end{figure}
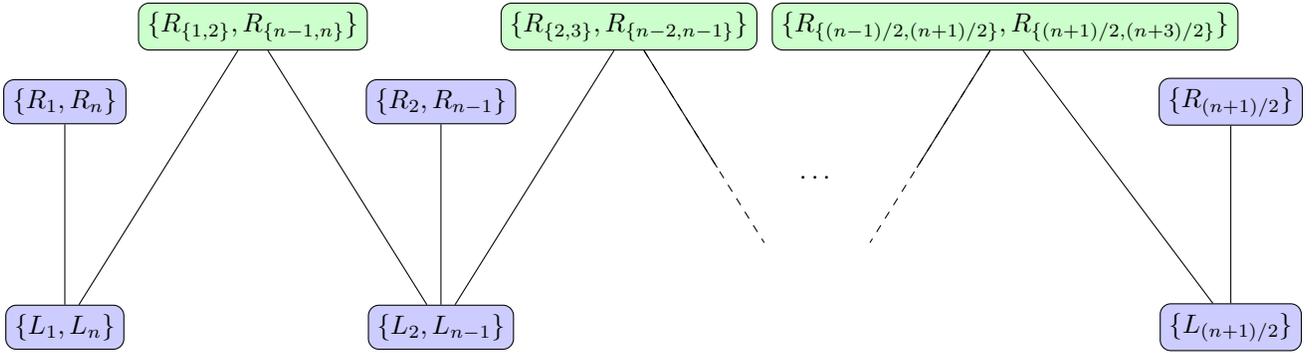

%%%%%%%%%%%%%%%%%%%%%%%%%%%%%%%%%%%%%%%%%%%%%%%%%%%%%%%%%%%%%%%%%%%%%%%%%%%%%%%%

%%%%%%%%%%%%%%%%%%%%%%%%%%%%%%%%%%%%%%%%%%%%%%%%%%%%%%%%%%%%%%%%%%%%%%%%%%%%%%%%
% Figure of Delta for PODn, n even
%%%%%%%%%%%%%%%%%%%%%%%%%%%%%%%%%%%%%%%%%%%%%%%%%%%%%%%%%%%%%%%%%%%%%%%%%%%%%%%%

\begin{figure}
  \begin{center}
    \begin{tikzpicture}
      % L-classes
      \node[rounded corners,rectangle,draw,fill=blue!20]
        (1)  at (0,  0) {$\{L_{1}, L_{n}\}$};
      \node[rounded corners,rectangle,draw,fill=blue!20]
        (2)  at (4.5,  0) {$\{L_{2}, L_{n-1}\}$};
      \node[rounded corners,rectangle,draw,fill=blue!20]
        (3)  at (13.95, 0) {$\{L_{n / 2}, L_{n / 2 + 1}\}$};

      % R-classes of partial perms
      \node[rounded corners,rectangle,draw,fill=blue!20]
        (6)  at (0,  3) {$\{R_{1}, R_{n}\}$};
      \node[rounded corners,rectangle,draw,fill=blue!20]
        (8)  at (4.5,  3) {$\{R_{2}, R_{n-1}\}$};
      \node[rounded corners,rectangle,draw,fill=blue!20]
        (11) at (13.95, 3) {$\{R_{n / 2}, R_{n / 2 + 1}\}$};

      % R-classes of transformations
      \node[rounded corners,rectangle,draw,fill=green!20]
        (7)  at (2.25,  4) {$\{R_{\{1, 2\}}, R_{\{n - 1, n\}}\}$};
      \node[rounded corners,rectangle,draw,fill=green!20]
        (9)  at (6.75,  4) {$\{R_{\{2, 3\}}, R_{\{n - 2, n - 1\}}\}$};
      \node[rounded corners,rectangle,draw,fill=green!20]
        (20) at (11.25, 4) {$\{R_{\{n / 2 - 1, n / 2\}},
                              R_{\{n / 2 + 1, n / 2 + 2\}}\}$};

      \node[rounded corners,rectangle,draw,fill=green!20]
        (30) at (15.9, 4) {$\{R_{\{n / 2, n / 2 + 1\}}\}$};

      % Blank spots
      \node (15) at (7.875,   2) {};
      \node (16) at (8.4375,  1) {};
      \node (18) at (10.125,  2) {};
      \node (19) at (9.5625,  1) {};
      \node (21) at (9, 2) {\ldots};

      % straight edges
      \edge{1}{6};
      \edge{2}{8};
      \edge{3}{11};

      % diagonal edges
      \edge{1}{7};
      \edge{2}{7};
      \edge{2}{9};
      \edge{3}{20};
      \edge{3}{30};

      % fading edges
      \edge{9}{15};
      \draw[dashed] (9) -- (16);
      \edge{20}{18};
      \draw[dashed] (20) -- (19);
    \end{tikzpicture}
  \end{center}
  \caption{The graph $\Delta(\mathcal{POD}_{n})$, when $n$ is
  even.}\label{fig-PODn-delta-even}
\end{figure}
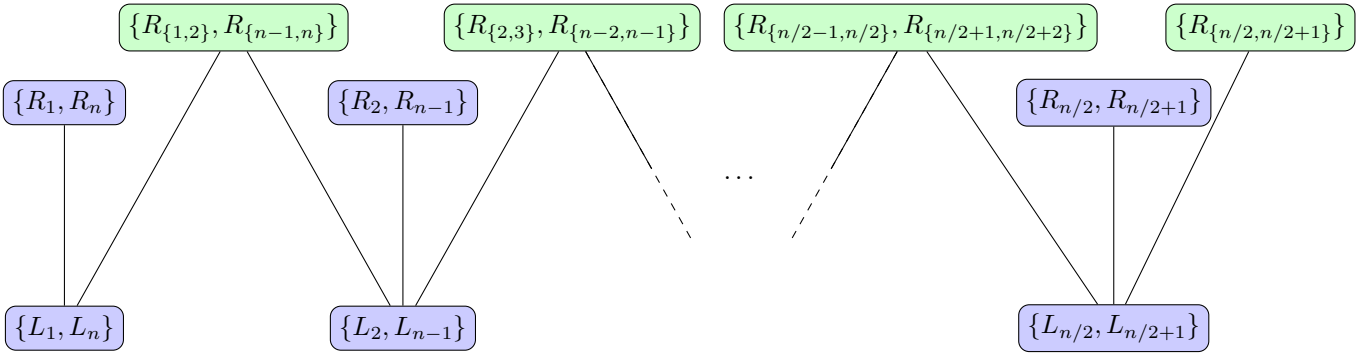

%%%%%%%%%%%%%%%%%%%%%%%%%%%%%%%%%%%%%%%%%%%%%%%%%%%%%%%%%%%%%%%%%%%%%%%%%%%%%%%%

Given these descriptions of $\Delta(\mathcal{PO}_{n})$
and $\Delta(\mathcal{POD}_{n})$, the following lemmas are established.

\begin{lem}\label{lem-POn-independent-subsets}
  Let $K$ be any collection of vertices of the graph $\Delta(\mathcal{PO}_{n})$.
  Then $K$ is a maximal independent subset of $\Delta(\mathcal{PO}_{n})$ if and
  only if
    $$K = \bigset{\{L_{i}\cap\mathcal{PO}_{n}\}}{i \in A}
     \cup \bigset{\{R_{i}\cap\mathcal{PO}_{n}\}}{i \not\in A}
     \cup \bigset{\{R_{\{i, i + 1\}}\cap\mathcal{PO}_{n}\}}{i, i + 1 \not\in
     A}$$
  for some subset $A$ of $\n$.
\end{lem}

\begin{proof}
  ($\Rightarrow$)
  Suppose that $K$ is a maximal independent subset of
  $\Delta(\mathcal{PO}_{n})$. There exists a set $A \subseteq \n$ of indices
  such that $\bigset{\{L_{i} \cap \mathcal{PO}_{n}\}}{i \in A}$ is the
  collection of $\L$-class vertices in $K$.  Since a vertex of the form $\{R_{i}
  \cap \mathcal{PO}_{n}\}$ is adjacent in $\Delta(\mathcal{PO}_{n})$ only to the
  vertex $\{L_{i}\cap\mathcal{PO}_{n}\}$, it follows by the maximality of $K$
  that $\{R_{i} \cap \mathcal{PO}_{n}\} \in K$ if and only if $i\notin A$.
  Similarly, since an orbit of the form $\{R_{\{i, i + 1\}} \cap
  \mathcal{PO}_{n}\}$ is adjacent in $\Delta(\mathcal{PO}_{n})$ only to the
  orbits $\{L_{i}\cap\mathcal{PO}_{n}\}$ and $\{L_{i +
  1}\cap\mathcal{PO}_{n}\}$, it follows that $\{R_{\{i, i + 1\}} \cap
  \mathcal{PO}_{n}\} \in K$ if and only if $i\notin A$ and $i + 1\notin A$.
  Since we have considered all vertices of $\Delta(\mathcal{PO}_{n})$, it
  follows that $K$ has the required form.

  ($\Leftarrow$)
  From the definition of $\Delta(\mathcal{PO}_{n})$, it is easy to verify that
  $K$ is a maximal independent subset of $\Delta(\mathcal{PO}_{n})$.
\end{proof}

The following lemma can be proved using an argument similar to that used in the
proof of Lemma~\ref{lem-POn-independent-subsets}.

\begin{lem}\label{lem-PODn-independent-subsets}
  Let $K$ be any collection of vertices of the graph
  $\Delta(\mathcal{POD}_{n})$.  Then $K$ is a maximal independent subset of
  $\Delta(\mathcal{POD}_{n})$ if and only if $K$ is equal to
  \begin{multline*}
    \bigset{\{L_{i}\cap\mathcal{POD}_{n}, L_{n -i + 1}\cap\mathcal{POD}_{n}\}}{i
    \in A}
        \cup
    \bigset{\{R_{i}\cap\mathcal{POD}_{n}, R_{n -i +
    1}\cap\mathcal{POD}_{n}\}}{i\notin A}\\
        \cup
    \bigset{\{R_{\{i, i + 1\}}\cap\mathcal{POD}_{n},
              R_{\{n - i, n - i + 1\}}\cap\mathcal{POD}_{n}\}}{i, i + 1 \not\in
              A},
 \end{multline*}
  for some subset $A$ of $\{1, \ldots, \ceiling{n / 2}\}$.
\end{lem}

\proofrefs{thm-POn}{thm-PODn}
  The group of units of $\mathcal{PO}_{n}$ is the trivial group $\{\id_{n}\}$,
  and the group of units of $\mathcal{POD}_{n}$ is $\genset{\gamma_{n}}$, which
  has order $2$. Thus by Corollary~\ref{cor-group-of-units}, the maximal
  subsemigroups that arise from the group of units in each instance are as
  described.
  
  Let $S \in \{ \mathcal{PO}_{n}, \mathcal{POD}_{n} \}$.  Since
  $\mathcal{PO}_{n}$ is idempotent generated~\cite[Theorem~3.13]{Gomes1992}, and
  since $\mathcal{POD}_{n} = \genset{\mathcal{PO}_{n}, \gamma_{n}}$, it follows
  by Lemma~\ref{lem-no-type-intersect} that there are no maximal subsemigroups
  of type~\ref{item-intersect} arising from $J_{n - 1} \cap S$. It follows
  directly from Proposition~\ref{prop-rectangle}, and
  Lemmas~\ref{lem-POn-independent-subsets}
  and~\ref{lem-PODn-independent-subsets}, that the maximal subsemigroups of
  type~\ref{item-rectangle} are those described in the theorems. There exist
  vertices of degree $1$ in $\Delta(S)$ --- the orbits of $\R$-classes of
  partial permutations. Each orbit of $\L$-classes is adjacent to such a vertex.
  Thus by Proposition~\ref{prop-remove-l}, there are no maximal subsemigroups of
  type~\ref{item-remove-l} arising from $S$, but by
  Proposition~\ref{prop-remove-r}, each orbit of $\R$-classes can be removed to
  provide a maximal subsemigroup of type~\ref{item-remove-r}. If $S =
  \mathcal{PO}_{n}$, then there are $2n - 1$ maximal subsemigroups of this type,
  and if $S = \mathcal{POD}_{n}$, then there are $n$.  By
  Proposition~\ref{prop-remove-j}, there is no maximal subsemigroup of
  type~\ref{item-remove-j}.
\qed{}

%%%%%%%%%%%%%%%%%%%%%%%%%%%%%%%%%%%%%%%%%%%%%%%%%%%%%%%%%%%%%%%%%%%%%%%%%%%%%%%%
\subsection{$\mathcal{O}_{n}$ and $\mathcal{OD}_{n}$}\label{sec-On-ODn}

The maximal subsemigroups of the singular ideal of $\mathcal{O}_{n}$ were
incorrectly described and counted in~\cite{xiuliang2000classification}: the
given formula for the number of maximal subsemigroups of the singular ideal of
$\mathcal{O}_{n}$ is correct for $2 \leq n \leq 5$, but gives only a lower bound
when $n \geq 6$. A correct description, although no number, was subsequently
given in~\cite{dimitrova2008maximal}. The maximal subsemigroups of the singular
ideal of $\mathcal{OD}_{n}$ were described in~\cite{gyudzhenov2006maximal}. The
group of units of $\mathcal{O}_{n}$ is trivial, and so the maximal subsemigroups
of its singular ideal correspond in an obvious way to the maximal subsemigroups
of $\mathcal{O}_{n}$. However, the group of units of $\mathcal{OD}_{n}$ is
non-trivial, and the units act on the monoid in such a way as to break the
correspondence between the maximal subsemigroups of the singular part, and the
maximal subsemigroups of $\mathcal{OD}_{n}$ itself.
Thus~\cite{dimitrova2008maximal, gyudzhenov2006maximal} solve an essentially
different problem than the description of the maximal subsemigroups of
$\mathcal{OD}_{n}$.

Let $S \in \{\mathcal{O}_{n}, \mathcal{OD}_{n}\}$.  Then $S$ is a regular monoid
and, by definition, $\mathcal{O}_{n} = \mathcal{PO}_{n} \cap \mathcal{T}_{n}$,
and $\mathcal{OD}_{n} = \mathcal{POD}_{n} \cap \mathcal{T}_{n}$.  From the
description of the Green's classes of $\mathcal{PO}_{n}$ and $\mathcal{POD}_{n}$
given in Section~\ref{sec-POn-PODn}, $J_{n - 1} \cap S$ is a regular $\J$-class
of $S$, the set of $\L$-classes of $J_{n - 1} \cap S$ is $\bigset{L_{i} \cap
S}{i \in \n}$, and the set of $\R$-classes of $J_{n - 1} \cap S$ is
$\bigset{R_{\{i, i + 1\}} \cap S}{i \in \{1, \ldots, n - 1\}}$.

We may identify the Green's classes of $J_{n - 1} \cap \mathcal{O}_{n}$ with
those of $J_{n - 1} \cap \mathcal{PO}_{n}$ that contain transformations,
so that $L_{i} \cap \mathcal{O}_{n}$ corresponds with $L_{i} \cap
\mathcal{PO}_{n}$ and $R_{\{i, i+1\}} \cap \mathcal{O}_{n}$ corresponds with
$R_{\{i, i+1\}} \cap \mathcal{PO}_{n}$.  In this way, we obtain the graph
$\Delta(\mathcal{O}_{n})$ as the subgraph of $\Delta(\mathcal{PO}_{n})$ induced
on those orbits of Green's classes that contain transformations.  A
similar statement relates $\Delta(\mathcal{OD}_{n})$ to the induced subgraph of
$\Delta(\mathcal{POD}_{n})$ on its orbits of Green's classes that contain 
transformations.  Thus the definitions of $\Delta(\mathcal{O}_{n})$ and
$\Delta(\mathcal{OD}_{n})$ are contained in those of $\Delta(\mathcal{PO}_{n})$
and $\Delta(\mathcal{POD}_{n})$.

For $k \in \N$, we define the path graph of order $k$ to be the graph with
vertices $\{1, \ldots, k\}$ and edges $$\bigset{\{i, i + 1\}}{i \in \{1, \ldots,
k - 1\}}.$$ The vertices of degree $1$ in the path graph of order $k$ are the
end-points, $1$ and $k$. It is easy to see that $\Delta(\mathcal{O}_{n})$ is
isomorphic to the path graph of order $2n - 1$, via the isomorphism that maps
the orbit $\{ L_{i} \cap \mathcal{O}_{n} \}$ to the vertex $2i - 1$, and maps
the orbit $\{ R_{\{i, i + 1\}} \cap \mathcal{O}_{n} \}$ to the vertex $2i$.
Similarly, $\Delta(\mathcal{OD}_{n})$ is isomorphic to the path graph of order
$n$. A picture of $\Delta(\mathcal{O}_{n})$ is shown in
Figure~\ref{fig-On-delta}.

%%%%%%%%%%%%%%%%%%%%%%%%%%%%%%%%%%%%%%%%%%%%%%%%%%%%%%%%%%%%%%%%%%%%%%%%%%%%%%%%
% Figure of Delta for On
%%%%%%%%%%%%%%%%%%%%%%%%%%%%%%%%%%%%%%%%%%%%%%%%%%%%%%%%%%%%%%%%%%%%%%%%%%%%%%%%

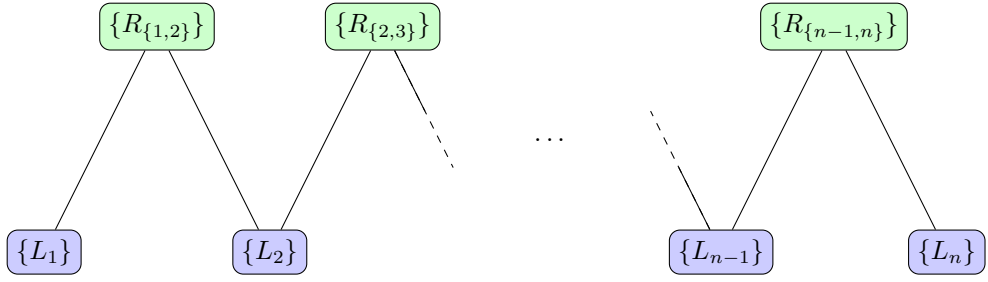
\begin{figure}
  \begin{center}
    \begin{tikzpicture}
      % L-classes
      \node[rounded corners,rectangle,draw,fill=blue!20]
        (1)  at (0, 0) {$\{L_{1}\}$};
      \node[rounded corners,rectangle,draw,fill=blue!20]
        (2)  at (3, 0) {$\{L_{2}\}$};
      \node[rounded corners,rectangle,draw,fill=blue!20]
        (3)  at (9, 0) {$\{L_{n - 1}\}$};
      \node[rounded corners,rectangle,draw,fill=blue!20]
        (4)  at (12,0) {$\{L_{n}\}$};

      % R-classes
      \node[rounded corners,rectangle,draw,fill=green!20]
        (21) at (1.5,  3) {$\{R_{\{1,2\}}\}$};
      \node[rounded corners,rectangle,draw,fill=green!20]
        (22) at (4.5,  3) {$\{R_{\{2,3\}}\}$};
      \node[rounded corners,rectangle,draw,fill=green!20]
        (23) at (10.5, 3) {$\{R_{\{n-1,n\}}\}$};

      % Blank spots
      \node (31) at (5.125, 1.75){};
      \node (32) at (5.5,   1)   {};
      \node (33) at (8.375, 1.25){};
      \node (34) at (8,     2)   {};

      \node (40) at (6.75, 1.5) {\ldots};

      % diagonal edges
      \edge{1}{21};
      \edge{2}{21};
      \edge{2}{22};
      \edge{3}{23};
      \edge{4}{23};

      % fading edges
      \edge{22}{31};
      \draw[dashed] (22) -- (32);
      \edge{3}{33};
      \draw[dashed] (3)  -- (34);
    \end{tikzpicture}
  \end{center}
  \caption{The graph $\Delta(\mathcal{O}_{n})$.}\label{fig-On-delta}
\end{figure}

%%%%%%%%%%%%%%%%%%%%%%%%%%%%%%%%%%%%%%%%%%%%%%%%%%%%%%%%%%%%%%%%%%%%%%%%%%%%%%%%

We can count the number of maximal independent subsets of
$\Delta(\mathcal{O}_{n})$ and $\Delta(\mathcal{OD}_{n})$ using a recurrence
relation.

\begin{lem}\label{lem-path}
  The number $A_{n}$ of maximal independent subsets of the path graph of order
  $n$ satisfies the recurrence relation $A_{1} = 1$, $A_{2} = A_{3} = 2$, and
  $A_{n} = A_{n - 2} + A_{n - 3}$ for $n \geq 4$.
\end{lem}

\begin{proof}
  For $n \in \N$, define $\Gamma_{n}$ to be the path graph of order $n$.  It is
  straightforward to verify that $A_{1} = 1$, and $A_{2} = A_{3} = 2$.
  
  Suppose that $n \geq 4$, and that $K$ is a maximal independent subset of
  $\Gamma_{n - 3}$. Then $K \cup\{n - 1\}$ is a maximal independent subset of
  $\Gamma_{n}$.  Similarly, if $K$ is a maximal independent subset of $\Gamma_{n
  - 2}$, then $K \cup \{ n \}$ is a maximal independent subset of $\Gamma_{n}$.
  In this way, distinct maximal independent subsets of $\Gamma_{n - 3}$ and
  $\Gamma_{n - 2}$ give rise to distinct maximal independent subsets of
  $\Gamma_{n}$. Thus $A_{n} \geq A_{n - 2} + A_{n - 3}$.

  Conversely, suppose that $n \geq 4$ and that $K$ is a maximal independent
  subset of $\Gamma_{n}$.  Since the vertex $n$ has degree $1$ in $\Gamma_{n}$,
  it follows by the maximality of $K$ that precisely one of $n - 1$ and $n$ is a
  member of $K$.  If $n - 1\in K$, then $n - 2 \notin K$, which implies that $K
  \setminus \{n - 1\}$ is a maximal independent subset of $\Gamma_{n - 3}$.
  Otherwise, $K \setminus \{n\}$ is a maximal independent subset of $\Gamma_{n
  - 2}$. Thus $A_{n} \leq A_{n - 2} + A_{n - 3}$.
\end{proof}

Note that $A_{n}$ is equal to the $(n + 6) ^ {\text{th}}$ term of the
Padovan sequence~\cite[\href{http://oeis.org/A000931}{A000931}]{OEIS}.

Given the proof of Lemma~\ref{lem-path}, it is clear that a subset $K \subseteq
\n$ is a maximal independent subset of $\Gamma_{n}$ if and only if $K$ contains
exactly one of $1$ and $2$, $K$ contains no consecutive numbers, and for any $i
\in K$ with $i \leq n - 2$, $K$ contains exactly one of $i + 2$ and $i + 3$.
There are two special maximal
independent subsets of $\Gamma$: the subset of all even vertices, and the subset
of all odd vertices. These maximal independent subsets correspond
to the bicomponents of $\Delta(\mathcal{O}_{n})$ and
$\Delta(\mathcal{OD}_{n})$, and so they are the unique maximal independent
subsets that do not give rise to maximal subsemigroups of $\mathcal{O}_{n}$ and
$\mathcal{OD}_{n}$ of type~\ref{item-rectangle} --- see
Proposition~\ref{prop-rectangle} and Corollary~\ref{cor-rectangle}.

The following theorems are the main results of this section.

\begin{thm}\label{thm-On}
  Let $n \in \N$, $n \geq 2$, be arbitrary and let $\mathcal{O}_{n}$ be the
  monoid of order-preserving transformations on $\n$ with the usual order.  Then
  the maximal subsemigroups of $\mathcal{O}_{n}$ are:
  \begin{enumerate}[label=\emph{(\alph*)}]
    \item
      $\mathcal{O}_{n} \setminus \{\id_{n}\}$
      \emph{(type~\ref{item-remove-j})};
    \item
      the union of $\mathcal{O}_{n} \setminus J_{n - 1}$ and the union of the
      Green's classes in $$\bigset{ L_{(i + 1) / 2} \cap \mathcal{O}_{n}} { i
      \in A,\ i\ \text{is odd}} \cup \bigset{ R_{\{ i / 2, (i / 2) + 1\}} \cap
      \mathcal{O}_{n}}{ i \in A,\ i\ \text{is even}},$$ where $A$ is a maximal
      independent subset of the path graph of order $2n - 1$ that contains both
      odd and even numbers
      \emph{(type~\ref{item-rectangle})};
    \item
      $\mathcal{O}_{n}\setminus L$, where $L$ is any $\L$-class in $J_{n - 1}
      \cap \mathcal{O}_{n}$
      \emph{(type~\ref{item-remove-l})}; and
    \item
      $\mathcal{O}_{n}\setminus R_{\{i, i + 1\}}$, where $i \in \{2, \ldots, n -
      2\}$
      \emph{(type~\ref{item-remove-r})}.
  \end{enumerate}
  In particular, there are $3$ maximal subsemigroups of $\mathcal{O}_{2}$, and
  for $n \geq 3$, there are $A_{2n - 1} + 2n - 4$ maximal subsemigroups of
  $\mathcal{O}_{n}$, where $A_{2n - 1}$ is the $(2n - 1)^{\text{th}}$ term of
  the sequence defined by $A_{1} = 1,\ A_{2} = A_{3} = 2$,\ and $A_{k} = A_{k -
  2} + A_{k - 3}$ for $k \geq 4$.
\end{thm}

\begin{thm}\label{thm-ODn}
  Let $n \in \N$, $n \geq 3$, be arbitrary. Let $\mathcal{OD}_{n}$ be the monoid
  of order-preserving and -reversing transformations on $\n$ with the usual
  order, and let $\gamma_{n}$ be the permutation of degree $n$ that reverses
  this order. Then the maximal subsemigroups of $\mathcal{OD}_{n}$ are:
  \begin{enumerate}[label=\emph{(\alph*)}]
    \item
      $\mathcal{OD}_{n} \setminus \{\gamma_{n}\}$
      \emph{(type~\ref{item-intersect})};
    \item
      the union of $\mathcal{OD}_{n} \setminus J_{n - 1}$ and the union of the
      Green's classes in
      \begin{multline*}
        \bigset{(L_{(i + 1) / 2} \cup L_{n + 1 - (i + 1) / 2}) \cap
        \mathcal{OD}_{n}}{i \in A,\ i\ \text{is odd}}\\
        \cup
        \bigset{(R_{\{ i / 2, (i / 2) + 1\}} \cup R_{\{ n - (i / 2), n + 1 - (i
        / 2) \}}) \cap \mathcal{OD}_{n}}{i \in A,\ i\ \text{is even}},
      \end{multline*}
      where $A$ is a maximal independent subset of the path graph of order $n$
      that contains both odd and even numbers
      \emph{(type~\ref{item-rectangle})};
    \item
      $\mathcal{OD}_{n} \setminus (L_{i} \cup L_{n -i + 1})$, where
      $\begin{cases}
        i \in \{1, \ldots, (n + 1) / 2\} & \text{if}\ n\ \text{is odd,}\\
        i \in \{1, \ldots, n / 2 - 1\}   & \text{if}\ n\ \text{is even}
      \end{cases}$
      \emph{(type~\ref{item-remove-l})}; and
    \item
      $\mathcal{OD}_{n}\setminus (R_{\{i, i + 1\}} \cup R_{\{n -i, n -i +
      1\}})$, where
      $\begin{cases}
        i \in \{2, \ldots, (n - 3) / 2\} & \text{if}\ n\ \text{is odd,}\\
        i \in \{2, \ldots, n / 2\}       & \text{if}\ n\ \text{is even}
      \end{cases}$
      \emph{(type~\ref{item-remove-r})}.
  \end{enumerate}
  In particular, there are $3$ maximal subsemigroups of $\mathcal{OD}_{3}$, and
  for $n \geq 4$, there are $A_{n} + n - 3$ maximal subsemigroups of
  $\mathcal{OD}_{n}$, where $A_{1} = 1$, $A_{2} = A_{3} = 2$, and $A_{n} = A_{n
  - 2} + A_{n - 3}$ for $n \geq 4$.  The monoid $\mathcal{OD}_{2} =
  \mathcal{T}_{2}$ has $2$ maximal subsemigroups.
\end{thm}

\proofrefs{thm-On}{thm-ODn}
  Let $S \in \{\mathcal{O}_{n}, \mathcal{OD}_{n}\}$.  The group of units of
  $\mathcal{O}_{n}$ is trivial and the group of units of $\mathcal{OD}_{n}$ is
  $\genset{\gamma_{n}}$, so by Corollary~\ref{cor-group-of-units}, the maximal
  subsemigroups that arise from the group group of units of $S$ are as stated.
  
  Since $\mathcal{O}_{n}$ is generated by its idempotents of
  rank $n - 1$~\cite{Aizenstat1962aa} and
  $\mathcal{OD}_{n} = \genset{\mathcal{O}_{n},
  \gamma_{n}}$~\cite{Fernandes2005aa}, it follows by
  Lemma~\ref{lem-no-type-intersect} that there are no maximal subsemigroups of
  type~\ref{item-intersect} arising from $J_{n - 1} \cap S$.

  We have already noted that $\Delta(\mathcal{O}_{n})$ and
  $\Delta(\mathcal{OD}_{n})$ are paths of length $2n - 1$ and $n$, respectively.
  It follows from Proposition~\ref{prop-rectangle} that the maximal
  subsemigroups of type~\ref{item-rectangle} are those described in the
  theorems.  By Corollary~\ref{cor-rectangle} and Lemma~\ref{lem-path}, the
  number of maximal subsemigroups of type~\ref{item-rectangle} is $A_{2n - 1} -
  2$ for $\mathcal{O}_{n}$, and $A_{n} - 2$ for $\mathcal{OD}_{n}$.

  To describe the maximal subsemigroups of types~\ref{item-remove-l}
  and~\ref{item-remove-r}, it suffices to identify the two vertices of
  $\Delta(S)$ that are adjacent to the end-points of $\Delta(S)$. From
  this, the description of the maximal subsemigroups of
  type~\ref{item-remove-l} follows from Proposition~\ref{prop-remove-l}, and
  the description of the maximal subsemigroups of type~\ref{item-remove-r}
  follows from Proposition~\ref{prop-remove-r}. In particular, the total number
  of both types of maximal subsemigroups is two less than the number of vertices
  of $\Delta(S)$.

  By Proposition~\ref{prop-remove-j}, and since $n \geq 3$, there is no maximal
  subsemigroup of $S$ of type~\ref{item-remove-j}.
\qed{}

%%%%%%%%%%%%%%%%%%%%%%%%%%%%%%%%%%%%%%%%%%%%%%%%%%%%%%%%%%%%%%%%%%%%%%%%%%%%%%%%
\subsection{$\mathcal{POI}_{n}$ and $\mathcal{PODI}_{n}$}\label{sec-POIn-PODIn}

The maximal subsemigroups of
$\mathcal{POI}_{n}$ are described and counted
in~\cite[Theorem~2]{Ganyushkin2003}, and the maximal subsemigroups of
$\mathcal{PODI}_{n}$ are described and counted
in~\cite[Theorem~4]{dimitrova2009maximal}. Additional information about
$\mathcal{POI}_{n}$ may be found in~\cite{Fernandes2001,
Ganyushkin2003}.

Let $S \in \{\mathcal{POI}_{n}, \mathcal{PODI}_{n}\}$.  Then $S$ is an inverse
monoid, and $J_{n - 1} \cap S$ is a $\J$-class of $S$. By definition,
$\mathcal{POI}_{n} = \mathcal{PO}_{n} \cap \mathcal{I}_{n}$, and
$\mathcal{PODI}_{n} = \mathcal{POD}_{n} \cap \mathcal{I}_{n}$, and so given the
description of the Green's classes of $\mathcal{PO}_{n}$ and $\mathcal{POD}_{n}$
from Section~\ref{sec-POn-PODn}, the set of $\L$-classes of $J_{n - 1} \cap S$
is $\bigset{L_{i} \cap S}{i \in \n}$, and the set of $\R$-classes of $J_{n - 1}
\cap S$ is $\bigset{R_{i} \cap S}{i \in \n}$.

In Theorem~\ref{thm-POIn} we describe the maximal subsemigroups of
$\mathcal{POI}_{n}$, and in Theorem~\ref{thm-PODIn} we describe those of
$\mathcal{PODI}_{n}$.

\begin{thm}\label{thm-POIn}
  Let $n \in \N$, $n \geq 2$, be arbitrary and let $\mathcal{POI}_{n}$ be the
  inverse monoid of order-preserving partial permutations on $\n$ with the usual
  order. Then the maximal subsemigroups of $\mathcal{POI}_{n}$ are:
  \begin{enumerate}[label=\emph{(\alph*)}]
    \item
      $\mathcal{POI}_{n} \setminus \{ \id_{n} \}$
      \emph{(type~\ref{item-remove-j})}; and
    \item
      the union of $\mathcal{POI}_{n}\setminus J_{n - 1}$ and the union of
      $$\bigset{L_{i}\cap \mathcal{POI}_{n}}{i \in A} \cup \bigset{R_{i}\cap
      \mathcal{POI}_{n}}{i \not\in A},$$
      where $A$ is any non-empty proper subset of $\n$
      \emph{(type~\ref{item-rectangle})}.
  \end{enumerate}
  In particular, for $n \geq 2$, there are $2 ^ {n} - 1$ maximal subsemigroups
  of $\mathcal{POI}_{n}$. The monoid $\mathcal{POI}_{1} = \mathcal{PT}_{1}$ is a
  semilattice of order $2$: its maximal subsemigroups are each of its singleton 
  subsets.
\end{thm}

\begin{proof}
  Since $\mathcal{POI}_{n}$ is $\H$-trivial, there are no maximal subsemigroups
  of type~\ref{item-intersect} by Lemma~\ref{lem-no-type-intersect}, and by
  Corollary~\ref{cor-group-of-units}, the maximal subsemigroup arising from the
  group of units is formed by removing the identity. The fact that the group of
  units is trivial implies that there are $n$ orbits of $\L$-classes of $J_{n -
  1} \cap \mathcal{POI}_{n}$, each one being a singleton.  Since the $\R$-class
  $R_{i} \cap \mathcal{POI}_{n}$ is equal to ${(L_{i} \cap
  \mathcal{POI}_{n})}^{- 1}$, it follows by
  Corollary~\ref{cor-delta-projections} that the maximal subsemigroups of
  $\mathcal{POI}_{n}$ arising from $J_{n - 1} \cap \mathcal{POI}_{n}$ are those
  described in the theorem.
\end{proof}

\begin{thm}\label{thm-PODIn}
  Let $n \in \N$, $n \geq 3$, be arbitrary. Let $\mathcal{PODI}_{n}$ be the
  inverse monoid of order-preserving and order-reversing partial permutations on
  $\n$ with the usual order, and let $\gamma_{n}$ be the permutation of degree
  $n$ that reverses this order.  For $i, j \in \n$ define $\alpha_{i, j}$ to be
  the order-preserving partial permutation with domain $\n \setminus \{i\}$ and
  image $\n \setminus \{j\}$, and define $\beta_{i, j}$ to be the
  order-reversing partial permutation with this domain and image. Then the
  maximal subsemigroups of $\mathcal{PODI}_{n}$ are:
  \begin{enumerate}[label=\emph{(\alph*)}]
    \item
      $\mathcal{PODI}_{n}\setminus \{\gamma_{n}\}$
      \emph{(type~\ref{item-intersect})};
    \item
      $(\mathcal{PODI}_{n} \setminus J_{n - 1}) \cup I_{A}$, where $n$ is even,
      \begin{multline*}
        I_{A} = \bigset{\alpha_{i, j},\ \beta_{i, n - j + 1},\ \beta_{n - i + 1,
        j},\ \alpha_{n - i + 1, n - j + 1}}{i, j \in A\ \text{or}\ i, j \not\in
        A}\\ \cup \bigset{\beta_{i, j},\ \alpha_{i, n - j + 1},\ \alpha_{n - i +
        1, j},\ \beta_{n - i + 1, n - j + 1}}{i \in A, j \not\in A\ \text{or}\ i
        \not\in A, j \in A},
      \end{multline*}
      and $A$ is any subset of $\{2, \ldots, n / 2\}$
      \emph{(type~\ref{item-intersect})}; and
    \item
      the union of $\mathcal{PODI}_{n} \setminus J_{n - 1}$ and the union of
      $$\bigset{(L_{i} \cup L_{n - i + 1})\cap \mathcal{PODI}_{n}}{i \in A}
        \cup
        \bigset{(R_{i}\cup R_{n - i + 1})\cap \mathcal{PODI}_{n}}{i \notin A},$$
      where $A$ is any non-empty proper subset of $\{1, \ldots, \ceiling{n /
      2}\}$ \emph{(type~\ref{item-rectangle})}.
  \end{enumerate}
  There are $2$ maximal subsemigroups of $\mathcal{PODI}_{2} = \mathcal{I}_{2}$.
  In particular, for $n \geq 2$ there are $3 \cdot 2 ^ {(n / 2) - 1} - 1$
  maximal subsemigroups of $\mathcal{PODI}_{n}$ when $n$ is even, and $2 ^ {(n +
  1) / 2} - 1$ when $n$ is odd. The monoid $\mathcal{PODI}_{1} =
  \mathcal{PT}_{1}$ is a semilattice of order $2$: its maximal subsemigroups are
  each of its singleton subsets. 
\end{thm}

\begin{proof}
  Since the group of units of $\mathcal{PODI}_{n}$ is $\genset{\gamma_{n}}$,
  which contains two elements, it follows by Corollary~\ref{cor-group-of-units}
  that the maximal subsemigroup arising from the group of units of
  $\mathcal{PODI}_{n}$ is $\mathcal{PODI}_{n} \setminus \{\gamma_{n}\}$.

  The graph $\Delta(\mathcal{PODI}_{n})$ may be obtained as the induced subgraph
  of $\Delta(\mathcal{POD}_{n})$ on those orbits of Green's classes that contain
  partial permutations. In particular, the orbits of $\genset{\gamma_{n}}$ on
  the $\L$-classes of $J_{n - 1} \cap \mathcal{PODI}_{n}$ are the sets
  $\big\{L_{i} \cap \mathcal{PODI}_{n},\ L_{n - i + 1} \cap
  \mathcal{PODI}_{n}\big\}$, for $i \in \{1, \ldots, \ceiling{n / 2}\}$.
  By Corollary~\ref{cor-delta-projections}, the maximal subsemigroups of
  $\mathcal{PODI}_{n}$ that arise from $J_{n - 1} \cap \mathcal{PODI}_{n}$ are
  those of type~\ref{item-rectangle} described in the theorem --- of which there
  are $2 ^ {\ceiling{n / 2}} - 2$ --- along with any maximal subsemigroups of
  type~\ref{item-intersect}.

  Suppose that $M$ is a maximal subsemigroup of $\mathcal{PODI}_{n}$ of
  type~\ref{item-intersect} arising from $J_{n - 1} \cap \mathcal{PODI}_{n}$.
  By~\cite[Proposition~4, Case~1]{Graham1968aa}, the intersection of $M$ with
  each $\H$-class of $J_{n - 1} \cap \mathcal{PODI}_{n}$ is non-empty, and each
  of these intersections has some common size, $q$. Since an $\H$-class in $J_{n
  - 1} \cap \mathcal{PODI}_{n}$ contains two elements, and $M$ is a proper
  subsemigroup of $\mathcal{PODI}_{n}$ that lacks only elements from $J_{n - 1}
  \cap \mathcal{PODI}_{n}$, it follows that $q = 1$.  In other words, the
  intersection of $M$ with each $\H$-class of $J_{n - 1} \cap
  \mathcal{PODI}_{n}$ contains a single element. For $i, j \in \n$, let
  $\delta_{i, j}$ denote the unique element of $M$ that is contained in the
  $\H$-class $\mathcal{PODI}_{n} \cap (L_{i} \cap R_{j}) = \{ \alpha_{i, j},
  \beta_{i, j} \}$ of $\mathcal{PODI}_{n}$.  In other words, $M \cap (L_{i} \cap
  R_{j}) = \{\delta_{i, j}\}$.

  Since $M$ contains $\gamma_{n}$, it follows that $\delta_{i, j} \in M$ if and
  only if $\delta_{i, j} \gamma_{n},\ \gamma_{n} \delta_{i, j},\ \gamma_{n}
  \delta_{i, j} \gamma_{n} \in M$. In particular, $\alpha_{i, j} \in M$ if and
  only if $\alpha_{i, j},\ \beta_{i, n - j + 1},\ \beta_{n - i + 1, j},\
  \alpha_{n - i + 1, n - j + 1} \in M$, and $\beta_{i, j} \in M$ if and only if
  $\beta_{i, j},\ \alpha_{i, n - j + 1},\ \alpha_{n - i + 1, j},\ \beta_{n - i +
  1, n - j + 1} \in M$.  For odd $n$, this leads to the contradictory statement
  that $\alpha_{(n + 1) / 2, (n + 1) / 2} \in M$ if and only if $\beta_{(n + 1)
  / 2, (n + 1) / 2} \in M$, and so $n$ is even.

  Given these observations, in order to describe $M$, it suffices to specify
  $\delta_{i, j}$ for each $i, j \in \{1, \ldots, n / 2\}$. Indeed, our
  description can be even more concise.  We observe that $\delta_{i, i} =
  \alpha_{i, i}$, since $M$ contains every idempotent of $\mathcal{PODI}_{n}$.
  This implies that $\delta_{i, j}\delta_{j, i} = \delta_{i, i} = \alpha_{i,
  i}$, and so $\delta_{i, j} = \alpha_{i, j}$ if and only if $\delta_{j, i} =
  \alpha_{j, i}$.  Furthermore, $\delta_{i, j} = \delta_{i, 1} \delta_{1, j}$.
  Thus, to specify $\delta_{i, j}$ for each $i, j \in \n$, it suffices to
  specify $\delta_{1, i}$ for each $i \in \{2, \ldots, n / 2\}$.  Let $A =
  \bigset{i \in \{2, \ldots, n / 2\}}{\delta_{1, i} = \beta_{1, j}}$. A routine
  calculation shows that $M = (\mathcal{PODI}_{n} \setminus J_{n - 1}) \cup
  I_{A}$, where $I_{A}$ is the set defined in the statement of the theorem.

  Conversely, for an even number $n \geq 4$ and a subset $A \subseteq \{2,
  \ldots, n / 2\}$, it is tedious, but routine, to verify that
  $(\mathcal{PODI}_{n} \setminus J_{n - 1}) \cup I_{A}$ is a subsemigroup of
  $\mathcal{PODI}_{n}$; by construction, it intersects every $\H$-class of
  $\mathcal{PODI}_{n}$ non-trivially.
  Any maximal subsemigroup of $\mathcal{PODI}_{n}$ that contains
  $(\mathcal{PODI}_{n} \setminus J_{n - 1}) \cup I_{A}$ has
  type~\ref{item-intersect}, and so by the preceding arguments,  we see that it
  is equal to $(\mathcal{PODI}_{n} \setminus J_{n - 1}) \cup I_{A}$.  Thus
  $(\mathcal{PODI}_{n} \setminus J_{n - 1}) \cup I_{A}$ is a maximal
  subsemigroup of $\mathcal{PODI}_{n}$ of type~\ref{item-intersect}.
  
  For two subsets $A, A' \subseteq \{2, \ldots, n / 2\}$, it is clear from the
  definitions that $I_{A} = I_{A'}$ if and only if $A = A'$. Thus there are $2 ^
  {(n / 2) - 1}$ maximal subsemigroups of type~\ref{item-intersect} when $n$ is
  even, and none when $n$ is odd.
\end{proof}

%%%%%%%%%%%%%%%%%%%%%%%%%%%%%%%%%%%%%%%%%%%%%%%%%%%%%%%%%%%%%%%%%%%%%%%%%%%%%%%%
\subsection{$\mathcal{POP}_{n}$ and $\mathcal{POR}_{n}$}\label{sec-POPn-PORn}

As far as we are aware, the maximal subsemigroups of $\mathcal{POP}_{n}$ and
$\mathcal{POR}_{n}$ have not been previously considered in the literature.  To
state the results of this section,
we require the following notation. Let $S \in \{\mathcal{POP}_{n},
\mathcal{POR}_{n} \}$.  Then $J_{n - 1} \cap S$ is a regular $\J$-class of $S$.
The $\L$-classes of $J_{n - 1} \cap S$ are the sets $L_{i} \cap S$ for each $i
\in \n$, and the $\R$-classes of $J_{n - 1} \cap S$ are the sets $R_{i} \cap S$
for each $i \in \n$ and $R_{\{i,i+1\}} \cap S$ for each $i \in \{1, \ldots, n -
1\}$, along with the set $R_{\{1, n\}} \cap S$. The group of units of
$\mathcal{POP}_{n}$ is $\mathcal{C}_{n}$, and the group of units of
$\mathcal{POR}_{n}$ is $\mathcal{D}_{n}$ --- see
Section~\ref{sec-trans-definitions} for the definitions of these groups.

The following theorems are the main results of this section.

\begin{thm}\label{thm-POPn}
  Let $n \in \N$, $n \geq 2$, be arbitrary and let $\mathcal{POP}_{n}$
  be the monoid of orientation-preserving partial transformations on $\n$ with
  the usual order.  Then the maximal subsemigroups of $\mathcal{POP}_{n}$ are:
  \begin{enumerate}[label=\emph{(\alph*)}]
    \item
      $(\mathcal{POP}_{n} \setminus \mathcal{C}_{n}) \cup L$, where $L$ is a
      maximal subgroup of the cyclic group $\mathcal{C}_{n}$
      \emph{(type~\ref{item-intersect})};
    \item
      $\mathcal{POP}_{n} \setminus \bigset{\alpha \in \mathcal{OP}_{n}}{
      \rank(\alpha) = n - 1}$
      \emph{(type~\ref{item-remove-r})}; and
    \item
      $\mathcal{POP}_{n} \setminus \set{\alpha \in \mathcal{POPI}_{n}}{
      \rank(\alpha) = n - 1}$
      \emph{(type~\ref{item-remove-r})}.
  \end{enumerate}
  The monoid $\mathcal{POP}_{1} = \mathcal{PT}_{1}$ is a semilattice of order
  $2$: its maximal subsemigroups are each of its singleton subsets.  In
  particular, for $n \in \N$, there are $|\mathbb{P}_{n}| + 2$ maximal
  subsemigroups of $\mathcal{POP}_{n}$, where $\mathbb{P}_{n}$ is the set of
  primes that divide $n$. 
\end{thm}

\begin{thm}\label{thm-PORn}
  Let $n \in \N$, $n \geq 2$, be arbitrary and let $\mathcal{POR}_{n}$
  be the monoid of orientation-preserving and -reversing partial transformations
  on $\n$ with the usual order. Then the maximal subsemigroups of
  $\mathcal{POR}_{n}$ are:
  \begin{enumerate}[label=\emph{(\alph*)}]
    \item
      $(\mathcal{POR}_{n} \setminus \mathcal{D}_{n}) \cup L$, where $L$ is a
      maximal subgroup of the group $\mathcal{D}_{n}$
      \emph{(type~\ref{item-intersect})};
    \item
      $\mathcal{POR}_{n} \setminus \set{\alpha \in \mathcal{OR}_{n}}{
      \rank(\alpha) = n - 1}$
      \emph{(type~\ref{item-remove-r})};
      and
    \item
      $\mathcal{POR}_{n} \setminus \bigset{\alpha \in \mathcal{PORI}_{n}}{
      \rank(\alpha) = n - 1}$
      \emph{(type~\ref{item-remove-r})}.
  \end{enumerate}
  In particular, there are $3$ maximal subsemigroups of $\mathcal{POR}_{2}$, and
  for $n \geq 3$, there are $\sum_{p \in \mathbb{P}_{n}} p + 3$ maximal
  subsemigroups of $\mathcal{POR}_{n}$, where $\mathbb{P}_{n}$ is the set of the
  primes that divide $n$.  The monoid $\mathcal{POR}_{1} = \mathcal{PT}_{1}$ is
  a semilattice of order $2$: its maximal subsemigroups are each of its
  singleton subsets.
\end{thm}

\proofrefs{thm-POPn}{thm-PORn}
  Let $S \in \{ \mathcal{POP}_{n}, \mathcal{POR}_{n} \}$, and let $G$ be the
  group of units of $S$.  The maximal subsemigroups arising from the group of
  units follow by Lemma~\ref{lem-maximals-cyclic},
  Lemma~\ref{lem-maximals-dihedral}, and Corollary~\ref{cor-group-of-units}.
  Since $\mathcal{PO}_{n}$ is idempotent
  generated~\cite[Theorem~3.13]{Gomes1992}, and $S = \genset{\mathcal{PO}_{n},\
  G}$, it follows by Lemma~\ref{lem-no-type-intersect} that there are no maximal
  subsemigroups of type~\ref{item-intersect} arising from $J_{n - 1} \cap S$.

  The remainder of the proof is similar to the proof of
  Theorem~\ref{thm-maximals-PTn-Tn-In} concerning the maximal subsemigroups of
  $\mathcal{PT}_{n}$.  The group of units $G$ of $S$ acts transitively on the
  $\L$-classes of $J_{n - 1} \cap S$, and so there are no maximal subsemigroups
  of types~\ref{item-rectangle} and~\ref{item-remove-l} by
  Lemma~\ref{lem-transitive}. On the other hand, $G$ has two orbits on the set
  of $\R$-classes of $J_{n - 1} \cap S$: it transitively permutes the
  $\R$-classes of transformations, and it transitively permutes the
  $\R$-classes of partial permutations.  By Proposition~\ref{prop-remove-r}, the
  two maximal subsemigroups of $S$ of type~\ref{item-remove-r} are found by
  removing either the partial permutations, or the transformations, of
  rank $n - 1$. By Proposition~\ref{prop-remove-j}, there is no maximal
  subsemigroup of type~\ref{item-remove-j}.
\qed{}

%%%%%%%%%%%%%%%%%%%%%%%%%%%%%%%%%%%%%%%%%%%%%%%%%%%%%%%%%%%%%%%%%%%%%%%%%%%%%%%%
\subsection{$\mathcal{OP}_{n}$ and $\mathcal{OR}_{n}$}\label{sec-OPn-ORn}

The maximal subsemigroups of $\mathcal{OP}_{n}$ and $\mathcal{OR}_{n}$ were
originally described in~\cite{dimitrova2012maximal}.  We restate these results
in the following theorems.

\begin{thm}\label{thm-OPn}
  Let $n \in \N$, $n \geq 2$, be arbitrary and let $\mathcal{OP}_{n}$
  be the monoid of orientation-preserving transformations on $\n$ with the usual
  order.  Then the maximal subsemigroups of $\mathcal{OP}_{n}$ are:
  \begin{enumerate}[label=\emph{(\alph*)}]
    \item
      $(\mathcal{OP}_{n} \setminus \mathcal{C}_{n}) \cup L$, where $L$ is a
      maximal subgroup of the cyclic group $\mathcal{C}_{n}$
      \emph{(type~\ref{item-intersect})}; and
    \item
      $\mathcal{OP}_{n} \setminus \set{\alpha \in \mathcal{OP}_{n}}{
      \rank(\alpha) = n - 1}$
      \emph{(type~\ref{item-remove-j})}.
  \end{enumerate}
  In particular, for $n \geq 2$, there are $|\mathbb{P}_{n}| + 1$ maximal
  subsemigroups of $\mathcal{OP}_{n}$, where $\mathbb{P}_{n}$ is the set of
  primes that divide $n$
\end{thm}

\begin{thm}\label{thm-ORn}
  Let $n \in \N$, $n \geq 2$, be arbitrary and let $\mathcal{OR}_{n}$ be the
  monoid of orientation-preserving and orientation-reversing transformations on
  $\n$ with the usual order. Then the maximal subsemigroups of
  $\mathcal{OR}_{n}$ are:
  \begin{enumerate}[label=\emph{(\alph*)}]
    \item
      $(\mathcal{OR}_{n} \setminus \mathcal{D}_{n}) \cup L$, where $L$ is a
      maximal subgroup of the group $\mathcal{D}_{n}$
      \emph{(type~\ref{item-intersect})}; and
    \item
      $\mathcal{OR}_{n} \setminus \set{\alpha \in \mathcal{OR}_{n}}{
      \rank(\alpha) = n - 1}$
      \emph{(type~\ref{item-remove-j})}.
  \end{enumerate}
  In particular, there are $4$ maximal subsemigroups of $\mathcal{OR}_{2}$, and
  for $n \geq 3$, there are $\sum_{p \in \mathbb{P}_{n}} p + 2$ maximal
  subsemigroups of $\mathcal{OR}_{n}$, where $\mathbb{P}_{n}$ is the set of the
  primes that divide $n$.
\end{thm}

\proofrefs{thm-OPn}{thm-ORn}
  Let $S \in \{\mathcal{OP}_{n}, \mathcal{OR}_{n}\}$, and let $G$ be the group
  of units of $S$. The group of units of $\mathcal{OP}_{n}$ is $\mathcal{C}_{n}$
  and the group of units of $\mathcal{OR}_{n}$ is $\mathcal{D}_{n}$.  Therefore
  the description of the maximal subsemigroups arising from $G$ follows by
  Corollary~\ref{cor-group-of-units} and Lemma~\ref{lem-maximals-cyclic}, or
  Lemma~\ref{lem-maximals-dihedral}, as appropriate.  Let $\alpha \in J_{n - 1}
  \cap S$, and let $\varepsilon$ be an idempotent of $\mathcal{O}_{n}$ of rank
  $n - 1$.  Clearly there exist permutations $\sigma, \tau \in G$ such that
  $\ker(\sigma \alpha \tau) = \ker(\varepsilon)$ and $\im(\sigma \alpha \tau) =
  \im(\varepsilon)$, i.e. $\sigma \alpha \tau \in H_{\varepsilon}^{S}$.
  Therefore $\varepsilon = (\sigma \alpha \tau)^{k}$ for some $k \in \N$.  Since
  $\mathcal{O}_{n}$ is generated by its idempotents of rank $n -
  1$~\cite{Aizenstat1962aa} and $S = \genset{G,\ \mathcal{O}_{n}}$, it follows
  that $S = \genset{G,\ \alpha}$ if and only if $\alpha \in J_{n - 1} \cap S$.
  The result follows by Corollary~\ref{cor-Xi}.
\qed{}

%%%%%%%%%%%%%%%%%%%%%%%%%%%%%%%%%%%%%%%%%%%%%%%%%%%%%%%%%%%%%%%%%%%%%%%%%%%%%%%%
\subsection{$\mathcal{POPI}_{n}$ and
$\mathcal{PORI}_{n}$}\label{sec-POPIn-PORIn}

The maximal subsemigroups of the inverse monoids $\mathcal{POPI}_{n}$ and
$\mathcal{PORI}_{n}$ have not been previously determined in the literature, as
far as we are aware. These monoids exhibit maximal subsemigroups of
type~\ref{item-intersect} arising from a $\J$-class covered by the group of
units, and to which we can apply the results of
Section~\ref{sec-intersect-inverse}.

Let $S \in \{\mathcal{POPI}_{n}, \mathcal{PORI}_{n}\}$. Then $J_{n - 1} \cap S$
is a regular $\J$-class of $S$ consisting of partial permutations. By
definition, $\mathcal{POPI}_{n} = \mathcal{POP}_{n} \cap \mathcal{I}_{n}$ and
$\mathcal{PORI}_{n} = \mathcal{POR}_{n} \cap \mathcal{I}_{n}$. Therefore the
group of units of $\mathcal{POPI}_{n}$ is $\mathcal{C}_{n}$ and the group of
units of $\mathcal{PORI}_{n}$ is $\mathcal{D}_{n}$, and given the description of
the Green's classes of $\mathcal{POP}_{n}$ and $\mathcal{POR}_{n}$ in
Section~\ref{sec-POPn-PORn}, it follows that the $\L$-classes and $\R$-classes
of $J_{n - 1} \cap S$ are $\bigset{L_{i} \cap S}{i \in \n}$ and $\bigset{R_{i}
\cap S}{i \in \n}$, respectively. See Section~\ref{sec-trans-definitions} for
more information about this notation.

In the following theorems, which are the main results of this section, we use
Proposition~\ref{prop-regularstar-intersect} to describe the maximal
subsemigroups of $\mathcal{POPI}_{n}$ and $\mathcal{PORI}_{n}$.

\begin{thm}\label{thm-POPIn}
  Let $n \in \N$, $n \geq 3$, be arbitrary and let $\mathcal{POPI}_{n}$ be the
  inverse monoid of orientation-preserving partial permutations of
  $\n$ with the usual order. For $k \in \N$, let $\mathbb{P}_{k}$ denote the set
  of all primes that divide $k$. Then the maximal subsemigroups of
  $\mathcal{POPI}_{n}$ are:
  \begin{enumerate}[label=\emph{(\alph*)}]
    \item
      $(\mathcal{POPI}_{n} \setminus \mathcal{C}_{n}) \cup U$, where $U$ is a
      maximal subgroup of $\mathcal{C}_{n}$
      \emph{(type~\ref{item-intersect})}, and
    \item
      $\genset{\mathcal{POPI}_{n} \setminus J_{n - 1},\ \zeta_{n}^{p}}$,
      where $p \in \mathbb{P}_{n - 1}$ and the partial permutation
      $\zeta_{n}$ is defined by
      $$\zeta_{n} =
      \begin{pmatrix}
        1 & 2 & \cdots & n - 2 & n - 1 & n\\
        2 & 3 & \cdots & n - 1 &     1 & -
        % TODO should we define two-line notation for partial transformations?
      \end{pmatrix}$$
      \emph{(type~\ref{item-intersect})}.
  \end{enumerate}
  In particular, there are $|\mathbb{P}_{n}| + |\mathbb{P}_{n - 1}|$ maximal
  subsemigroups of $\mathcal{POPI}_{n}$ for $n \geq 3$. For $i \in \{1, 2\}$,
  $\mathcal{POPI}_{n} = \mathcal{I}_{n}$; see
  Theorem~\ref{thm-maximals-PTn-Tn-In}.
\end{thm}

\begin{proof}
  By Corollary~\ref{cor-group-of-units} and Lemma~\ref{lem-maximals-cyclic}, the
  maximal subsemigroups arising from the group of units are those described in
  the statement of the theorem, and there are $|\mathbb{P}_{n}|$ of them.  It
  remains to describe the maximal subsemigroups that arise from $J_{n - 1} \cap
  \mathcal{POPI}_{n}$.
  
  Define $$H = H_{\id_{n - 1}}^{\mathcal{POPI}_{n}} = \bigset{ \alpha \in
  \mathcal{POPI}_{n}}{\dom(\alpha) = \im(\alpha) = \{1, \ldots, n - 1\}}.$$ 
  Then $H$ is a group $\H$-class of $\mathcal{POPI}_{n}$ contained in the
  $\J$-class $J_{n - 1} \cap \mathcal{POPI}_{n}$. Note that $H$ is isomorphic to
  the cyclic group of order $n - 1$, and is generated by $\zeta_{n}$.
  The setwise
  stabilizer $\stab_{\mathcal{C}_{n}}(H)$ is equal to the pointwise stabilizer
  of $n$ in $\mathcal{C}_{n}$.  Since this stabilizer is trivial,
  Proposition~\ref{prop-regularstar-intersect} implies that any maximal subgroup
  $U$ of $H$ gives rise to a maximal subsemigroup $\genset{\mathcal{POPI}_{n}
  \setminus J_{n - 1},\ U}$ of $\mathcal{POPI}_{n}$.  By
  Lemma~\ref{lem-maximals-cyclic}, the maximal subgroups of $H$ are
  $\genset{\zeta_{n}^{p}}$ for each $p \in \mathbb{P}_{n - 1}$.  It follows from
  Proposition~\ref{prop-regularstar-intersect} that these are the only maximal
  subsemigroups to arise from $J_{n - 1} \cap \mathcal{POPI}_{n}$.
\end{proof}

\begin{thm}\label{thm-PORIn}
  Let $n \in \N$, $n \geq 4$, be arbitrary and let $\mathcal{PORI}_{n}$ be the
  inverse monoid of orientation-preserving or -reversing partial
  permutations of $\n$ with the usual order. For $k \in \N$, let
  $\mathbb{P}_{k}$ denote the set of all primes that divide $k$.  Then the
  maximal subsemigroups of $\mathcal{PORI}_{n}$ are:
  \begin{enumerate}[label=\emph{(\alph*)}]
    \item
      $(\mathcal{PORI}_{n} \setminus \mathcal{D}_{n}) \cup U$, where $U$ is a
      maximal subgroup of $\mathcal{D}_{n}$ \emph{(type~\ref{item-intersect})},
      and
    \item
      $\genset{\mathcal{PORI}_{n} \setminus J_{n - 1},\ \zeta_{n}^{p},\ 
               \tau_{n}}$,
      where $p \in \mathbb{P}_{n - 1}$, and the partial permutations $\zeta_{n}$
      and $\tau_{n}$ are defined by
      $$\zeta_{n} =
      \begin{pmatrix}
        1 & 2 & \cdots & n - 2 & n - 1 & n\\
        2 & 3 & \cdots & n - 1 & 1     & -
      \end{pmatrix},
      \quad \text{and} \quad
      \tau_{n} =
      \begin{pmatrix}
        1     & 2     & \cdots & n - 1 & n\\
        n - 1 & n - 2 & \cdots & 1     & -
      \end{pmatrix}$$
      \emph{(type~\ref{item-intersect})}.
  \end{enumerate}
  For $i \in \{1, 2, 3\}$, $\mathcal{PORI}_{n} = \mathcal{I}_{n}$.  In
  particular, there are $1 + |\mathbb{P}_{n - 1}| + \sum_{p \in \mathbb{P}_{n}}
  p$ maximal subsemigroups of $\mathcal{PORI}_{n}$ for $n \geq 3$.
\end{thm}

\begin{proof}
  The group of units of $\mathcal{PORI}_{n}$ is $\mathcal{D}_{n}$, and so by
  Corollary~\ref{cor-group-of-units} and Lemma~\ref{lem-maximals-dihedral}, the
  maximal subsemigroups arising from the group of units are as stated in the
  theorem, and there are $1 + \sum_{p \in \mathbb{P}_{n}} p$ of them.

  Define $$H = H_{\id_{n - 1}}^{\mathcal{PORI}_{n}} = \bigset{ \alpha \in
  \mathcal{PORI}_{n}}{\dom(\alpha) = \im(\alpha) = \{1, \ldots, n - 1\}}.$$
  Then $H$ is a group $\H$-class of $\mathcal{PORI}_{n}$ contained in the
  $\J$-class $J_{n - 1} \cap \mathcal{POPI}_{n}$.
  Note that since $n \geq 4$, $H$ is a dihedral group of order $2(n - 1)$, and
  it is generated by the partial permutations $\zeta_{n}$ and $\tau_{n}$.
  An element of
  $\mathcal{D}_{n}$ belongs to the setwise stabilizer
  $\stab_{\mathcal{D}_{n}}(H)$ if and only if it stabilizes the set $\{1,
  \ldots, n - 1\}$. Equivalently, $\stab_{\mathcal{D}_{n}}(H)$ consists of those
  permutations in $\mathcal{D}_{n}$ that fix the point $n$.  Thus
  $\stab_{\mathcal{D}_{n}}(H)$ contains only $\id_{n}$ and the permutation in
  $\mathcal{D}_{n}$ that fixes $n$ and reverses the order of $\{1, \ldots, n -
  1\}$.  In particular,
  $$\id_{n - 1} \stab_{\mathcal{D}_{n}}(H) =
  \bigset{\id_{n - 1} \cdot h}{h \in \stab_{\mathcal{D}_{n}}(H)} =
  \{\id_{n - 1}, \tau_{n} \}.$$
  Since any subgroup of $H$ contains $\id_{n - 1}$, it follows from
  Proposition~\ref{prop-regularstar-intersect} that the maximal subsemigroups
  arising from $\mathcal{PORI}_{n} \cap J_{n - 1}$ are
  $\genset{\mathcal{PORI}_{n} \setminus J_{n - 1},\ U}$, for each maximal
  subgroup $U$ of $H$ that contains $\tau_{n}$.  By
  Lemma~\ref{lem-maximals-dihedral}, the maximal subgroups of $H$ are
  $\genset{\zeta_{n}}$ and the subgroups $\genset{\zeta_{n}^{p},\ 
  \tau_{n} \zeta_{n}^{i}}$, where $p \in \mathbb{P}_{n - 1}$ and
  $0 \leq i \leq p - 1$.  Thus
  the maximal subgroups of $H$ that contain $\tau_{n}$ are those in the latter
  form
  where $i = 0$. It follows that the maximal subsemigroups arising from
  $\mathcal{PORI}_{n} \cap J_{n - 1}$ are those stated in the theorem, and that
  there are $|\mathbb{P}_{n - 1}|$ such maximal subsemigroups.
\end{proof}

%%%%%%%%%%%%%%%%%%%%%%%%%%%%%%%%%%%%%%%%%%%%%%%%%%%%%%%%%%%%%%%%%%%%%%%%%%%%%%%%
\section{Diagram monoids}\label{sec-diagram}

In this section, we determine the maximal subsemigroups of
those monoids of partitions defined in Section~\ref{sec-diagram-definitions}.
It is clear that each of these monoids is closed under the $^{\ast}$ operation,
and is therefore a regular $\ast$-monoid.  Furthermore, the monoid
$\mathcal{I}_{n}^{*}$ and its submonoid $\mathfrak{F}_{n}$ are inverse; the
remaining monoids are not inverse, in general.

There exists a natural injective function from the partial transformation
monoid of degree $n$ to the partition monoid of degree $n$.  Thus we may think
of partitions as generalisations of transformations.  In this way,
the notion of a \emph{planar} partition is a generalisation of the notion of an
\emph{order-preserving} partial transformation, and the notion of an
\emph{annular} partition is a generalisation of the notion of an
\emph{orientation-preserving} partial transformation.

Let $\alpha \in \mathcal{P}_{n}$.  We define $\ker(\alpha)$, the \emph{kernel}
of $\alpha$, to be the restriction of the equivalence $\alpha$ to $\n$. We also
define $\dom(\alpha)$, the \emph{domain} of $\alpha$, to be the subset of $\n$
comprising those points that are contained in a transverse block of $\alpha$.
Given these definitions, we define $\coker(\alpha) = \ker(\alpha^{*})$ and
$\codom(\alpha) = \dom(\alpha^{*})$, the \emph{cokernel} and \emph{codomain} of
$\alpha$, respectively. For the majority of the monoids defined in
Section~\ref{sec-diagram-definitions}, the Green's relations are completely
determined by domain, kernel, and rank.
\begin{lem}
  [\!\!\mbox{\cite{Mazorchuk1998aa},~\cite[Theorem~17]{Wilcox2007aa},
  \cite[Theorem~5]{maltcev2007}, and~\cite[Theorem~2.4]{Dolinka2017251}}]
  \label{lem-green-partition}
  Let $S \in \{ \mathcal{P}_{n}, \mathcal{PB}_{n}, \mathcal{B}_{n},
  \mathcal{I}_{n}^{*}, \mathcal{M}_{n}, \mathcal{J}_{n} \}$ and let $\alpha,
  \beta \in S$.
  Then:
  \begin{enumerate}[label=\emph{(\alph*)}, ref=(\alph*)]
    \item\label{item-partition-green-r}
      $\alpha \R \beta$ if and only if $\dom(\alpha) = \dom(\beta)$ and
      $\ker(\alpha) = \ker(\beta)$;
    \item\label{item-partition-green-l}
      $\alpha \L \beta$ if and only if $\alpha^{*} \R \beta^{*}$, i.e.\
      $\alpha \L \beta$ if and only if
      $\codom(\alpha) = \codom(\beta)$ and  $\coker(\alpha) = \coker(\beta)$;
      and
    \item\label{item-partition-green-j}
      $\alpha \J \beta$ if and only if $\rank(\alpha) = \rank(\beta)$.
  \end{enumerate}
\end{lem}
Parts~\ref{item-partition-green-r} and~\ref{item-partition-green-l} of
Lemma~\ref{lem-green-partition} hold for all of the submonoids of
$\mathcal{P}_{n}$ defined in Section~\ref{sec-diagram-definitions}, since each
is a regular submonoid of
$\mathcal{P}_{n}$~\cite[Proposition~2.4.2]{Howie1995aa}.  The condition $\alpha
\J \beta$ if and only if $\rank(\alpha) = \rank(\beta)$ can be shown to hold for
the remaining monoids defined in Section~\ref{sec-diagram-definitions}, with
the exception of $\mathfrak{F}_{n}$ --- although here the condition does hold
for uniform block bijections of ranks $n$ or $n - 1$. For $n \in \N$ and $k \in
\{0, 1, \ldots, n\}$, we define $J_{k} = \set{\alpha \in
\mathcal{P}_{n}}{\rank(\alpha) = k}$ to be the $\J$-class of $\mathcal{P}_{n}$
that comprises the partitions of rank $k$.

In general, if $S$ is any of the diagram monoids defined in
Section~\ref{sec-diagram-definitions}, then $S$ has a unique $\J$-class $J$ that
is covered by
the group of units of $S$.  In several cases, to determine the maximal
subsemigroups of $S$ that arise from $J$, we require the graph $\Delta(S, J)$,
as defined in Section~\ref{sec-type-234}.  Given a description of the
$\L$-classes and $\R$-classes of $J$, to describe $\Delta(S, J)$ it remains to
describe the action of the group of units on the $\R$-classes of $J$.  Since $S$
is a regular $\ast$-monoid, a description of the action of the group of units on
the $\L$-classes of $J$ is obtained as a consequence; see
Section~\ref{sec-type-234}.  We observe that, for
$\alpha \in \mathcal{P}_{n}$ and $\sigma \in \mathcal{S}_{n}$,
\begin{equation}\label{eq-partition-action}
  \dom(\sigma\alpha) = \bigset{i\sigma^{-1}}{i \in \dom(\alpha)}
\quad
\text{and}
\quad
\ker(\sigma\alpha) = \bigset{\big(i\sigma^{-1}, j\sigma^{-1}\big)}{(i, j) \in
\ker(\alpha)}.
\end{equation}
Given this description and Lemma~\ref{lem-green-partition}, the action of a
subgroup of $\mathcal{S}_{n}$ on the $\R$-classes of a particular $\J$-class is
straightforward to determine.

%%%%%%%%%%%%%%%%%%%%%%%%%%%%%%%%%%%%%%%%%%%%%%%%%%%%%%%%%%%%%%%%%%%%%%%%%%%%%%%%
\subsection{The partition monoid $\mathcal{P}_{n}$}\label{sec-Pn}

Let $n \in \N$, $n \geq 2$. We require the following information about the
Green's classes of $\mathcal{P}_{n}$ in the $\J$-class $J_{n - 1}$.
Let $\alpha \in J_{n - 1}$. By definition, $\alpha$ contains $n - 1$ transverse
blocks. Since each transverse block contains at least two points, and there are
only $2n$ points in $\n \cup \np$, there are few possible
combinations of kernel and domain for $\alpha$. In particular, either
$\ker(\alpha)$ is trivial and $\dom(\alpha) = \n \setminus \{i\}$ for some $i
\in \n$, or $\dom(\alpha) = \n$ and $\{i, j\}$ is the unique non-trivial kernel
class of $\alpha$, for some distinct $i, j \in \n$.  By
Lemma~\ref{lem-green-partition}, these properties describe the $\R$-classes of
$J_{n - 1}$.  Since the $\L$-classes and $\R$-classes of a regular
$\ast$-semigroup correspond via the $^{\ast}$ operation, analogous statements
hold for the $\L$-classes of $J_{n - 1}$. Thus, for distinct $i, j \in \n$, we
make the following definitions:
\begin{itemize}
  \item
    $R_{i} = \bigset{\alpha \in J_{n - 1}}{\dom(\alpha) = \n \setminus \{i\}}$,
    an $\R$-class;
  \item
    $R_{\{i, j\}} = \bigset{\alpha \in J_{n - 1}}{(i, j) \in \ker(\alpha)}$, an
    $\R$-class;
  \item
    $L_{i} = R_{i}^{*} = \bigset{\alpha \in J_{n - 1}}{\codom(\alpha) = \n
    \setminus \{i\}}$, an $\L$-class;
  \item
    $L_{\{i, j\}} = R_{\{i, j\}}^{*} = \bigset{\alpha \in J_{n - 1}}{(i, j) \in
    \coker(\alpha)}$, an $\L$-class.
\end{itemize}
An $\H$-class of the form $L_{i} \cap R_{j}$ is a group if and only if $i = j$,
an $\H$-class of the form $L_{i} \cap R_{\{j, k\}}$ or $R_{i} \cap L_{\{j, k\}}$
is a group if and only if $i \in \{j, k\}$, and an $\H$-class of the form
$L_{\{i, j\}} \cap R_{\{k, l\}}$ is a group if and only if $\{i, j\} = \{k,
l\}$.

The main result of this section is the following theorem.

\begin{thm}\label{thm-partition}
  Let $n \in \N$, $n \geq 2$, and let $\mathcal{P}_{n}$ be the partition monoid
  of degree $n$. Then the maximal subsemigroups of $\mathcal{P}_{n}$ are:
  \begin{enumerate}[label=\emph{(\alph*)}]
    \item
      $(\mathcal{P}_{n} \setminus \mathcal{S}_{n}) \cup U$, where $U$ is a
      maximal subgroup of $\mathcal{S}_{n}$ \emph{(type~\ref{item-intersect})};
    \item
      $\mathcal{P}_{n} \setminus \bigset{\alpha \in
      \mathcal{P}_{n}}{\rank(\alpha) = n - 1\ \text{and}\ \ker(\alpha)\ \text{is
      trivial}}$
      \emph{(type~\ref{item-remove-r})};
    \item
      $\mathcal{P}_{n} \setminus \bigset{\alpha \in
      \mathcal{P}_{n}}{\rank(\alpha) = n - 1\ \text{and}\ \dom(\alpha) = \n}$
      \emph{(type~\ref{item-remove-r})};
    \item
      $\mathcal{P}_{n} \setminus \bigset{\alpha \in
      \mathcal{P}_{n}}{\rank(\alpha) = n - 1\ \text{and}\ \coker(\alpha)\
      \text{is trivial}}$
      \emph{(type~\ref{item-remove-l})}; and
    \item
      $\mathcal{P}_{n} \setminus \bigset{\alpha \in
      \mathcal{P}_{n}}{\rank(\alpha) = n - 1\ \text{and}\ \codom(\alpha) = \n}$
      \emph{(type~\ref{item-remove-l})}.
  \end{enumerate}
  In particular, for $n \geq 2$, there are $s_{n} + 4$ maximal subsemigroups of
  $\mathcal{P}_{n}$, where $s_{n}$ denotes the number of
  maximal subgroups of the symmetric group of degree $n$.  The partition monoid
  of degree $1$ is a semilattice of order 2: its maximal subsemigroups are each
  of its singleton subsets.
\end{thm}

\begin{proof}
  By~\cite[Section~6]{East2011ab}, the ideal $\mathcal{P}_{n} \setminus
  \mathcal{S}_{n}$ is generated by its idempotents of rank $n - 1$.  Thus, since
  $\J$-equivalence in $\mathcal{P}_{n}$ is determined by rank, the maximal
  subsemigroups of $\mathcal{P}_{n}$ arise from its group of units,
  $\mathcal{S}_{n}$, and the $\J$-class $J_{n - 1}$.  It follows by
  Corollary~\ref{cor-group-of-units} that the maximal subsemigroups that arise
  from $\mathcal{S}_{n}$ are those stated in the theorem.

  By Lemma~\ref{lem-no-type-intersect}, there are no maximal subsemigroups
  arising from $J_{n - 1}$ of type~\ref{item-intersect}.  It is clear
  from~\eqref{eq-partition-action} that $\mathcal{S}_{n}$ transitively permutes
  the $\R$-classes of $J_{n - 1}$ with trivial kernel, and it transitively
  permutes the $\R$-classes of $J_{n - 1}$ with domain $\n$. Thus there are two
  orbits of $\R$-classes of $J_{n - 1}$ under the action of $\mathcal{S}_{n}$;
  therefore there are two corresponding orbits of $\L$-classes.  A picture
  of $\Delta(\mathcal{P}_{n}, J_{n - 1})$ is shown in Figure~\ref{fig-Pn-delta}.

  Since the two bicomponents are the only maximal independent subsets of
  $\Delta(\mathcal{P}_{n}, J_{n - 1})$, Corollary~\ref{cor-rectangle} implies
  that there are no maximal subsemigroups of type~\ref{item-rectangle}.  Each
  vertex of $\Delta(\mathcal{P}_{n}, J_{n - 1})$ has degree $2$, and so by
  Proposition~\ref{prop-remove-r}, there are two maximal subsemigroups of
  type~\ref{item-remove-r}, formed by removing each orbit of $\R$-classes in
  turn, and by Lemma~\ref{lem-l-or-r}, the corresponding maximal subsemigroups
  of type~\ref{item-remove-l} are those stated in the theorem.  By
  Proposition~\ref{prop-remove-j}, there are no maximal subsemigroups of
  type~\ref{item-remove-j}.
\end{proof}

%%%%%%%%%%%%%%%%%%%%%%%%%%%%%%%%%%%%%%%%%%%%%%%%%%%%%%%%%%%%%%%%%%%%%%%%%%%%%%%%
% Figure of Delta for Pn
%%%%%%%%%%%%%%%%%%%%%%%%%%%%%%%%%%%%%%%%%%%%%%%%%%%%%%%%%%%%%%%%%%%%%%%%%%%%%%%%

\begin{figure}
  \begin{center}
    \begin{tikzpicture}
      % L-classes
      \node[rounded corners,rectangle,draw,fill=blue!20]
        (1)  at (0,  0) {$\bigset{L_{i}}{i \in \n}$};
      \node[rounded corners,rectangle,draw,fill=green!20]
        (2)  at (5,  0) {$\bigset{L_{\{i, j\}}}{ i, j \in \n,\ i \neq j}$};

      % R-classes
      \node[rounded corners,rectangle,draw,fill=blue!20]
        (11) at (0,  4) {$\bigset{R_{i}}{i \in \n}$};
      \node[rounded corners,rectangle,draw,fill=green!20]
        (12) at (5,  4) {$\bigset{R_{\{i, j\}}}{ i, j \in \n,\ i \neq j}$};

      % \dots

      \edge{1}{11};
      \edge{2}{12};
      \edge{1}{12};
      \edge{2}{11};
    \end{tikzpicture}
  \end{center}
  \caption{The graph $\Delta(\mathcal{P}_{n}, J_{n - 1})$.}\label{fig-Pn-delta}
\end{figure}
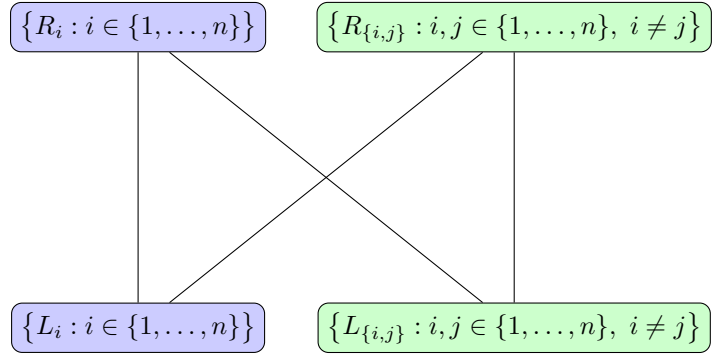

%%%%%%%%%%%%%%%%%%%%%%%%%%%%%%%%%%%%%%%%%%%%%%%%%%%%%%%%%%%%%%%%%%%%%%%%%%%%%%%%
\subsection{The partial Brauer monoid $\mathcal{PB}_{n}$}\label{sec-PBn}

The symmetric inverse monoid degree $n$ embeds in the partition monoid
$\mathcal{P}_{n}$,
via the injective homomorphism $\phi$ that maps a partial
permutation $\alpha$ to the partition whose non-trivial blocks are $\{i,
(i\alpha)'\}$ for each $i \in \dom(\alpha)$.  Thus we define
$\mathcal{I}_{n}$ to be copy of the symmetric inverse monoid of degree $n$
embedded in $\mathcal{P}_{n}$ by $\phi$. Clearly $\mathcal{I}_{n}$ is a
submonoid of $\mathcal{PB}_{n}$.

To describe the maximal subsemigroups of $\mathcal{PB}_{n}$, we require a
description of the elements of $\mathcal{PB}_{n}$ whose rank is at least $n -
2$.  Partitions of degree $n$ that have rank $n$ are units, and the group of
units of $\mathcal{PB}_{n}$ is $\mathcal{S}_{n}$.

Let $\alpha \in J_{n - 1} \cap \mathcal{PB}_{n}$.  By definition, $\alpha$
contains precisely $n - 1$ transverse blocks of size two, and two
singleton blocks $\{i\}$ and $\{j'\}$, for some $i, j \in \n$. Therefore
$\alpha$ is the image of some partial permutation under the embedding $\phi$.
Since $\alpha$ was arbitrary, and $\mathcal{I}_{n} \subseteq \mathcal{PB}_{n}$,
it follows that $J_{n - 1} \cap \mathcal{PB}_{n} = J_{n - 1} \cap
\mathcal{I}_{n}$.

Let $\alpha \in \mathcal{PB}_{n}$, and suppose that $\rank(\alpha) = n - 2$.
Then $\alpha$ contains $n - 2$ transverse blocks, which leaves a pair of points
of $\n$ and a pair of points of $\np$ that are not contained in transverse
blocks.  Each of these pairs forms either a block of size $2$, or two singleton
blocks. In particular, $\dom(\alpha)$ lacks some two points $i$ and $j$, and
either $\ker(\alpha)$ is trivial, or $\{i, j\}$ is the unique non-trivial kernel
class of $\alpha$.  A similar statement holds for the codomain and cokernel of
$\alpha$.

\begin{thm}\label{thm-partial-brauer}
  Let $n \in \N$, $n \geq 2$, be arbitrary and let $\mathcal{PB}_{n}$ be the
  partial Brauer monoid of degree $n$. Then the maximal subsemigroups of
  $\mathcal{PB}_{n}$ are:
  \begin{enumerate}[label=\emph{(\alph*)}]
    \item
      $(\mathcal{PB}_{n} \setminus \mathcal{S}_{n}) \cup U$, where $U$ is a
      maximal subgroup of $\mathcal{S}_{n}$
      \emph{(type~\ref{item-intersect})};
    \item
      $\mathcal{PB}_{n} \setminus \set{\alpha \in
      \mathcal{PB}_{n}}{\rank(\alpha) = n - 1}$
      \emph{(type~\ref{item-remove-j})};
    \item
      $\mathcal{PB}_{n} \setminus \set{\alpha \in
      \mathcal{PB}_{n}}{\rank(\alpha) = n - 2\ \text{and}\ \ker(\alpha)\
      \text{is non-trivial}}$
      \emph{(type~\ref{item-remove-r})}; and
    \item
      $\mathcal{PB}_{n} \setminus \set{\alpha \in
      \mathcal{PB}_{n}}{\rank(\alpha) = n - 2\ \text{and}\ \coker(\alpha)\
      \text{is non-trivial}}$
      \emph{(type~\ref{item-remove-l})}.
  \end{enumerate}
  In particular, for $n \geq 2$, there are $s_{n} + 3$ maximal subsemigroups of
  $\mathcal{PB}_{n}$, where $s_{n}$ denotes the number of
  maximal subgroups of the symmetric group of degree $n$. The partial Brauer
  monoid $\mathcal{PB}_{1} = \mathcal{P}_{1}$ is a semilattice of order $2$: its
  maximal subsemigroups are each of its singleton subsets.
\end{thm}

\begin{proof}
  By~\cite{Dolinka2017251}, $\mathcal{PB}_{n}$ is
  generated by its elements with rank at least $n - 2$, and any generating set
  contains elements of ranks $n$, $n - 1$, and $n - 2$. By
  Lemma~\ref{lem-green-partition}, Green's $\J$-relation in $\mathcal{PB}_{n}$
  is determined by rank. Thus, the $\J$-classes of $\mathcal{PB}_{n}$ from which
  there arise maximal subsemigroups are its group of units,
  $J_{n - 1} \cap \mathcal{PB}_{n}$, and $J_{n - 2} \cap
  \mathcal{PB}_{n}$. The group of units of $\mathcal{PB}_{n}$ is
  $\mathcal{S}_{n}$, and so by Corollary~\ref{cor-group-of-units}, the maximal
  subsemigroups that arise from the group of units are those stated.

  Since the $\J$-class $J_{n - 1} \cap \mathcal{PB}_{n}$ is covered by the group
  of units, $\mathcal{PB}_{n} \setminus J_{n - 1}$ is a subsemigroup of
  $\mathcal{PB}_{n}$.  Let $\alpha \in J_{n - 1} \cap \mathcal{PB}_{n} = J_{n -
  1} \cap \mathcal{I}_{n}$ be arbitrary. As stated in the proof of
  Theorem~\ref{thm-maximals-PTn-Tn-In}, $\mathcal{I}_{n}$ is generated by its
  group of units $\mathcal{S}_{n}$ along with any element of rank $n - 1$. Thus
  $\genset{\mathcal{PB}_{n} \setminus J_{n - 1},\ \alpha} \supseteq
  \genset{\mathcal{PB}_{n} \setminus J_{n - 1},\ \mathcal{I}_{n}} =
  \mathcal{PB}_{n}$.  By using Lemma~\ref{lem-Xi} with $k = 1$ and $X_{1} = J_{n
  - 1} \cap \mathcal{PB}_{n}$, we find that the unique maximal subsemigroup of
  $\mathcal{PB}_{n}$ arising from this $\J$-class has type~\ref{item-remove-j}.

  In order to determine the maximal subsemigroups of $\mathcal{PB}_{n}$ that
  arise from its $\J$-class of rank $n - 2$, we define the subsets
  \begin{align*}
    X & = \set{\alpha \in \mathcal{PB}_{n}}{\rank(\alpha) = n - 2\
    \text{and}\ \ker(\alpha)\ \text{is non-trivial}},\ \text{and}\\
    X^{*} = \set{\alpha^{*}}{\alpha \in X} & = \set{\alpha \in
    \mathcal{PB}_{n}}{\rank(\alpha) = n - 2 \ \text{and}\ \coker(\alpha)\
    \text{is non-trivial}}.
  \end{align*}
  Note that $X$ is a union of $\R$-classes of $\mathcal{PB}_{n}$, and
  $X^{*}$ is a union of $\L$-classes.  Let $A$ be a subset of $J_{n - 2}
  \cap \mathcal{PB}_{n}$ such that $(\mathcal{PB}_{n} \setminus J_{n - 2}) \cup
  A$ generates $\mathcal{PB}_{n}$.  Let $\alpha \in X$ be arbitrary.  Then
  $\alpha$ can the written as a product $\alpha = \beta_{1} \cdots \beta_{k}$ of
  some of these generators.  Clearly the generators $\beta_{1}, \ldots,
  \beta_{k}$ have rank at least $n - 2$. Every element in $\mathcal{PB}_{n}$ of
  rank $n$ and $n - 1$ has a trivial kernel, and the subset of partitions with
  trivial kernel in $\mathcal{P}_{n}$ forms a subsemigroup. Thus there exists
  some $r \in \{1, \ldots, k\}$ such that $\rank(\beta_{r}) = n - 2$ and
  $\ker(\beta_{r})$ is non-trivial --- in other words, $\beta_{r} \in X$.  A
  dual argument shows that $A \cap X^{*} \neq \varnothing$. Conversely, for
  any subset $A$ of $J_{n - 2} \cap \mathcal{PB}_{n}$ that intersects $X$
  and $X^{*}$ non-trivially, we have $\mathcal{PB}_{n} =
  \genset{\mathcal{PB}_{n} \setminus J_{n - 2},\ A}$.  By
  Lemma~\ref{lem-Xi}, the maximal subsemigroups of $\mathcal{PB}_{n}$ arising
  from $J_{n - 2} \cap \mathcal{PB}_{n}$ are $\mathcal{PB}_{n} \setminus X$
  and $\mathcal{PB}_{n} \setminus X^{*}$; these maximal subsemigroups have
  types~\ref{item-remove-r} and~\ref{item-remove-l}, respectively.
\end{proof}

%%%%%%%%%%%%%%%%%%%%%%%%%%%%%%%%%%%%%%%%%%%%%%%%%%%%%%%%%%%%%%%%%%%%%%%%%%%%%%%%
\subsection{The Brauer monoid $\mathcal{B}_{n}$ and the uniform block bijection
monoid $\mathfrak{F}_{n}$}\label{sec-Bn-Fn}

The main results of this section are the following theorems, which describe the
maximal subsemigroup of $\mathcal{B}_{n}$ and $\mathfrak{F}_{n}$.

\begin{thm}\label{thm-Brauer}
  Let $n \in \N$, $n \geq 2$, and let $\mathcal{B}_{n}$ be the Brauer monoid of
  degree $n$. Then the maximal subsemigroups of $\mathcal{B}_{n}$ are:
  \begin{enumerate}[label=\emph{(\alph*)}]
    \item
      $(\mathcal{B}_{n} \setminus \mathcal{S}_{n}) \cup U$, where $U$ is a
      maximal subgroup of $\mathcal{S}_{n}$ \emph{(type~\ref{item-intersect})};
      and
    \item
      $\mathcal{B}_{n} \setminus \set{\alpha \in \mathcal{B}_{n}}{\rank(\alpha)
      = n - 2}$ \emph{(type~\ref{item-remove-j})}.
  \end{enumerate}
  In particular, for $n \geq 2$, there are $s_{n} + 1$ maximal subsemigroups of
  $\mathcal{B}_{n}$, where $s_{n}$ denotes the number of
  maximal subgroups of the symmetric group of degree $n$.
\end{thm}

\begin{thm}\label{thm-factorisable}
  Let $n \in \N$, $n \geq 2$, and let $\mathfrak{F}_{n}$ be the uniform block
  bijection monoid of degree $n$. Then the maximal subsemigroups of
  $\mathfrak{F}_{n}$ are:
  \begin{enumerate}[label=\emph{(\alph*)}]
    \item
      $(\mathfrak{F}_{n} \setminus \mathcal{S}_{n}) \cup U$, where $U$ is a
      maximal subgroup of $\mathcal{S}_{n}$ \emph{(type~\ref{item-intersect})};
      and
    \item
      $\mathfrak{F}_{n} \setminus \set{\alpha \in
      \mathfrak{F}_{n}}{\rank(\alpha) = n - 1}$
      \emph{(type~\ref{item-remove-j})}.
  \end{enumerate}
  In particular, for $n \geq 2$, there are $s_{n} + 1$ maximal subsemigroups of
  $\mathfrak{F}_{n}$, where $s_{n}$ is the number of maximal subgroups of the
  symmetric group of degree $n$.
\end{thm}

Let $n \in \N$, $n \geq 2$. By~\cite{auinger2012krohn}, $\mathcal{B}_{n}$ is
generated by $\mathcal{S}_{n}$ and any projection of rank $n - 2$, and
by~\cite[Section~5]{maltcev2007}, $\mathfrak{F}_{n}$ is generated by
$\mathcal{S}_{n}$ and any projection of rank $n - 1$. These facts are used in
the following proof.\newline

\proofrefs{thm-Brauer}{thm-factorisable}
  The group of units of $\mathcal{B}_{n}$ and $\mathfrak{F}_{n}$ is
  $\mathcal{S}_{n}$.  By Corollary~\ref{cor-group-of-units}, the maximal
  subsemigroups that arise from the group of units in each case are those
  described in the theorems.

  Let $\alpha \in J_{n - 2} \cap \mathcal{B}_{n}$. The non-transverse blocks of
  $\alpha$ are $\{i, j\}$ and $\{k',l'\}$ for some $i, j, k, l \in \n$ with $i
  \neq j$ and $k \neq l$.  Let $\tau \in \mathcal{S}_{n}$ be a permutation that
  contains the blocks $\{k, i'\}$ and $\{l, j'\}$. Therefore the non-transverse
  blocks of $\alpha \tau$ are $\{i, j\}$ and $\{i', j'\}$, and so $(\alpha
  \tau)^{m}$ is a projection of rank $n - 2$ for some $m \in \N$. Thus
  $\genset{\mathcal{S}_{n},\ \alpha} \supseteq \genset{\mathcal{S}_{n},\ (\alpha
  \tau)^{m}} = \mathcal{B}_{n}$, and so $\mathcal{B}_{n} =
  \genset{\mathcal{S}_{n},\ \alpha}$.  By a similar argument, $\mathfrak{F}_{n}
  = \genset{\mathcal{S}_{n},\ \beta}$ for any uniform block bijection of rank $n
  - 1$. By Corollary~\ref{cor-Xi}, the remaining maximal subsemigroups are those
  stated in the theorems.
\qed{}

%%%%%%%%%%%%%%%%%%%%%%%%%%%%%%%%%%%%%%%%%%%%%%%%%%%%%%%%%%%%%%%%%%%%%%%%%%%%%%%%
\subsection{The dual symmetric inverse monoid
$\mathcal{I}_{n}^{*}$}\label{sec-dual}

The maximal subsemigroups of the dual symmetric inverse monoid were
first described in~\cite[Theorem~19]{maltcev2007}. We reprove this result in the
following theorem.

\begin{thm}\label{thm-dual-symmetric}
  Let $n \in \N$, $n \geq 3$, and let $\mathcal{I}_{n}^{*}$ be the dual
  symmetric inverse monoid of degree $n$. Then the maximal subsemigroups of
  $\mathcal{I}_{n}^{*}$ are:
  \begin{enumerate}[label=\emph{(\alph*)}]
    \item
      $(\mathcal{I}_{n}^{*} \setminus \mathcal{S}_{n}) \cup U$, where $U$ is a
      maximal subgroup of $\mathcal{S}_{n}$ \emph{(type~\ref{item-intersect})};
      and
    \item
      $\mathcal{I}_{n}^{*} \setminus \set{\alpha \in
      \mathcal{I}_{n}^{*}}{\rank(\alpha) = n - 1\ \text{and}\ \alpha\ 
      \text{is not uniform}}$
      \emph{(type~\ref{item-intersect})}.
  \end{enumerate}
  In particular, for $n \geq 3$ there are $s_{n} + 1$ maximal subsemigroups of
  $\mathcal{I}_{n}^{*}$, where $s_{n}$ is the number of maximal subgroups of the
  symmetric group of degree $n$.  For $n \in \{1, 2\}$, $\mathcal{I}_{n}^{*} =
  \mathfrak{F}_{n}$; see Theorem~\ref{thm-factorisable}.
\end{thm}

\begin{proof}
  The group of units of $\mathcal{I}_{n}^{*}$ is $\mathcal{S}_{n}$.  By
  Corollary~\ref{cor-group-of-units}, the maximal subsemigroups arising from the
  group of units are those described. By~\cite[Proposition~16]{maltcev2007},
  $\mathcal{I}_{n}^{*} = \genset{\mathcal{S}_{n},\ \alpha}$ if and only if
  $\alpha \in J_{n - 1} \cap (\mathcal{I}_{n}^{*} \setminus \mathfrak{F}_{n})$.
  Using Corollary~\ref{cor-Xi} with $X = J_{n - 1} \cap (\mathcal{I}_{n}^{*}
  \setminus \mathfrak{F}_{n})$, the result follows.
\end{proof}

The maximal subsemigroup of $\mathcal{I}_{n}^{*}$ that arises from its
$\J$-class $J_{n - 1} \cap \mathcal{I}_{n}^{*}$ can also be found by using
Proposition~\ref{prop-regularstar-intersect}, since $\mathcal{S}_{n}$ acts
transitively on the $\R$-classes of $J_{n - 1} \cap \mathcal{I}_{n}^{*}$,
and each of the idempotents in this $\J$-class is a projection.

%%%%%%%%%%%%%%%%%%%%%%%%%%%%%%%%%%%%%%%%%%%%%%%%%%%%%%%%%%%%%%%%%%%%%%%%%%%%%%%%
\subsection{The Jones monoid $\mathcal{J}_{n}$ and the annular Jones monoid
$\mathcal{AJ}_{n}$}\label{sec-jones}

Let $n \in \N$.  In this section, we determine the maximal subsemigroups of the
Jones monoid $\mathcal{J}_{n}$ (also known as the Temperley-Lieb monoid), and
the annular Jones monoid $\mathcal{AJ}_{n}$.
Since the planar partition monoid of degree $n$ is isomorphic to the Jones
monoid of degree $2n$~\cite{Halverson2005869}, by determining the maximal
subsemigroups of $\mathcal{J}_n$ we obtain those of $\mathcal{PP}_n$.

Suppose that $n \geq 2$.
By~\cite{Borisavljevic2002aa}, $\mathcal{J}_{n}$ is generated by the
identity partition $\id_{n}$ and its projections of rank $n - 2$. By
Lemma~\ref{lem-green-partition}, the set $J_{n - 2} \cap \mathcal{J}_{n}$ is a
$\J$-class of $\mathcal{J}_{n}$, and since there are no elements of
$\mathcal{J}_{n}$ with rank $n - 1$, it follows that this $\J$-class is covered
by the group of units $\{\id_{n}\}$. Note that $\mathcal{J}_{n}$ is
$\H$-trivial, since it consists of planar partitions.

To describe the maximal subsemigroups of $\mathcal{J}_{n}$ that arise from its
$\J$-class of rank $n - 2$ partitions, we require the graph
$\Delta(\mathcal{J}_{n}, J_{n - 2}\cap\mathcal{J}_{n})$, which we hereafter
refer to as $\Delta(\mathcal{J}_{n})$. Thus we require a description of the
Green's classes of $J_{n - 2} \cap \mathcal{J}_{n}$.  Let $\alpha \in J_{n - 2}
\cap \mathcal{J}_{n}$.  Then $\alpha$ has $n - 2$ transverse blocks, and these
contain two points. By planarity, the remaining blocks are of the form $\{i, i +
1\}$ and $\{j', (j + 1)'\}$ for some $i, j \in \{1, \ldots, n - 1\}$.  By
Lemma~\ref{lem-green-partition}, the $\J$-class $J_{n - 2} \cap \mathcal{J}_{n}$
contains $n - 1$ $\R$-classes and $n - 1$ $\L$-classes. For $i \in \{1, \ldots,
n - 1\}$, we define
\begin{itemize}
  \item
    $R_{i} = \bigset{\alpha \in \mathcal{J}_{n}}{\rank(\alpha) = n - 2\
    \text{and}\ \{i, i+1 \}\ \text{is a block of}\ \alpha}$, an $\R$-class; and
  \item
    $L_{i} = \bigset{\alpha \in \mathcal{J}_{n}}{\rank(\alpha) = n - 2\
    \text{and}\ \{i', (i+1)' \}\ \text{is a block of}\ \alpha}$, an $\L$-class.
\end{itemize}
The intersection of the $\L$-class $L_{i}$ and the $\R$-class $R_{j}$ is a group
if and only if $|i - j| \leq 1$.  Since the group of units of $\mathcal{J}_{n}$
is trivial, its action on the $\L$-classes and $\R$-classes of $J_{n - 2} \cap
\mathcal{J}_{n}$ is trivial. A picture of $\Delta(\mathcal{J}_{n})$ is shown in
Figure~\ref{fig-Jn-delta}. 

%%%%%%%%%%%%%%%%%%%%%%%%%%%%%%%%%%%%%%%%%%%%%%%%%%%%%%%%%%%%%%%%%%%%%%%%%%%%%%%%
% Figure of Delta for Jn
%%%%%%%%%%%%%%%%%%%%%%%%%%%%%%%%%%%%%%%%%%%%%%%%%%%%%%%%%%%%%%%%%%%%%%%%%%%%%%%%

\begin{figure}
  \begin{center}
    \begin{tikzpicture}
      % L-classes
      \node[rounded corners,rectangle,draw,fill=blue!20]
        (1)  at (0,  0) {$\{ L_{1} \}$};
      \node[rounded corners,rectangle,draw,fill=blue!20]
        (2)  at (2,  0) {$\{ L_{2} \}$};
      \node[rounded corners,rectangle,draw,fill=blue!20]
        (3)  at (4,  0) {$\{ L_{3} \}$};
      \node[rounded corners,rectangle,draw,fill=blue!20]
        (4)  at (6,  0) {$\{ L_{4} \}$};
      \node[rounded corners,rectangle,draw,fill=blue!20]
        (5)  at (11, 0) {$\{ L_{n - 2} \}$};
      \node[rounded corners,rectangle,draw,fill=blue!20]
        (6)  at (13, 0) {$\{ L_{n - 1} \}$};

      % R-classes
      \node[rounded corners,rectangle,draw,fill=blue!20]
        (11) at (0,  3.2) {$\{ R_{1} \}$};
      \node[rounded corners,rectangle,draw,fill=blue!20]
        (12) at (2,  3.2) {$\{ R_{2} \}$};
      \node[rounded corners,rectangle,draw,fill=blue!20]
        (13) at (4,  3.2) {$\{ R_{3} \}$};
      \node[rounded corners,rectangle,draw,fill=blue!20]
        (14) at (6,  3.2) {$\{ R_{4} \}$};
      \node[rounded corners,rectangle,draw,fill=blue!20]
        (15) at (11, 3.2) {$\{ R_{n - 2} \}$};
      \node[rounded corners,rectangle,draw,fill=blue!20]
        (16) at (13, 3.2) {$\{ R_{n - 1} \}$};

      % \dots

      \node (99) at (8.5, 1.5) {$\cdots$};

      % blank spots
      \node (21) at (7.5, 0.8) {};
      \node (22) at (7.5, 2.4) {};

      \node (31) at (9.5, 0.8) {};
      \node (32) at (9.5, 2.4) {};

      \node (40) at (7,  1.6) {};
      \node (50) at (10, 1.6) {};

      % straight edges
      \edge{1}{11};
      \edge{2}{12};
      \edge{3}{13};
      \edge{4}{14};
      \edge{5}{15};
      \edge{6}{16};

      % top left -> bottom right diagonal edges
      \edge{1}{12}
      \edge{2}{13}
      \edge{3}{14}
      \edge{5}{16}

      % top right -> bottom left diagonal edges
      \edge{2}{11}
      \edge{3}{12}
      \edge{4}{13}
      \edge{6}{15}

      % fading edges
      \draw[dashed] (4)  -- (22);
      \edge{4}{40};
      \draw[dashed] (14) -- (21);
      \edge{14}{40};
      \draw[dashed] (5)  -- (32);
      \edge{5}{50};
      \draw[dashed] (15) -- (31);
      \edge{15}{50};

    \end{tikzpicture}
  \end{center}
  \caption{The graph $\Delta(\mathcal{J}_{n}, \mathcal{J}_{n} \cap J_{n -
  2})$.}\label{fig-Jn-delta}
\end{figure}
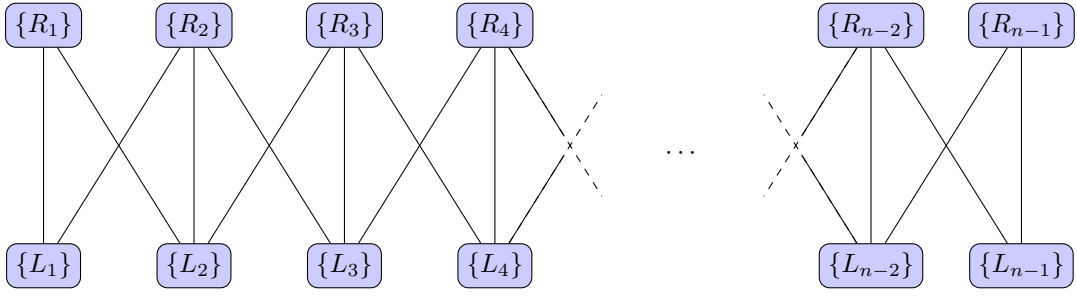

%%%%%%%%%%%%%%%%%%%%%%%%%%%%%%%%%%%%%%%%%%%%%%%%%%%%%%%%%%%%%%%%%%%%%%%%%%%%%%%%

The maximal independent subsets of $\Delta(\mathcal{J}_{n})$ are counted in the
following lemma.

\begin{lem}\label{lem-Jn-delta-number}
  Let $n \geq 2$.  The number of maximal independent subsets of
  $\Delta(\mathcal{J}_{n})$ is $2F_{n - 1}$, where $F_{n - 1}$ is the $(n -
  1)^{\text{th}}$ term of the Fibonacci sequence, defined by $F_{1} = F_{2} = 1$
  and $F_{k} = F_{k - 1} + F_{k - 2}$ for $k \geq 3$.
\end{lem}

\begin{proof}
  The result may be verified directly for $n \in \{2, 3\}$.
  Suppose that $n \geq 4$, and let $K$ be a maximal independent subset of
  $\Delta(\mathcal{J}_{n})$. We first show that precisely one of $\{L_{n - 1}\}$
  and $\{R_{n - 1}\}$ is contained in $K$.  Since $\{L_{n - 1}\}$ and $\{R_{n -
  1}\}$ are adjacent in $\Delta(\mathcal{J}_{n})$, they are not both contained
  in $K$. Similarly, at least one of $\{L_{n - 2}\}$ and $\{R_{n - 2}\}$ is not
  contained in $K$. If $\{L_{n - 2}\} \not\in K$, then either $\{L_{n - 1}\} \in
  K$, or the maximality of $K$ implies that $\{R_{n - 1}\} \in K$.  If instead
  $\{R_{n - 2}\} \not\in K$, then it follows similarly that either $\{L_{n -
  1}\} \in K$ or $\{R_{n - 1}\} \in K$.  Therefore we shall count $a(n)$,
  the number of maximal independent subsets of $\Delta(\mathcal{J}_{n})$ that
  contain $\{L_{n - 1}\}$. By symmetry, the total number of maximal independent
  subsets is $2a(n)$.
  
  For $i \in \{1, 2\}$, define $\Lambda_{n - i}$ be the induced subgraph of
  $\Delta(\mathcal{J}_{n})$ on the vertices $\big\{ \{L_{1}\}, \ldots, \{L_{n -
  1 - i}\}, \{R_{1}\}, \ldots, \{R_{n - 1 - i}\} \big\}$. Clearly
  $\Lambda_{n - i}$ is isomorphic to $\Delta(\mathcal{J}_{n - i})$, and so the
  number of maximal independent subsets of $\Lambda_{n - 1}$ that contain
  $\{L_{n - 2}\}$ is $a(n - 1)$, and the number of maximal independent subsets
  of $\Lambda_{n - 2}$ that contain $\{R_{n - 3}\}$ is $a(n - 2)$.

  Let $K$ be  a maximal independent subset of $\Delta(\mathcal{J}_{n})$
  containing $\{L_{n - 1}\}$.  Certainly $K$ contains neither $\{R_{n - 1}\}$
  nor $\{R_{n - 2}\}$, since these vertices are adjacent to $\{L_{n - 1}\}$ in
  $\Delta(\mathcal{J}_{n})$.  However, $K$ contains precisely one of $\{L_{n -
  2}\}$ or $\{R_{n - 3}\}$: it does not contain both, since $\{L_{n - 2}\}$ and
  $\{R_{n - 3}\}$ are adjacent, but if $\{R_{n - 3}\} \not\in K$, then $K \cup
  \big\{\{L_{n - 2}\}\big\}$ is a maximal independent subset containing $K$, and
  so $\{L_{n - 2}\} \in K$.

  If $K$ contains $\{L_{n - 2}\}$, then $K \setminus \big\{\{L_{n - 1}\}\big\}$
  is a maximal independent subset of $\Lambda_{n - 1}$ that contains $\{L_{n -
  2}\}$, while if $K$ contains $\{R_{n - 3}\}$, then $K \setminus \big\{\{L_{n -
  1}\}\big\}$ is a maximal independent subset of $\Lambda_{n - 2}$ that contains
  $\{R_{n - 3}\}$.  Conversely, maximal independent subsets of $\Lambda_{n - 1}$
  containing $\{L_{n - 2}\}$, and maximal independent subsets of $\Lambda_{n -
  2}$ containing $\{R_{n - 3}\}$, give rise to distinct maximal independent
  subsets of $\Delta(\mathcal{J}_{n})$ that contain $\{L_{n - 1}\}$, via the
  addition of $\{L_{n - 1}\}$.  It follows that $a(n) = a(n - 1) + a(n - 2)$.
  By this recurrence, and since $a(2) = F_{1}$ and $a(3) = F_{2}$, it follows
  that $a(n) = F_{n - 1}$.
\end{proof}

Whilst the maximal independent subsets of $\Delta(\mathcal{J}_{n})$ may be
readily counted, it is more difficult to describe them. Let $n \in \N$, $n \geq
3$ be arbitrary, and let $K$ be a maximal independent subset of
$\Delta(\mathcal{J}_{n})$. As described in the proof of
Lemma~\ref{lem-Jn-delta-number}, $K$ contains precisely one of $\{L_{n - 1}\}$
and $\{R_{n - 1}\}$, and a similar argument shows that $K$ contains precisely
one of $\{L_{1}\}$ and $\{R_{1}\}$. Let $i \in \{1, \ldots, n - 2\}$. If
$\{L_{i}\} \in K$, then either $\{L_{i + 1}\} \in K$ or $\{R_{i + 2}\} \in K$,
while if $\{R_{i}\} \in K$, then either $\{R_{i + 1}\} \in K$, or $\{L_{i + 2}\}
\in K$.  In other words, $K$ contains either $\{L_{1}\}$ or $\{R_{1}\}$, and
each subsequent vertex contains either the same type of Green's class of index
one higher, or it contains the other type of Green's class of index two higher;
the final vertex is either $\{L_{n - 1}\}$ or $\{R_{n - 1}\}$. Conversely, any
subset of vertices of $\Delta(\mathcal{J}_{n})$ satisfying these requirements is
a maximal independent subset of $\Delta(\mathcal{J}_{n})$.

We may now describe the maximal subsemigroups of $\mathcal{J}_{n}$.

\begin{thm}\label{thm-jones}
  Let $n \in \N$, $n \geq 3$, and let $\mathcal{J}_{n}$ be the Jones monoid of
  degree $n$. Then the maximal subsemigroups of $\mathcal{J}_{n}$ are:
  \begin{enumerate}[label=\emph{(\alph*)}]
    \item
      $\mathcal{J}_{n} \setminus \{ \id_{n} \}$
      \emph{(type~\ref{item-remove-j})};
    \item
      The union of $\mathcal{J}_{n} \setminus J_{n - 2}$ and the union of the
      Green's classes contained in a maximal independent subset of
      $\Delta(\mathcal{J}_{n})$ that is not a bicomponent of
      $\Delta(\mathcal{J}_{n})$
      \emph{(type~\ref{item-rectangle})};
    \item
      $\mathcal{J}_{n} \setminus L$, where $L$ is any $\L$-class in
      $\mathcal{J}_{n}$ of rank $n - 2$
      \emph{(type~\ref{item-remove-l})}; and
    \item
      $\mathcal{J}_{n} \setminus R$, where $R$ is any $\R$-class in
      $\mathcal{J}_{n}$ of rank $n - 2$
      \emph{(type~\ref{item-remove-r})}.
  \end{enumerate}
  In particular, for $n \geq 3$, there are $2F_{n - 1} + 2n - 3$ maximal
  subsemigroups of $\mathcal{J}_{n}$, where $F_{n - 1}$ is the $(n -
  1)$\textsuperscript{th} term of the Fibonacci sequence, defined by $F_{1} =
  F_{2} = 1$ and $F_{k} = F_{k - 1} + F_{k - 2}$ for $k \geq 3$.  The Jones
  monoid $\mathcal{J}_{2}$ is a semilattice of order $2$: its maximal
  subsemigroups are each of its singleton subsets.
\end{thm}

\begin{proof}
  Since the Jones monoid $\mathcal{J}_{n}$ is $\H$-trivial, by
  Corollary~\ref{cor-group-of-units} the unique maximal subsemigroup arising
  from the group of units is formed by removing the identity of
  $\mathcal{J}_{n}$, and by Lemma~\ref{lem-no-type-intersect}, there are no
  maximal subsemigroups of type~\ref{item-intersect}.  By
  Lemma~\ref{lem-Jn-delta-number} and Corollary~\ref{cor-rectangle}, there are
  $2 F_{n - 1} - 2$ maximal subsemigroups of type~\ref{item-rectangle} arising
  from the $\J$-class of rank $n - 2$ in $\mathcal{J}_{n}$; their description is
  a restatement of Proposition~\ref{prop-rectangle}.  Since each vertex of
  $\Delta(\mathcal{J}_{n})$ has degree at least $2$, it follows by
  Proposition~\ref{prop-remove-l} that any $\L$-class of rank $n - 2$ can be
  removed to form a maximal subsemigroup of type~\ref{item-remove-l}, and
  similarly any $\R$-class of rank $n - 2$ can be removed to form a maximal
  subsemigroup of type~\ref{item-remove-r}.  Thus there are $n - 1$ maximal
  subsemigroups of each of these types.  By Proposition~\ref{prop-remove-j},
  there is no maximal subsemigroup of type~\ref{item-remove-j}.
\end{proof}

It remains to describe the maximal subsemigroups of $\mathcal{AJ}_{n}$, the
annular Jones monoid of degree $n$, which was defined in
Section~\ref{sec-diagram-definitions}.  Recall that $\rho_{n} \in
\mathcal{AJ}_{n}$ is the partition of degree $n$ with blocks $\{n, 1'\}$ and
$\{i, (i + 1)'\}$ for $i \in \{1, \ldots, n - 1\}$, and that a partition is
annular if it equals $\rho_{n}^{i} \beta \rho_{n}^{j}$ for some planar partition
$\beta$ and indices $i, j \in \mathbb{Z}$; see
Section~\ref{sec-diagram-definitions}.  To determine the maximal subsemigroups
of $\mathcal{AJ}_{n}$, we require a small generating set for $\mathcal{AJ}_{n}$.

Let $\alpha = \rho_{n}^{i} \beta \rho_{n}^{j} \in \mathcal{AJ}_{n}$ be
arbitrary, where $\beta$ is planar and $i, j \in \Z$.
It follows that $\beta = \rho_{n}^{-i} \alpha \rho_{n}^{-j}
\in \mathcal{AJ}_{n}$, and so $\beta$ is a planar partition whose blocks have
size two. Therefore $\beta \in \mathcal{J}_{n}$, and
$$\mathcal{AJ}_{n} = \bigset{ \rho_{n}^{i} \beta \rho_{n}^{j} }
                            { i, j \in \{ 1, \ldots, n \},\
                              \beta \in \mathcal{J}_{n} }
                   = \genset{\mathcal{J}_{n},\ \rho_{n}}.$$
Given projections $\xi, \zeta \in J_{n - 2} \cap \mathcal{AJ}_{n}$, there exists
some $i \in \n$ such that $\xi = \rho_{n}^{-i} \zeta \rho_{n}^{i}$.  Since
$\mathcal{J}_{n}$ is generated by the identity partition $\id_{n}$ and its
projections of rank $n - 2$~\cite{Borisavljevic2002aa}, it follows that
$\mathcal{AJ}_{n} = \genset{\rho_{n},\ \xi}$, where $\xi$ is an arbitrary
projection in $\mathcal{AJ}_{n}$ of rank $n - 2$.  We proceed by using the same
technique as was used in the proof of Theorems~\ref{thm-Brauer}
and~\ref{thm-factorisable}. Let $\alpha \in J_{n - 2} \cap \mathcal{AJ}_{n}$.
There exists some index $i \in \n$ such that the non-transverse blocks of
$\alpha \rho_{n}^{i}$ are either $\{k, k + 1\}$ and $\{k', (k + 1)'\}$ for some
$k \in \{1, \ldots, n - 1\}$, or are $\{1, n\}$ and $\{1', n'\}$. Thus $(\alpha
\rho_{n}^{i})^{m}$ is a projection for some $m \in \N$. It follows that
$\genset{\rho_{n},\ \alpha} \supseteq \genset{\rho_{n},\ (\alpha
\rho_{n}^{i})^{m}} = \mathcal{AJ}_{n}$.

We may now state and prove the following theorem.

\begin{thm}\label{thm-annular-jones}
  Let $n \in \N$, $n \geq 2$, let $\mathcal{AJ}_{n}$ be the annular Jones monoid
  of degree $n$, let $\rho_{n}$ be the partition of degree $n$ with blocks
  $\{n, 1'\}$ and $\{i, (i+1)'\}$ for $i \in \{1,\ldots,n-1\}$, and let
  $\mathbb{P}_{n}$ be the set of
  primes that divide $n$. Then the maximal
  subsemigroups of $\mathcal{AJ}_{n}$ are:
  \begin{enumerate}[label=\emph{(\alph*)}]
    \item
      $(\mathcal{AJ}_{n} \setminus \genset{\rho_{n}}) \cup
      \genset{\rho_{n}^{d}}$, where $d \in \mathbb{P}_{n}$
      \emph{(type~\ref{item-intersect})}; and
    \item
      $\mathcal{AJ}_{n} \setminus \set{\alpha \in
      \mathcal{AJ}_{n}}{\rank(\alpha) = n - 2}$
      \emph{(type~\ref{item-remove-j})}.
  \end{enumerate}
  In particular, for $n \geq 2$ there are $|\mathbb{P}_{n}| + 1$ maximal
  subsemigroups of $\mathcal{AJ}_{n}$.
\end{thm}

\begin{proof}
  Since the group of units of $\mathcal{AJ}_{n}$ is $\genset{\rho_{n}}$, a
  cyclic group of order $n$, it follows by Lemma~\ref{lem-maximals-cyclic} and
  Corollary~\ref{cor-group-of-units} that the maximal subsemigroups arising from
  the group of units are those given in the theorem.  As described above,
  $\mathcal{AJ}_{n} = \genset{\rho_{n},\ \alpha}$ if and only if $\alpha \in
  J_{n - 2} \cap \mathcal{AJ}_{n}$. By Corollary~\ref{cor-Xi}, the sole
  remaining maximal subsemigroup is $\mathcal{AJ}_{n} \setminus J_{n - 2}$, as
  required.
\end{proof}

%%%%%%%%%%%%%%%%%%%%%%%%%%%%%%%%%%%%%%%%%%%%%%%%%%%%%%%%%%%%%%%%%%%%%%%%%%%%%%%%
\subsection{The Motzkin monoid $\mathcal{M}_{n}$}\label{sec-motzkin}

Finally, in this section, we describe and count the maximal subsemigroups of the
Motzkin monoid $\mathcal{M}_{n}$.  Let $n \in \N$, $n \geq 2$.
By~\cite[Proposition~4.2]{Dolinka2017251}, $\mathcal{M}_{n}$ is generated by
its elements of rank at least $n - 2$, and any generating set for
$\mathcal{M}_{n}$ contains elements of ranks $n$, $n - 1$, and $n - 2$.  By
Lemma~\ref{lem-green-partition}, Green's $\J$-relation on $\mathcal{M}_{n}$ is
determined by rank, and so the maximal subsemigroups of $\mathcal{M}_{n}$ arise
from the $\J$-classes that correspond to these ranks.  To describe the
maximal subsemigroups of $\mathcal{M}_{n}$, we therefore require a description
of its elements that have rank at least $n - 2$.

Clearly the unique element of $\mathcal{M}_{n}$ of rank $n$ is $\id_{n}$.

An arbitrary element of rank $n - 1$ in $\mathcal{M}_{n}$ has trivial kernel and
cokernel, and is uniquely determined by the point $i$ that it lacks from its
domain and the point $j$ that it lacks from its codomain.  By
Lemma~\ref{lem-green-partition}, this information determines the $\L$- and
$\R$-classes of $J_{n - 1} \cap \mathcal{M}_{n}$.  In other words, for any
element in $J_{n - 1} \cap \mathcal{M}_{n}$, there exist points $i, j \in \n$
such that $\{i\}$ and $\{j'\}$ are its unique singleton blocks. An element of
$J_{n - 1} \cap \mathcal{M}_{n}$ is an idempotent if its domain and codomain
are equal, and so every idempotent in $J_{n - 1} \cap \mathcal{M}_{n}$ is a
projection.

Let $\alpha$ be an arbitrary element of rank $n - 2$ in $\mathcal{M}_{n}$. Then
$\alpha$ lacks two points $i, j \in \n$ from its domain, and either
$\ker(\alpha)$ is trivial, or it contains the unique non-trivial kernel class
$\{i, j\}$, in which case $|i - j| = 1$.  Similarly, $\alpha$ lacks two points
from its codomain, and either $\coker(\alpha)$ is trivial, or it contains a
unique non-trivial class with two consecutive points. By
Lemma~\ref{lem-green-partition}, we obtain a description of the Green's classes
of $\mathcal{M}_{n}$ of rank $n - 2$.

Let $\alpha, \beta \in \mathcal{M}_{n}$ be elements of rank $n - 1$.  There
exist numbers $i, j, k, l \in \n$ such that $\{i\}$ and $\{j'\}$ are the
singleton blocks of $\alpha$ and $\{k\}$ and $\{l'\}$ are the singleton blocks
of $\beta$.  If $j = k$, then $\alpha\beta$ has rank $n - 1$, and its singleton
blocks are $\{i\}$ and $\{l'\}$. Otherwise, $\alpha\beta \in \mathcal{M}_{n}$
has rank $n - 2$, and has trivial kernel and cokernel. Conversely, any element
of $\mathcal{M}_{n}$ of rank $n - 2$ with trivial kernel and cokernel can be
written as the product of two elements of $\mathcal{M}_{n}$ of rank $n - 1$.
By~\cite[Lemma~4.11]{Dolinka2017251}, it follows that $\mathcal{M}_{n}$ is
generated by its elements of ranks $n$ and $n - 1$, along with the collection of
projections of rank $n - 2$ that have non-trivial kernel and cokernel.

The main result of this section is the following theorem.

\begin{thm}\label{thm-motzkin}
  Let $n \in \N$, $n \geq 2$, and let $\mathcal{M}_{n}$ be the Motzkin monoid of
  degree $n$. Then the maximal subsemigroups of $\mathcal{M}_{n}$ are:
  \begin{enumerate}[label=\emph{(\alph*)}]
    \item
      $\mathcal{M}_{n} \setminus \{ \id_{n} \}$
      \emph{(type~\ref{item-remove-j})};
    \item
      The union of $\mathcal{M}_{n} \setminus J_{n - 1}$ and $$\bigcup\limits_{i
      \in A} \bigset{\alpha \in\mathcal{M}_{n}}{\rank(\alpha) = n - 1\
      \text{and}\ \{i\}\ \text{is a block of}\ \alpha} \cup \bigcup\limits_{i
      \not\in A} \bigset{\alpha \in\mathcal{M}_{n}}{\rank(\alpha) = n - 1\
      \text{and}\ \{i'\}\ \text{is a block of}\ \alpha},$$ where $A$ is any
      non-empty proper subset of $\n$ \emph{(type~\ref{item-rectangle})};
    \item
      $\mathcal{M}_{n} \setminus \bigset{\alpha \in J_{n - 2}}{\{i, i + 1\} \
      \text{is a block of}\ \alpha}$ for $i \in \{1, \ldots, n - 1\}$
      \emph{(type~\ref{item-remove-r})}; and
    \item
      $\mathcal{M}_{n} \setminus \bigset{\alpha \in J_{n - 2}}{\{i', (i + 1)'\}
      \ \text{is a block of}\ \alpha}$ for $i \in \{1, \ldots, n - 1\}$
      \emph{(type~\ref{item-remove-l})}.
  \end{enumerate}
  In particular, for $n \geq 2$, there are $2^{n} + 2n - 3$ maximal
  subsemigroups of the Motzkin monoid of degree $n$. The Motzkin monoid
  $\mathcal{M}_{1} = \mathcal{P}_{1}$ is a semilattice of order $2$: its maximal
  subsemigroups are each of its singleton subsets.
\end{thm}

\begin{proof}
  The group of units of $\mathcal{M}_{n}$ is the trivial group $\{\id_{n}\}$,
  and so by Corollary~\ref{cor-group-of-units} the sole maximal subsemigroup of
  $\mathcal{M}_{n}$ that arises from its group of units is formed by removing
  it.  Given the above description of the $\J$-class $J_{n - 1} \cap
  \mathcal{M}_{n}$, it follows from Corollary~\ref{cor-delta-projections} that
  the maximal subsemigroups that arise from this $\J$-class are those described
  in the theorem of type~\ref{item-rectangle}, and that there are $2^{n} - 2$ of
  them;  there are no maximal subsemigroups of type~\ref{item-intersect} since
  $\mathcal{M}_{n}$ is $\H$-trivial. It remains to describe the maximal
  subsemigroups that arise from the $\J$-class of rank $n - 2$.

  For $i \in \{1, \ldots, n - 1\}$, we define the subsets
  \begin{align*}
    X_{i} & = \set{\alpha \in \mathcal{M}_{n}}{\rank(\alpha) = n - 2\
    \text{and}\ \{i, i + 1\}\ \text{is a block of}\ \alpha},\ \text{and}\\
    X_{i}^{*} = \set{\alpha^{*}}{\alpha \in X_{i}} & = \set{\alpha \in
    \mathcal{M}_{n}}{\rank(\alpha) = n - 2 \ \text{and}\ \{i', (i + 1)'\}\
    \text{is a block of}\ \alpha}
  \end{align*}
  of the $\J$-class $J_{n - 2} \cap \mathcal{M}_{n}$.  Note that $X_{i}$ is an
  $\R$-class of $\mathcal{M}_{n}$ and $X_{i}^{*}$ is an $\L$-class of
  $\mathcal{M}_{n}$.
  
  Let $A$ be a subset of $J_{n - 2} \cap \mathcal{M}_{n}$ such that
  $(\mathcal{M}_{n} \setminus J_{n - 2}) \cup A$ generates $\mathcal{M}_{n}$.
  Let $i \in \{1,\ldots,n-1\}$ and $\alpha \in X_{i}$ be arbitrary.  Then
  $\alpha$ can be written as a product $\alpha = \beta_{1}\cdots\beta_{k}$ of
  the generators that have rank $n - 1$ or $n - 2$.  If $\rank(\beta_{1}) = n -
  1$, then $\beta_{1}$, and each of its right-multiples, contains a singleton
  block of the form $\{j\}$ for some $j \in \n$. However, $\alpha$ is a
  right-multiple of $\beta_{1}$ and $\alpha$ contains no such block;
  thus $\rank(\beta_{1}) = n - 2 = \rank(\alpha)$.
  Lemmas~\ref{lem-stability} and~\ref{lem-green-partition} imply that
  $\ker(\alpha) = \ker(\beta_{1})$, i.e.\ $A \cap X_{i} \neq \varnothing$.  A
  dual argument shows that $A \cap X_{i}^{*} \neq \varnothing$.

  Conversely, for any subset $A$ of $J_{n - 2} \cap
  \mathcal{M}_{n}$ that intersects $X_{i}$ and $X_{i}^{*}$ non-trivially for all
  $i \in \{1,\ldots,n-1\}$, it is straightforward to see that
  $\genset{\mathcal{M}_{n} \setminus J_{n - 2},\ \cup A}$ contains every
  projection in $J_{n - 2} \cap \mathcal{M}_{n}$, and hence is equal to
  $\mathcal{M}_{n}$. By Lemma~\ref{lem-Xi}, the maximal subsemigroups of
  $\mathcal{M}_{n}$ arising from its $\J$-class of rank $n - 2$ are the sets
  $\mathcal{M}_{n} \setminus X_{i}$ and $\mathcal{M}_{n} \setminus X_{i}^{*}$
  for $i \in \{1, \ldots, n - 1\}$; these maximal subsemigroups have
  types~\ref{item-remove-r} and~\ref{item-remove-l}, respectively.
\end{proof}

%%%%%%%%%%%%%%%%%%%%%%%%%%%%%%%%%%%%%%%%%%%%%%%%%%%%%%%%%%%%%%%%%%%%%%%%%%%%%%%%
%%%%%%%%%%%%%%%%%%%%%%%%%%%%%%%%%%%%%%%%%%%%%%%%%%%%%%%%%%%%%%%%%%%%%%%%%%%%%%%%
\section*{Acknowledgements}

The first author gratefully acknowledges the support of the Glasgow Learning,
Teaching, and Research Fund in partially funding his visit to the third author
in July, 2014.
The second author wishes to acknowledge the support of research
initiation grant $[0076~\vert~2016]$ provided by BITS Pilani, Pilani.
The fourth author wishes to acknowledge the support of his Carnegie Ph.D.
Scholarship from the Carnegie Trust for the Universities of Scotland.

\bibliography{maximals-EKMW}{}
\bibliographystyle{plain}
\end{document}